\documentclass[a4paper,12pt]{amsart}
\usepackage{amsmath,mathtools} 
\usepackage{amssymb,enumerate} 
\usepackage{geometry}
\usepackage{todonotes}
\newcommand{\chf}[1]{\mathbf{1}_{#1}}
\usepackage{cancel}

\newtheorem{theorem}{Theorem}[section]
\newtheorem{corollary}[theorem]{Corollary}

\newtheorem{proposition}[theorem]{Proposition}
\newtheorem{definition}[theorem]{Definition}
\newtheorem{remark}[theorem]{Remark}

\newtheorem{example}[theorem]{Example}

\newtheorem{assumption}[theorem]{Assumption}

\newcommand{\norm}[1]{\left\Vert#1\right\Vert}
\newcommand{\Semi}{{\normalfont Arc} }
\theoremstyle{nonumberplain}
\theoremstyle{empty}

\numberwithin{equation}{section} 
\def\A{\text{\rm vN}(G(\HH_\R))}

\def\C{{\mathbb C}}
\def\R{{\mathbb R}}
\def\N{{\mathbb N}}

\def\I{{s}}

\def\ProdQ{{{q}^{\Cr(\pi|_{{[n]}})}}{{t}^{\Cov(\pi|_{[n]})}}{{\v}^{\Cr(\pi|_{{[\bar n]}})}}{{\w}^{\Cov(\pi|_{[\bar n]})}}}

\def\ri{{\rm i}}

\def\B{{\bf A}}
\def\G{G}

\def\l{r}

\def\S{\mathfrak{S}}
\def\F{\mathcal{F}_{\rm dig}(\HH)}

\def\FQ{\mathcal{F}_{\rm dig}^{\rm  \S}(\HH)}
\def\id{I}
\def\m{\kappa_{\q}}
\def\qMP{{\mu}}
\def\qqMP{{\rm MP}}
\def\P{\mathcal{P}}

\def\PS{\mathcal{PS}}
\def\NC{\mathcal{NC}}

\def\lb{l}
\def\rb{r}
\def\PB{\mathcal{P}^{\O}}
\def\BL{{\mathbf{B}}}

\def\Cr{\text{\normalfont cr}}
\def\InS{\text{\normalfont CS}}

\def\SL{\text{\normalfont SR}}
\def\Pair{\text{\normalfont Pair}}
\def\MPair{{\mathcal{D}_{block}}}
\def\MPairPair{{\mathcal{D}_{block}^{\PS}}}
\def\Sing{\text{\normalfont Sing}}
\renewcommand{\epsilon}{\varepsilon}
\def\Cov{\text{\normalfont nest}}
\def\H{{\bf \mathcal H}_{ n}}
\def\HH{{\bf \mathcal H}}

\def\O{\otimes}%{{{\otimes_{{\mathop{\scalebox{0.4}{\raisebox{-0.2ex}{$\S$}}}}}}}}%{\bm\otimes}}
\def\W{{\S}_{\pm n}}

\def\v{v}
\def\w{w}

\def\xeta{ {  \vec{ \eta}}}
\def\yeta{ {  \vec{ \eta}}}
\def\x{  \vec{ \xi}_{n}}
\def\y  { \vec{ \xi}_{\bar n}}

\def\q{q,t,\v,\w}
\def\qq{q,t,\v,\w}

\def\xx{\xi\O \eta } 

\def\nawiasl{{\big \langle }}
\def\nawiasp{ \rangle_{ \HH }}

\def\state{\varphi }
\def\T{ {\bar {T}}}

\def\TG{{ T\O\T}}

\def\1{1}

\def\nest{\text{\normalfont rnest}}
\def\X{ X}
\def\Cum { R^{\xi,T}(}

\def\OptT{\mb{T}^{{\xi}}}
\def\OptTTyl{\mb{\widehat{T}}^{{\xi}}}
\def\Cumm {{\mathrm K}^\mb{\xi,T}_{\pi}}
\def\CummTensor {\widehat{\mathrm{K}}^\mb{\xi,T}_{\pi}}
\newcommand{\BC}{S_\primes, \bar K_\primes }
\newcommand{\ith}{( s_{p_1},k_{p_2}) }
\newcommand{\Tens}{\OptTTyl} 
\DeclareMathOperator{\Part}{\mathcal{P}}
\DeclareMathOperator{\Rtr}{\mathrm{LH}}
\newcommand{\primes}{E}
\newcommand{\set}[1]{\left\{#1\right\}}
\newcommand{\mf}[1]{\mathbb{#1}}
\newcommand{\mc}[1]{\mathcal{#1}}
\newcommand{\mb}[1]{\mathbf{#1}}

\newcommand{\abs}[1]{\left\vert#1\right\vert}
\newcommand\Matching[4]{%
  \begin{tikzpicture}[scale=0.5]
    \foreach \x in {1,...,#1}{
       \draw[circle,fill] (\x,0)circle[radius=1mm]node[below]{$\x$};
    }
    \foreach \x/\y in {#2} {
       \pgfmathsetmacro{\Radius}{\y/2-\x/2}
       \draw(\x,0) arc[radius=\Radius, start angle=180, end angle=0];
    }
  \end{tikzpicture}%
    \begin{tikzpicture}[scale=0.5]
    \foreach \x in {1,...,#3}{
       \draw[circle,fill] (\x,0)circle[radius=1mm]node[below]{ $\bar{\x}$};
    }
    \foreach \x/\y in {#4} {
       \pgfmathsetmacro{\Radius}{\y/2-\x/2}
       \draw(\x,0) arc[radius=\Radius, start angle=180, end angle=0];
    }
  \end{tikzpicture}%
}

\newcommand\MatchingE[4]{%
  \begin{tikzpicture}[scale=0.5]
    \foreach \x in {1,...,#1}{
       \draw[circle,fill] (\x,0)circle[radius=1mm]node[below]{};
    }
    \foreach \x/\y in {#2} {
       \pgfmathsetmacro{\Radius}{\y/2-\x/2}
       \draw(\x,0) arc[radius=\Radius, start angle=180, end angle=0];
         \node[color=red] at (1,-0.71) {$1$};
            \node[color=red] at (4,-0.71) {$4$};
              \node[color=blue] at (2,-0.71) {$2$};
            \node[color=blue] at (6,-0.71) {$6$};
             \node[color=green] at (3,-0.71) {$3$};
            \node[color=green] at (5,-0.71) {$5$};
    }
  \end{tikzpicture}%
    \begin{tikzpicture}[scale=0.5]
    \foreach \x in {1,...,#3}{
       \draw[circle,fill] (\x,0)circle[radius=1mm]node[below]{ };
    }
    \foreach \x/\y in {#4} {
       \pgfmathsetmacro{\Radius}{\y/2-\x/2}
       \draw(\x,0) arc[radius=\Radius, start angle=180, end angle=0];
      \node[color=red] at (1,-0.721) {$\bar 1$};
            \node[color=red] at (6,-0.721) {$\bar 6$};
             \node[color=blue] at (2,-0.721) {$\bar 2$};
            \node[color=blue] at (5,-0.721) {$\bar 5$};
              \node[color=green] at (3,-0.721) {$\bar 3$};
            \node[color=green] at (4,-0.721) {$\bar 4$};
    }
  \end{tikzpicture}%
}

\newcommand\MatchingEE[4]{%
  \begin{tikzpicture}[scale=0.5]
    \foreach \x in {1,...,#1}{
       \draw[circle,fill] (\x,0)circle[radius=1mm]node[below]{};
    }
    \foreach \x/\y in {#2} {
       \pgfmathsetmacro{\Radius}{\y/2-\x/2}
       \draw(\x,0) arc[radius=\Radius, start angle=180, end angle=0];
         \node[color=red] at (1,-0.71) {$1$};
            \node[color=red] at (6,-0.71) {$6$};
              \node[color=blue] at (2,-0.71) {$2$};
            \node[color=blue] at (5,-0.71) {$5$};
             \node[color=green] at (3,-0.71) {$3$};
            \node[color=green] at (4,-0.71) {$4$};
    }
  \end{tikzpicture}%
    \begin{tikzpicture}[scale=0.5]
    \foreach \x in {1,...,#3}{
       \draw[circle,fill] (\x,0)circle[radius=1mm]node[below]{ };
    }
    \foreach \x/\y in {#4} {
       \pgfmathsetmacro{\Radius}{\y/2-\x/2}
       \draw(\x,0) arc[radius=\Radius, start angle=180, end angle=0];
      \node[color=red] at (1,-0.721) {$\bar 1$};
            \node[color=red] at (6,-0.721) {$\bar 6$};
             \node[color=blue] at (2,-0.721) {$\bar 2$};
            \node[color=blue] at (5,-0.721) {$\bar 5$};
              \node[color=green] at (3,-0.721) {$\bar 3$};
            \node[color=green] at (4,-0.721) {$\bar 4$};
    }
  \end{tikzpicture}%
}

\newcommand\MatchingEEE[4]{%
  \begin{tikzpicture}[scale=0.5]
    \foreach \x in {1,...,#1}{
       \draw[circle,fill] (\x,0)circle[radius=1mm]node[below]{};
    }
    \foreach \x/\y in {#2} {
       \pgfmathsetmacro{\Radius}{\y/2-\x/2}
       \draw(\x,0) arc[radius=\Radius, start angle=180, end angle=0];
         \node[color=red] at (1,-0.71) {$1$};
            \node[color=red] at (3,-0.71) {$3$};
              \node[color=red] at (6,-0.71) {$6$};
            \node[color=blue] at (2,-0.71) {$2$};
             \node[color=blue] at (4,-0.71) {$4$};
            \node[color=blue] at (5,-0.71) {$5$};
    }
  \end{tikzpicture}%
    \begin{tikzpicture}[scale=0.5]
    \foreach \x in {1,...,#3}{
       \draw[circle,fill] (\x,0)circle[radius=1mm]node[below]{ };
    }
    \foreach \x/\y in {#4} {
       \pgfmathsetmacro{\Radius}{\y/2-\x/2}
       \draw(\x,0) arc[radius=\Radius, start angle=180, end angle=0];
      \node[color=red] at (1,-0.721) {$\bar 1$};
            \node[color=red] at (6,-0.721) {$\bar 6$};
             \node[color=red] at (4,-0.721) {$\bar 4$};
            \node[color=blue] at (2,-0.721) {$\bar 2$};
              \node[color=blue] at (3,-0.721) {$\bar 3$};
            \node[color=blue] at (5,-0.721) {$\bar 5$};
    }
  \end{tikzpicture}%
}

\newcommand\MatchingEEEE[4]{%
  \begin{tikzpicture}[scale=0.5]
    \foreach \x in {1,...,#1}{
       \draw[circle,fill] (\x,0)circle[radius=1mm]node[below]{};
    }
    \foreach \x/\y in {#2} {
       \pgfmathsetmacro{\Radius}{\y/2-\x/2}
       \draw(\x,0) arc[radius=\Radius, start angle=180, end angle=0];
         \node[color=red] at (1,-0.71) {$1$};
            \node[color=red] at (3,-0.71) {$3$};
              \node[color=blue] at (6,-0.71) {$6$};
            \node[color=blue] at (2,-0.71) {$2$};
             \node[color=blue] at (4,-0.71) {$4$};
            \node[color=blue] at (5,-0.71) {$5$};
    }
  \end{tikzpicture}%
    \begin{tikzpicture}[scale=0.5]
    \foreach \x in {1,...,#3}{
       \draw[circle,fill] (\x,0)circle[radius=1mm]node[below]{ };
    }
    \foreach \x/\y in {#4} {
       \pgfmathsetmacro{\Radius}{\y/2-\x/2}
       \draw(\x,0) arc[radius=\Radius, start angle=180, end angle=0];
      \node[color=red] at (1,-0.721) {$\bar 1$};
            \node[color=red] at (6,-0.721) {$\bar 6$};
             \node[color=red] at (4,-0.721) {$\bar 4$};
            \node[color=blue] at (2,-0.721) {$\bar 2$};
              \node[color=blue] at (3,-0.721) {$\bar 3$};
            \node[color=red] at (5,-0.721) {$\bar 5$};
    }
  \end{tikzpicture}%
}
\newcommand\MatchingTrace[4]{%
  \begin{tikzpicture}[scale=0.5]
    \foreach \x in {1,...,#1}{
       \draw[circle,fill] (\x,0)circle[radius=1mm]node[below]{};
    }
    \foreach \x/\y in {#2} {
       \pgfmathsetmacro{\Radius}{\y/2-\x/2}
       \draw(\x,0) arc[radius=\Radius, start angle=180, end angle=0];
         \node at (1,-0.71) {$2$};
            \node at (2,-0.71) {$3$};
              \node at (3,-0.71) {$4$};
            \node at (4,-0.71) {$1$};
              \node at (-2,0.4) { $\xrightarrow{\text{trace action }}$};
    }
  \end{tikzpicture}%
    \begin{tikzpicture}[scale=0.5]
    \foreach \x in {1,...,#3}{
       \draw[circle,fill] (\x,0)circle[radius=1mm]node[below]{ };
    }
    \foreach \x/\y in {#4} {
       \pgfmathsetmacro{\Radius}{\y/2-\x/2}
       \draw(\x,0) arc[radius=\Radius, start angle=180, end angle=0];
      \node at (1,-0.721) {$\bar 2$};
            \node at (2,-0.721) {$\bar 3$};
             \node at (3,-0.721) {$\bar 4$};
            \node at (4,-0.721) {$\bar 1$};
    }
  \end{tikzpicture}%
}

\newcommand\MatchingMeanders[4]{%
  \begin{tikzpicture}[scale=0.5]
    \foreach \x in {1,...,#1}{
       \draw[circle,fill] (\x,0)circle[radius=1mm]node[below]{};
    }
    \foreach \x/\y in {#2} {
       \pgfmathsetmacro{\Radius}{\y/2-\x/2}
       \draw(\x,0) arc[radius=\Radius, start angle=180, end angle=0];
;

              \node at (-1,0.4) { $\longleftrightarrow$};
    }
    
         \foreach \x in {1,...,#3}{
       \draw[circle,fill] (\x,0)circle[radius=1mm]node[below]{};
    }
    \foreach \x/\y in {#4} {
       \pgfmathsetmacro{\Radius}{\y/2-\x/2}
       \draw(\x,0) arc[radius=\Radius, start angle=-180, end angle=0];
        ;}
  \end{tikzpicture}
}

\newcommand\MMatching[4]{%
  \begin{tikzpicture}[scale=0.5]
    \foreach \x in {1,...,#1}{
       \draw[circle,fill] (\x,0)circle[radius=1mm]node[below]{};
    }
    \foreach \x/\y in {#2} {
       \pgfmathsetmacro{\Radius}{\y/2-\x/2}
       \draw(\x,0) arc[radius=\Radius, start angle=180, end angle=0];
       \draw (1,-1.5) -- (#1,-1.5);
        \node at (#1/2+0.5,-2.2) {$[n]$};
        \node at (1,-1) {$1$};
        \node at (7.2,-1) {$s_{p_1}$};
          \node at (9.1,-1) {$s_{r}$};
    }
  \end{tikzpicture}
  \hspace{1cm}
    \begin{tikzpicture}[scale=0.5]
    \foreach \x in {1,...,#3}{
       \draw[circle,fill] (\x,0)circle[radius=1mm]node[below]{ };
    }
    \foreach \x/\y in {#4} {
       \pgfmathsetmacro{\Radius}{\y/2-\x/2}
       \draw(\x,0) arc[radius=\Radius, start angle=180, end angle=0];
            \draw (1,-1.5) -- (#3,-1.5);
             \node at (#1/2+0.5,-2.2) {$[\bar n]$};
                \node at (1,-1) {$\bar 1$};
        \node at (6.2,-1) {$k_{p_2}$};
         \node at (10.1,-1) {$k_{r}$};
    }
  \end{tikzpicture}%
}

\newcommand\MMMatching[8]{%
  \begin{tikzpicture}[scale=0.5]
    \foreach \x in {1,...,#1}{
       \draw[circle,fill] (\x,0)circle[radius=1mm]node[below]{$\x$};
    }
    \foreach \x/\y in {#2} {
       \pgfmathsetmacro{\Radius}{\y/2-\x/2}
       \draw(\x,0) arc[radius=\Radius, start angle=180, end angle=0];
        \node at (#1/2+0.5,2.2) {block};
    }
  \end{tikzpicture}%
  \hspace{1cm}
    \begin{tikzpicture}[scale=0.5]
    \foreach \x in {1,...,#3}{
       \draw[circle,fill] (\x,0)circle[radius=1mm]node[below]{ $\x$};
    }
    \foreach \x/\y in {#4} {
       \pgfmathsetmacro{\Radius}{\y/2-\x/2}
       \draw(\x,0) arc[radius=\Radius, start angle=180, end angle=0];
    }
  \end{tikzpicture}%
      \begin{tikzpicture}[scale=0.5]
    \foreach \x in {2,...,3}{
       \draw[circle,fill] (\x,0)circle[radius=1mm]node[below]{ $\x$};
    }
    \foreach \x/\y in {#6} {
       \pgfmathsetmacro{\Radius}{\y/2-\x/2}
       \draw(\x,0) arc[radius=\Radius, start angle=180, end angle=0];
             \node at (#1/2+0.5,2.2) {arcs};
    }
  \end{tikzpicture}%
       \begin{tikzpicture}[scale=0.5]
    \foreach \x in {3,...,4}{
       \draw[circle,fill] (\x,0)circle[radius=1mm]node[below]{ $\x$};
    }
    \foreach \x/\y in {#8} {
       \pgfmathsetmacro{\Radius}{\y/2-\x/2}
       \draw(\x,0) arc[radius=\Radius, start angle=180, end angle=0];
    }
  \end{tikzpicture}%
  \hspace{1cm}
    \begin{tikzpicture}[scale=0.5]
        \node at (0.5,4.7) {opener -- $1$};
          \node at (0.3,3.5) {closer -- $4$};
            \node at (1,2.2) {middle -- $2,$ $3$};
  \end{tikzpicture}%
  
}

\newcommand{\ip}[2]{\left \langle #1, #2 \right \rangle}

\DeclareMathOperator{\ID}{{\mathcal{ID}}_{\q}}
\DeclareMathOperator{\IF}{\mathcal{IF}}
\DeclareMathOperator{\Falg}{\mathcal{F}_{\mathrm{alg}}}
\newcommand{\rc}{\mathrm{rc} }

\newcommand{\ls}[1]{\mathrm{span} \left( #1 \right)}
\title[Fock space  associated with quadrabasic Hermite orthogonal polynomials]{Fock space associated with quadrabasic Hermite orthogonal polynomials}
\author[Wiktor Ejsmont]{Wiktor Ejsmont}
\address { 
Department of Telecommunications and Teleinformatics, Wroclaw University
of Science and Technology\\
Wybrze\.ze Wyspia\'nskiego 27, 50-370 Wroc\l aw, Poland}
\email{wiktor.ejsmont@gmail.com}
\thanks{Supported  by the Narodowe Centrum Nauki grant N$^{\textrm{o}}$ 2018/29/B/HS4/01420}
\begin{document}
\maketitle
\begin{abstract}
This paper introduces a new idea for constructing operators
associated with a certain class of probability measures. 
  Special cases include several know classical and noncommutative probability. 
 The main example is derived from 
 Feller \cite[Page 503, Example 10]{Feller1971}, i.e. the hyperbolic secant distribution. 
In probability theory and statistics, the hyperbolic secant distribution is a continuous probability distribution whose probability density function and characteristic function are proportional to the hyperbolic secant function.  
\end{abstract}
%\tableofcontents{}
\section{Introduction}
The study of $q$-Gaussian  distributions \cite{BozejkoSpeicher1991} has been an active field of research during the last decade.  
 A noncommutative analog of a Brownian motion (or Gaussian process, more generally) is the
family of operators $(a_q^*(x)+a_q(x))_{x\in H}$. When equipped with the vacuum expectation state $\langle\Omega,\cdot\, \Omega\rangle_q $, the
$q$-Gaussian algebra yields a rich non-commutative probability space. For $q=1$ (corresponding to the Bose statistics) the operator $ a_1^*(x)+a_1(x)$ is %a natural deformation of 
the  standard Gaussian  random variable, i.e.\  its spectral measure relative to the vacuum state satisfies 
$$\langle (a_1^*(x)+a_1(x))^n \Omega,\Omega\rangle_1=\frac{1}{\sqrt{2\pi}}\int_{\mathbb{R}}t^n e^{-\frac{t^2}{2}}\,dt
$$
when $\|x\|=1$. 
 Moreover, $\{a_1^*(x)+a_1(x)\}_{x\in H}$ are commutative in the classical sense. 
The  case $q=-1$ corresponds to the Fermi statistics. It should be stressed that, for $q\neq \pm 1$, the $q$-modification of the
(anti)symmetrization operator is a strictly positive operator. Therefore, unlike the
classical Bose and Fermi cases, there are no commutation relations between the creation
operators. For $q = 0$, the $q$-Fock
space recovers the full Fock space of Voiculescu's free probability \cite{VoiculescuDykemaNica92}. For $q = 0$, the $q$-Gaussian random variables are distributed according to the semi-circle law 
$$
\langle (a_0^*(x)+a_0(x))^n \Omega,\Omega\rangle_0=\frac{1}{2\pi}\int_{-2}^2 t^n\sqrt{4-t^2}\, dt 
$$
when $\|x\|=1$. 

 The study of the noncommutative Brownian motion $\{a_q^\ast(x)+a_q(x)\}_{x\in H}$ was initiated in \cite{BozejkoSpeicher1991,BozejkoSpeicher1994,BozejkoSpeicher1997}. For further generalizations of
a $q$-Brownian motion, see \cite{GutaMaassen2002,Blitvic2012,BozejkoEjsmontHasebe2015,BozejkoEjsmontHasebe2017}. 
 In particular, this setting
gives rise to $q$-deformed versions of the  stochastic calculus  \cite{Ans01,Bryc2001, Ans04,Donati-Martin,Deya2018}.

One of the most
beautiful and important results in this area was initiated by  Blitvi\'c \cite{Blitvic2012} where a second-parameter refinement of the $q$-Fock space,
formulated as a $(q,t)$-Fock space $\mathcal{F}_{q,t} (H )$ was introduced. It is constructed via a direct generalization of
Bo\.zejko and Speicher's framework \cite{BozejkoSpeicher1991}, yielding the $q$-Fock space when $t = 1$.   These are
the defining relations of the Chakrabarti-Jagannathan deformed quantum oscillator
algebra; see  \cite{Blitvic2012} and references therein for more details. The moments of
the deformed Gaussian process $\{a_{q,t}(x) +a_{q,t}^\ast(x)\}_{x\in H}$ are encoded by the joint statistics of crossings
and nestings in pair partitions. In particular, it is shown that the distribution of a single Gaussian operator orthogonalizes the $(q, t)$-Hermite polynomials.

The goal of this paper is to introduce \emph{quadrabasic Fock space}. Our approach is to replace the  permutation group  by the tensor product of two permutation groups.
The orthogonal polynomials of a Gaussian type  arising in the present framework satisfying the recurrence relation
\begin{align} \label{recursion}
&x Q_n^{(\q)}(x) = Q_{n+1}^{(\q)}(x) +[n]_{q,t}[n]_{\v,\w}Q_{n-1}^{(\q)}(x), \qquad 
%\\ &\I Q_n^{\q}(\I) = Q_{n+1}^{\q}(\I) +[n]_{q,t}[n]_{\v,\w}Q_{n-1}^{\q}(\I)+[n]_{q,t}[n]_{\v,\w}Q_{n-1}^{\q}(\I), \qquad 
n=0,1,2,\dots
\end{align}
where $Q_{-1}^{(\q)}(x)=0,Q_0^{(\q)}(x)=1$, $\abs{q}\leq t\leq 1$, $ 
\abs{\v}\leq \w\leq 1$ and 
 $[n]_{q,t}$ is the $q,t$-number 
$$
[n]_{q,t}:= t^{n-1}+qt^{n-2}+\cdots+q^{n-2}t+q^{n-1},\text{ and } [n]_q:=[n]_{q,1} \qquad n \geq1. 
$$
%and let $[n]_{q,t}!$ be the $q,t$-factorial 
%$$
%[n]_{q,t} !:= [1]_{q,t} \cdots [n]_{q,t},\qquad n \geq1. 
%$$   
 We  call  these polynomials  \emph{quadrabasic Hermite orthogonal polynomials}, because they depend on four parameters 
 and it is a natural extension of   $q$ or $(q,t)$-Hermite orthogonal
polynomials. 
Considering families built around more general hypergeometric functions, the quadrabasic Hermite sequence belongs to the octabasic Laguerre family (or its symmetric version) introduced by Simion and Stanton \cite{Simion1996} and recently extended by Blitvi\'c and Steingr\'{\i}msson \cite{BlitvicStein} or Sokal and Zeng \cite{Sokal} (see also the earlier work \cite{Randrianarivony}). 
This formula recovers  the hyperbolic secant case when $q=t=\v=\w=1$. 
 The hyperbolic secant function is
equivalent to the reciprocal hyperbolic cosine, and thus this distribution is also called the \emph{hyperbolic cosine distribution}. 
This measure 
orthogonalizes a special class of Meixner-Pollaczek polynomials 
which satisfy the recurrence relation
\begin{align}\label{recursionhiperbolic}
x{Q}_n(x) = {Q}_{n+1}(x) +n^2{Q}_{n-1}(x), \qquad n=0,1,2,\dots
\end{align}
with initial conditions ${Q}_{-1}(x)=0$ and ${Q}_0(x)=1$. 
To the best of our knowledge, the relation \eqref{recursionhiperbolic}  appears for the first time in the literature in \cite[eq. (5.4)]{Meixner1934}. In literature it is also possible to find a relatively large number of
works addressing the relation \eqref{recursionhiperbolic} for example   
\cite[eq. (4.7), with rescaling $Q_n(x)=n! A_n ((x-1)/2)$ ]{Carlitz1959};
 it was shown that these polynomials satisfy a symbolic orthogonality relation with respect to the Euler numbers. 
 Moreover, we investigate this construction in the context of a Poisson-type operator and apply this idea  to introduce a new class of noncommutative  L\'evy processes.

The plan of the paper is following: first we present definitions of the $(q,t)$-Fock space and corresponding creation, annihilation and  gauge operators. 
By using this  we present the definition of  quadrabasic Fock space  and the
creation and annihilation operators acting on it.
In Section 
3 we present a new type of partitions and the relevant statistics.
Next, in Sections 3 and 4    we introduce the generalized Gaussian process and gauge operators and
some of their natural properties, including norm estimates and the self-adjointness.
In these two sections we mainly study  an explicit Wick formula for the
mixed moments. Finally, in Section 5 we  apply this technique for
constructing new L\'evy processes, which  allows us to define a $(\q)$-convolution for a large
class of probability measures. 

\section{Preliminaries and Quadrabasic  Fock space}

\subsection{The Blitvi\'c $(q,t)$-Fock space \cite{BozejkoYosida2006, Bozejko2007, Blitvic2012}} \label{sec:qausian}
Let $H_{\R}$  be a separable real Hilbert space and let $H$ be its complexification with inner product $\langle\cdot,\cdot\rangle$ linear on the right component and anti-linear on the left. 
When considering elements in $H_{\R}$, it holds true that $\langle x,y\rangle=\langle y,x\rangle$. Let $\Falg(H)$ be its algebraic full Fock space, $\Falg(H) = \bigoplus_{n=0}^\infty H^{\otimes n}$, where $H^{\otimes 0} = \mf{C} \Omega$ and $\Omega$ is the vacuum vector. For each $n \geq 0$, define the operator $P_n$ on $H^{\otimes n}$ by 
\begin{align*}
& P_{q,t}^{(0)} (\Omega) = \Omega, \\
& P_{q,t}^{(n)}(\eta_1 \otimes \eta_2 \otimes \ldots \otimes \eta_n) = \sum_{\sigma \in \S_n} q^{l_1(\sigma)} t^{l_2(\sigma)}\eta_{\sigma(1)} \otimes \eta_{\sigma(2)} \otimes \ldots \otimes \eta_{\sigma(n)}, \text{ for }   q,t\in [-1,1]\text{, }\abs{q}\leq t,
\end{align*}
where $\S_n$ is the group of permutations of $1,2,\dots,n$ elements, and $l_1(\sigma)$ (inversions) is the minimal number of $\sigma_i$
, $1 \leq i \leq n - 1$, appearing in $\sigma$ and $l_2(\sigma)$  is  ${{n}\choose{2}}-l_1(\sigma)$ (co-inversions). Remember that  the
symmetric group   is generated by $ \sigma_i =(i,i+1)$, $i=1,\dots,n-1$, which satisfy the generalized braid relations $\sigma_i^2=e$,   $1\leq i < n-1$ and $(\sigma_i \sigma_j)^2=e$ if $|i-j|\geq2, 0\leq i,j\leq n-1$. 
 For $q=t=0$ each $P_{q,t}^{(n)} = \id$. For $q=t=1$, $P_{q,t}^{(n)} = n!\ \times$ the projection onto the subspace of symmetric tensors. For $q = -1$, $t=1$, $P_{q,t}^{(n)} = n!\ \times$ the projection onto the subspace of anti-symmetric tensors.

Define the $(q,t)$-deformed inner product on $\Falg(H)$ by the rule that for $\zeta \in H^{\otimes k}$, $\eta \in  H^{\otimes n}$,
\[
\ip{\zeta}{\eta}_{q,t} := \delta_{nk}\langle {\zeta},{P_{q,t}^{(n)} \eta}\rangle,
\]
where the inner product on the right-hand-side is the usual inner product induced on $H^{\otimes n}$. All inner products are linear in the second variable. It is a result of \cite{Blitvic2012} that the inner product $\ip{\cdot}{\cdot}_{q,t}$ is positive definite for $q,t \in (-1, 1)$ and  $|q| < t$, while for $\abs{q}=t$  it is positive semi-definite. Let $\mc{F}_{q,t}(H)$ be the completion of $\Falg(H)$ with respect to the norm corresponding to $\ip{\cdot}{\cdot}_{q,t}$. For $\abs{q}=t$ one first needs to quotient out by the vectors of norm $0$ and then complete; the result is the anti-symmetric, respectively, symmetric Fock space, with the inner product multiplied by $n!$ on the $n$-particle space. 
For $\xi$ in $H_\R$, define the (left) creation and annihilation operators on $\Falg(H)$ by, respectively,
\begin{align*}
& a_{q,t}^\ast(\xi) \Omega = \xi, \\
& a_{q,t}^\ast(\xi) \eta_1 \otimes \eta_2 \otimes \ldots \otimes \eta_n = \xi \otimes \eta_1 \otimes \eta_2 \otimes \ldots \otimes \eta_n,
%\end{align*}
\intertext{and}
%\begin{align*}
& a_{q,t}(\xi) \Omega = 0, \\
& a_{q,t}(\xi)  \eta = \ip{\xi}{\eta} \Omega, \\
& a_{q,t}(\xi)  \eta_1 \otimes \eta_2 \otimes \ldots \otimes \eta_n = \sum_{i=1}^n q^{i-1} t^{n-i}\ip{\xi}{\eta_i} \eta_1 \otimes \ldots \otimes \hat{\eta}_i \otimes \ldots \otimes \eta_n,
\end{align*}
where as usually $\hat{\eta}_i$ means omitting the $i$-th term. 
They satisfy the commutation relations 
\begin{align}\label{commutationrelations}
 a_{q,t}(\xi) a_{q,t}^\ast(\eta) - q a_{q,t}^\ast(\eta) a_{q,t}(\xi) = \ip{\xi}{\eta} t^{N},
 \end{align}  where $t^{N}$ is the operator on $\mc{F}_{q,t}(H)$ defined by the linear extension of $t^{N}\Omega =0$ and $t^N \eta_1  \otimes \ldots \otimes \eta_n=t^n \eta_1 \otimes \ldots \otimes \eta_n $.

In the following range of parameters  \cite[Lemma 5]{Blitvic2012} the operators $a_{q,t}$ and $a_{q,t}^\ast$ extend to bounded linear operators (adjoints of each other) on $\mc{F}_{q,t}(H)$, with
the norm 
\begin{align}\label{normaoeratora}
\|a^\ast_{q,t}({\xi})\|_{q,t} = \left\{ \begin{array}{ll}
\|\xi\| &  -t\leq q \leq 0 <t\leq 1
\\
\frac{1}{\sqrt{1-q}}\|\xi\|  &  0<q<t=1\\
\sqrt{(n_\ast(t)+1)t^{n_\ast(t)}}\|\xi\|&  0< q = t< 1
\\\sqrt{\frac{t^{\hat{n}(q,t)+1}-q^{\hat{n}(q,t)+1}}{t-q}}\|\xi\|& 0< q < t< 1
\end{array} \right.
\end{align}
where $\xi\neq 0$, 
 $n_\ast(t):=\left \lfloor{t/(1-t)}\right \rfloor $ and $\hat{n}(q,t):=\left \lfloor{\frac{\log(1-t)-\log(1-q)}{\log{t}-\log{q}}}\right \rfloor$.

  For $q=\pm 1$ and $t=1$ we first need to compress the operators by the projection onto the symmetric/anti-symmetric Fock space, respectively, and the resulting operators differ from the usual ones by $\sqrt{n}$, but satisfy the usual commutation relations (thanks to a different inner product). For $q=t=1$ the resulting operators are unbounded, but still adjoints of each other. 
\subsection{$(q,t)$-gauge operators} \label{seqqtgauge}
In this subsection we recall a differential second quantization operator  which is partially investigated in  \cite{Ejsmont1}. 
In order to simplify the computation we  recall the properties  the
symmetric group. 
 There is a natural embedding $\S(n-1)=\langle \sigma_1, \dots, \sigma_{n-2}\rangle \subset \S(n)=\langle \sigma_1, \dots, \sigma_{n-1}\rangle$, which allows us to  decompose the operator $$P^{(n)}_q=( I\otimes P^{(n-1)}_q)R^{(n)}_{q}=
 {R^{(n)}_{q}}^*(\id \otimes P^{(n-1)}_{q})\text{ ~on $H^{\otimes n}$},$$ where $R^{(n)}_{q} = 1+\sum_{k=1}^{n-1}q^{k}\sigma_{1}\cdots \sigma_{k}$  (see \cite{BozejkoSpeicher1994}). From this we have 
\begin{align}\label{decomposition}
\begin{split}
P^{(n)}_{q,t}&=t^{{n\choose 2}}P^{(n)}_{q/t}=t^{{n\choose 2}}( I\otimes P^{(n-1)}_{q/t})R^{(n)}_{q/t}=( I\otimes t^{{n-1 \choose 2}}P^{(n-1)}_{q/t})t^{n-1}R^{(n)}_{q/t} \\
&=( I\otimes P^{(n-1)}_{q,t})R^{(n)}_{q,t} ={R^{(n)}_{q,t}}^*(\id \otimes P^{(n-1)}_{q,t})
\end{split}
\end{align}
where  $R^{(n)}_{q,t} = t^{n-1}+\sum_{k=1}^{n-1}q^{k}t^{n-k-1}\sigma_{1}\cdots \sigma_{k} \text{ and } n\geq 1$.
Let us further observe that  $\|R^{(n)}_{q,t}\|_{0,0}\leq [n]_{\abs{q},t} $ and so 
%$
%P_{q,t}^{(n)} \leq [n]_{\abs{q},t} (P_{q,t}^{(n-1)}\otimes I) ,
%$
%with respect to the $(0,0)$-inner product, so 
\begin{align} \label{eq:ograniczenieopertora}
\|P_{ q,t}^{(n)}\|_{0,0}\leq  [n]_{\abs{q},t} \|P_{q,t}^{(n-1)}\otimes I\|_{0,0}\leq \prod_{i=1}^n [i]_{\abs{q},t} \leq  n!.
\end{align}
First, we introduce an operator which acts on $(q,t)$-Fock space as
\begin{align*}
& p_0(T) \Omega = 0, \\
& p_0(T) (\xi_1 \otimes \ldots \otimes \xi_n) = T(\xi_1) \otimes \ldots \otimes \xi_n,
\end{align*}
where $T$ is an operator on Hilbert space $H$ with dense domain ${D}$. % and analogously we define  $p_0(T_2)$ on   $\bar{H}$ (with dense domain $\mc{D}^{(2)}$).   
The adjoint of this operator satisfies $\langle p_0(T)f | \zeta \rangle_{0,0}=\langle f | p_0(T^*)\zeta \rangle_{0,0}$, and allows us to define a gauge operator (preservation or differential second quantization). Let  $p_{T}^{(q,t)}:=p_0(T) R^{(n)}_{q,t}$. 
 Let us observe that directly from the generalized braid relations for $k\in [n-1]$, we have
$ 
 \sigma_{1}\cdots \sigma_{k}=\bigl(\begin{smallmatrix}
  1 & \cdots & k & \cdots & n \\
  k & \cdots & k+1 & \cdots  & n
 \end{smallmatrix}\bigr), 
$ 
which allows us to rewrite action of $p_{T}$ on $H^{\otimes n}$ as 
$$p_{T}^{(q,t)}(\xi_1 \otimes \ldots \otimes \xi_n) = \sum_{i=1}^n q^{i-1}t^{n-i}T(\xi_i) \otimes \xi_1\otimes  \ldots \otimes  \hat{\xi}_i\otimes \dots \otimes \xi_n.$$

\subsubsection{$p_{T}^{(q,t)}$ is symmetric operator on $\mc{F}_{q,t}(D)$ and  bounded for $q<1$}
First, we explain the
symmetry properties of  $p_{T}^{(q,t)}$.
Observe that $p_0(T^*) ( \id \otimes P^{(n-1)}_{q,t})= ( \id \otimes P^{(n-1)}_{q,t})p_0(T^*)$; indeed for $\xi_1 \otimes \xi_2 \otimes \ldots \otimes \xi_n \in {D}^{\otimes n}$, we have
\begin{align*}&p_0(T^*) (\id \otimes P^{(n-1)}_{q,t})(\xi_1 \otimes \ldots \otimes \xi_n)=
T^{\ast} (\xi_1) \otimes P^{(n-1)}_{q,t}(\xi_2 \otimes \ldots \otimes \xi_n)\\&= (\id \otimes P^{(n-1)}_{q,t})(T^*(\xi_1) \otimes \ldots \otimes \xi_n)=(\id \otimes P^{(n-1)}_{q,t})p_0(T^*) (\xi_1 \otimes \ldots \otimes \xi_n).
\end{align*}
Let us fix $n$, and $f,g \in {D}^{\otimes n}$, then
\begin{align*}
\langle {p^{(q,t)}_T f},{g}\rangle_{q,t} &=\langle p_T^{(q,t)} f, P^{(n)}_{q,t} g \rangle_{0,0} =\langle p_0(T) R^{(n)}_{q,t} f,  (\id \otimes P^{(n-1)}_{q,t}) R^{(n)}_{q,t}g \rangle_{0,0} \\
&= \langle R^{(n)}_{q,t} f, p_0(T^*) (\id \otimes P^{(n-1)}_{q,t}) R^{(n)}_{q,t}g \rangle_{0,0}
 = \langle R^{(n)}_{q,t} f, (\id \otimes  P^{(n-1)}_{q,t})p_0(T^*) R^{(n)}_{q,t}g \rangle_{0,0}
\intertext{by Equation \eqref{decomposition}, we have}
&=\langle  f,  {R^{(n)}_{q,t}}^*(\id \otimes P^{(n-1)}_{q,t})p_0(T^*) R^{(n)}_{q,t}g \rangle_{0,0}
= \langle{f},{p_{T^*}^{(q,t)}g}\rangle_{ q,t}.
\end{align*}

The following proposition is inspired by \cite[Proposition 1.9]{Ejsmont1}.
The proof is almost identical to that of \cite[Proposition 1.9]{Ejsmont1}, and it can be omitted (the only difference is estimate of norm from the end).
\begin{proposition} If  $T$ is a bounded operator on $H$ and $\abs q<1$, then $p_T^{(q,t)}$ is bounded in $\mathcal{F}_{q,t}(H)$.
\label{seqqtgauge}
\end{proposition}

\subsection{
Diagonal full Fock space }\label{sec2.1}
Let $[n]$ be the set $\{1,\dots,n\}$. 
We also consider additional numbers $\bar{1},\bar{2},\dots,\bar{n}$ and define  $[\bar{n}]$ as the set $\{\bar{1},\dots,\bar{n}\}$. We associate these indices with natural ordering $\bar{1}<\dots <\bar{n}$. 
% Our construction in article \cite{BozejkoEjsmontHasebe2015}  is motivated by the idea of defining a new deformed Fock space by using the Coxeter group of type B. 
% The Coxeter group of type B is the set of  all permutations $\pi$ of $-n,\dots,-1,1,\cdots, n$ such that  $\pi(- k)=-{\pi(k)}, k=1,\dots,n$. 
% Our approach
%is to replace the Coxeter groups of type B by the Weyl group  $\W=\S_n\times \S_{\bar n} $ which permutes separately indices of positive and negative part, as in the example:
%$
%{2}\hspace{0.2cm} {1} \hspace{0.2cm}  {4} \hspace{0.2cm}  {5} \hspace{0.2cm}  {3} \hspace{0.2cm}\times \hspace{0.2cm}  {\bar 5}\hspace{0.2cm}  {\bar 2} \hspace{0.2cm} {\bar 1} \hspace{0.2cm}  {\bar 3} \hspace{0.2cm}  {\bar 4}  
%$.
Our strategy
is to define the action  on  the Weyl group  $\W=\S_n\times \S_{\bar n} $ which permutes separately indices of positive and negative part, as in the example:
$
{2}\hspace{0.2cm} {1} \hspace{0.2cm}  {4} \hspace{0.2cm}  {5} \hspace{0.2cm}  {3} \hspace{0.2cm}\times \hspace{0.2cm}  {\bar 5}\hspace{0.2cm}  {\bar 2} \hspace{0.2cm} {\bar 1} \hspace{0.2cm}  {\bar 3} \hspace{0.2cm}  {\bar 4}  
$.
 
  Let $H$ and $\bar{H}$  be two separable  Hilbert spaces  with inner product $\langle\cdot,\cdot\rangle_{H}$ and $\langle\cdot,\cdot\rangle_{\bar{H}}$, respectively.
 We further assume that these Hilbert spaces  have real Hilbert subspaces, so that $H$ ($\bar{H}$) is the complexification of $H_\R$ ($\bar{H}_\R$). 
 
% The symbol $\O$ denotes the subspace of
%the tensor product consisting of vectors which are invariant under the action of
%$U_n(\sigma) \O  U_n(\gamma)$ for all $\sigma\in \S_n$,  $\gamma\in \S_{\bar{n}} $, where $$U_n(\sigma) x_1\otimes\dots \otimes  x_n=x_{\sigma(1)}\otimes\dots \otimes  x_{\sigma(n)} $$ is a unitary representation  of the symmetric group.

 Then the Hilbert space $\HH:=H\O \bar{H}$ is the complexification of its real subspace $\HH_\R:=H_\R\O \bar{H}_\R$, with inner product $\nawiasl\cdot,\cdot  \nawiasp = \langle\cdot,\cdot\rangle_{H}\langle\cdot,\cdot\rangle_{\bar{H}}$ (and so $\HH$ has a natural conjugation   defined on it).
We define $\H:=H^{\otimes n} \O {\bar{H}^{\otimes n}}$ and an action of $\W$ on it by $U_n(\sigma) \O  U_n(\gamma)$, where $$U_n(\sigma) x_1\otimes\dots \otimes  x_n=x_{\sigma(1)}\otimes\dots \otimes  x_{\sigma(n)} $$ is a unitary representation  of the symmetric group.
Moreover, %in accordance with the space $H^{\otimes n} \O {\bar{H}^{\otimes n}}$
 we
\begin{itemize}
%\item shall use
 %two different notations of tensor product  $\otimes$ and  $\O$;
 \item use the superscript $_{\bar i}$, in order to label vectors in $\bar H$ as  $\xi_{\bar i}\in \bar H$. This in particular  means that
there is no relation between vectors $\xi_{1}\in  H$ and $\xi_{\bar 1}\in \bar H$;
 \item denote the vectors  ${\xi_{1}\otimes \dots \otimes \xi_{n}}\in H^{\otimes n}$ and ${\xi_{\bar 1}\otimes \dots \otimes \xi_{\bar n}} \in \bar{H}^{\otimes n},
$ by $\x$ and $\y$, respectively.
%\item we note ta
\end{itemize} 
We introduce the  algebraic \emph{diagonal full Fock} space over $\HH$
\begin{equation}
\F:=\bigoplus_{n=0}^\infty \H= \bigoplus_{n=0}^\infty H^{\otimes n}\otimes \bar H^{\otimes n}.
\end{equation} 
with convention that $\HH^{\otimes 0}=\C \Omega \O\bar \Omega$ is a one-dimensional normed space along a unit vector $\Omega \O \bar \Omega $. Note that elements of $\F$ are finite linear combinations of the elements from $\HH^{\otimes n}, n\in \N\cup\{0\}$ and we do not take the completion. 
We define the following  inner product $\x\O\y\in \H,\text{ }\xeta_m\O\yeta_{\bar m} \in \HH^{\otimes m}$
\begin{align*}
\langle  \x\O\y,\xeta_m\O\yeta_{\bar m} \hspace{1pt} \rangle_{0,0,0,0}:=\delta_{n,m}\prod_{i=1}^n\langle \xi_i ,\eta_i\rangle_{H}\langle\xi_{\bar i}  , \eta_{\bar i}\rangle_{\bar{H}}.
 \end{align*}
 \begin{remark}
 %\begin{enumerate} 
 We may think that the elements of $\F$ arise from the diagonal elements of full Fock space $(\bigoplus_{n=0}^\infty H^{\otimes n})\otimes (\bigoplus_{n=0}^\infty \bar H^{\otimes n})$, that is, it is the sum of the  boxed entries in the table
 \begin{center}
  \begin{tabular}{  c  c  c  c r }
    %\hline
   $\boxed{H^{\otimes 0}\otimes \bar H^{\otimes0}}$ & $H^{\otimes 0}\otimes \bar H^{\otimes 1}$ & $H^{\otimes 0}\otimes \bar H^{\otimes 2}$ & $H^{\otimes 0}\otimes \bar H^{\otimes 3}$ & $\dots$ \\
   %\hline
    $H^{\otimes 1}\otimes \bar H^{\otimes 0}$  & $\boxed{H^{\otimes 1}\otimes \bar H^{\otimes 1}}$  & $H^{\otimes 1}\otimes \bar H^{\otimes 2}$  & $H^{\otimes 1}\otimes \bar H^{\otimes 3}$ & $\dots$ \\
    %\hline
    $H^{\otimes 2}\otimes \bar H^{\otimes 0}$  & $H^{\otimes 2}\otimes \bar H^{\otimes 1}$  & $\boxed{H^{\otimes 2}\otimes \bar H^{\otimes 2}}$  & $H^{\otimes 2}\otimes \bar H^{\otimes 3}$ & $\dots$ \\
    % \hline
    $H^{\otimes 3}\otimes \bar H^{\otimes 0}$  & $H^{\otimes 3}\otimes \bar H^{\otimes 1}$  & $H^{\otimes 3}\otimes \bar H^{\otimes 2}$  & $\boxed{H^{\otimes 3}\otimes \bar H^{\otimes 3}}$& $\dots$  \\
    % \hline 
     $\vdots$  & $\vdots$  & $ \vdots $ &  $\vdots$ & $\ddots$  
  %  \hline
  \end{tabular}
\end{center}
It is why we called them the diagonal Fock space.
%\item We use a special notation for $\O$ in order to distinguish on which level of tensor product we work. This is helpful in further part of article. 
%Space $\F$ can also be defined as  $\bigoplus_{n=0}^\infty H^{\otimes n}\otimes \bar H^{\otimes n}.$
\end{remark}

\subsection{Quadrabasic Fock space creation and annihilation operators}
Now we deform the inner product on $\F$. 
For $ q,t,\v,\w \in[-1,1]$, $|q|\leq t$ and $|\v|\leq \w$   we define the $(\q)$-symmetrization operator on $\H$  
 \begin{align*}
 %\label{def:symetryzator} 
&P_{\qq}^{(n)}:=P_{q,t}^{(n)}\O P_{\v,\w }^{(n)}
% \sum_{\sigma_1\times \dots \times \sigma_m \in \W}  q_1^{l(\sigma_1)}\dots q_m^{l(\sigma_m)}\, \sigma_1\times \dots \times \sigma_m ,\qquad n \geq1, 
\\&P_{\qq}^{(0)}:= \id_{\HH^{\otimes 0}}. 
%\intertext{where}
%&P_{q_i}^{(i,n)}:=P_{\underbrace{(0,\dots ,0,q_i,0\dots,0)}_{q_i \text{ on i'th position}}}^{(n)}\text{ for } i\in[m] \text{ and } |q_1|,\dots,|q_m|\leq 1.
\end{align*}
%Note that with convention $0^0=1$, we have $P_{0,0,0,0}^{(n)}=\id_{\H}$. 
Moreover, let  
$
P_{\qq}:=\bigoplus_{n=0}^\infty P_{\qq}^{(n)}.$
We equip $\F$ with the inner product  
by using the   deformed  operator: 
\begin{align*}
\langle  \x\O\y ,\xeta_m\O\yeta_{\bar m}\rangle_{\q}:&= \langle  \x\O\y,P_{\qq}^{(n)} \xeta_{ m}\O\yeta_{\bar m}\rangle_{0,0,0,0}\\&=\delta_{n,m}\langle  \x , \xeta_m \rangle_{q,t}\langle  \y , \yeta_{\bar m} \rangle_{\v,\w}.
\end{align*}
Remember that  a strictly positive operator means that it is positive and $\text{Ker}(P_{\qq}^{(n)})=\{0\}$, and directly from above definition and   information from Section \ref{sec:qausian}  we can state the following
\begin{proposition} The operator $P_{\qq}$
\begin{enumerate} 
\item [(a)] is positive for  $|q|\leq t $ and $|\v|\leq \w $;
\item [(b)] is strictly positive for  $|q|< t $ and $|\v|< \w $.
\end{enumerate}
\end{proposition}

If  $\xx \in\HH_\R$,  then  the adjoint of 
$  a_{q,t}({\xi})\O a_{\v,\w}({\eta})$ is $  a_{q,t}^\ast({\xi})\O a^\ast_{\v,\w}({\eta})$ with respect to  the inner product $\langle \cdot,\cdot \rangle_{\q}.$
This
allows us to define the generalized creator and annihilator. 
\begin{definition} We define  $\FQ$  the {quadrabasic Fock space} which is 
completion of $\F$ with respect to the norm corresponding to $\langle\cdot,\cdot \rangle_{\q}$.
Note that for  $\abs{q}=t$ or $\abs{\v}=\w$ we first have to divide by the
kernel of $P_{\qq}$ before taking the completion. 
For $\xx\in \HH_\R$ we define $\B^\ast_{\xx}:= a^\ast_{q,t}({\xi})\O a^\ast_{\v,\w}(\eta)$. Let $\B_{\xx}=a_{q,t}({\xi})\O a_{\v,\w}(\eta)$ be its adjoint with respect to the inner product $\langle\cdot,\cdot \rangle_{\q}$. 
 Denote by $\state$ the vacuum vector state $\state(\cdot)=\langle\Omega \O \bar \Omega, \cdot\text{ } \Omega\O \bar \Omega\rangle_{\q}$.

\end{definition}
\begin{example}
For $\xi\O \eta  \in \HH_\R$ and $\vec{\xi}_n\O\vec{\eta}_{\bar n}=(\xi_1 \otimes\dots\otimes \xi_n)\O(\eta_{\bar 1}\otimes\dots\otimes \eta_{\bar n})\in\H$ this works as follows
\begin{align*}
&\B^\ast_{\xx}\vec{\xi}_n\O\vec{\eta}_{\bar n}=  (\xi\otimes \xi_1 \otimes\dots\otimes \xi_n)\O(\eta \otimes \eta_{\bar 1}\otimes\dots\otimes \eta_{\bar n}), %\text{ for }n\geq 1\\
%&\B^\ast_{\xx}\Omega\O \bar \Omega=\xx
\\
&\B_{\xx}\vec{\xi}_n\O\vec{\eta}_{\bar n}=
\sum_{\substack{i,j\in [n]}} q^{i-1}t^{n-i}\v^{j-1}\w^{\bar{n}-j}\nawiasl\xi\O \xi_i, \eta\O \eta_{\bar j}\nawiasp
\\& \phantom{============} \times  (\xi_1\otimes \dots\otimes \hat{\xi}_i\otimes \dots\otimes \xi_n)\O(\eta_{\bar 1}\otimes \dots\otimes \hat{\eta}_{\bar j}\otimes \dots\otimes \eta_{\bar n}),
%&\sum_{\substack{i,j\in [n]}} q^{i-1}t^{n-i}\v^{j-1}\w^{\bar{n}-j}\nawiasl\xi\O \xi_i, \eta\O \eta_{\bar j}\nawiasp%\bigotimes_{\substack{(l,k) \text{ is the diagonal element  of  } \\ \{1,\dots, \hat i,\dots, n\} \times    \{\bar 1,\dots, %\hat{\bar{j}},\dots, \bar n\}}} (\xi_l\O \eta_k)%\otimes\dots\otimes(\xi_n\O \eta_n) &&n\geq2, 
%\\ 
%&\B_{\xx}\yy:=\nawiasl \xx, \yy\nawiasp\,\Omega, \\ 
\\&\B_{\xx}\Omega\O \bar \Omega =0. 
\end{align*}
\end{example}

\begin{remark}
\begin{enumerate} [(1).]
\item Quadrabasic Fock space creator and annihilator operators depend on four parameters  $q,t,\v,\w$.
If it is necessary to emphasize  this dependence then we  often write  %$\mathcal{F}^{(\qq)}_{\rm dig}$, 
$\B^{(\q) \ast}$ or $\B^{(\q)}$. We will omit the dependence on $\q$ in the notation if it is clear from the context.
The
same remark applies to other objects, which appear in further parts of the work.
\item  It is easy to see that $\B^\ast: \HH \to \mathbb{B}(\F)$ is linear and $\B: \HH\to \mathbb{B}(\F)$ is anti-linear.
\end{enumerate}

\end{remark}

\subsection{Properties of creation and annihilation operators}\label{Sec3}

If $t=\w=1$, then we get  a nice commutation relation. 
\begin{proposition}\label{commutation}
For $\xi_1\O\eta_1 ,\xi_2\O\eta_2\in\HH_\R$, we have the commutation relation  
\begin{align*}
\B_{\xi_1\O\eta_1}^{(q,1,\v,1)}\B^{(q,1,\v,1)\ast}_{\xi_2\O\eta_2} - q\v\B^{(q,1,\v,1)\ast}_{\xi_2\O\eta_2} \B_{\xi_1\O\eta_1}^{(q,1,\v,1)}
=A_1\O B_2 +  B_1 \O A_2+\nawiasl \xi_1\O\eta_1, \xi_2\O\eta_2\nawiasp \id\O\id
\end{align*}
where $A_1=q a^\ast_{q,1}({\xi_2})a_{q,1}({\xi_1}),$ $A_2=\v a^\ast_{\v,1}({\eta_2})a_{\v,1}({\eta_1}),$ $ B_1=\langle \xi_1,\xi_2\rangle_H \id$ and $ B_2=\langle \eta_1,\eta_2\rangle_{\bar H} \id$.  
\end{proposition}
\begin{remark}
Operators $A_1\O B_2$,  $B_1 \O A_2$ and $\id\O\id$  commute. 
\end{remark}
\begin{proof}
The proof  follows directly from relation \eqref{commutationrelations} between  $a^\ast_{q,1}$ and $a_{\v,1}$, and we see that
\begin{align*}
\B_{\xi_1\O\eta_1}^{(q,1,\v,1)}&\B^{(q,1,\v,1)\ast}_{\xi_2\O\eta_2}
= a_{q,1}({\xi_1})a^\ast_{q,1}({\xi_2})\O a_{\v,1}({\eta_1})a^\ast_{\v,1}({\eta_2})=q a^\ast_{q,1}({\xi_2})a_{q,1}({\xi_1}) \O \v a^\ast_{q,1}({\eta_2})a_{\v,1}({\eta_1})\\&+\langle \xi_1,\xi_2\rangle_H \id\O \v a^\ast_{q,1}({\eta_2})a_{\v,1}({\eta_1}) +  q a^\ast_{q,1}({\xi_2})a_{q,1}({\xi_1})  \O \langle \eta_1,\eta_2\rangle_{\bar H} \id+\nawiasl \xi_1\O\eta_1, \xi_2\O\eta_2\nawiasp \id\O\id
\end{align*}
and the conclusion follows. 
\end{proof}

The norm of the creation operators follows directly from equation \eqref{normaoeratora}. For  
 $\xx \in \HH$, $ \xx \neq0$,  we have 
$$
\|\B^\ast_{\xx}\|_{\q} = \|a^\ast_{q,t}({\xi})\|_{q,t}\|a^\ast_{\v,\w}({\eta})\|_{\v,\w}
$$ 
because $\|\B^\ast_{\xx} \x\O\y\|_{\q}^2=\langle a^\ast_{q,t}(\xi) \x ,a^\ast_{q,t}(\xi) \x \rangle_{q,t}\langle a^\ast_{\v,\w}(\eta)\y  , a^\ast_{\v,\w}(\eta)\y \rangle_{\v,\w}$.

\section{Combinatoric and partitions}\label{sec:kombinatorykapartycje}

\subsection{Diagonal partitions}
For an ordered set $S$, let $\Part(S)$ denote the lattice of set partitions of that set. 
We write $B \in \pi$ if $B$ is a class of $\pi$ and we say that $B$ is a \emph{block of $\pi$}. A block of $\pi$ is called a \emph{singleton} if it consists of one element. %, which we denote by $\Sing(\pi)$.
Similarly, a block of $\pi$ is called a \emph{pair} if it consists of two elements. %, which we denote by $\Pair(\pi)$.
 Let 
$\Sing(\pi)$ and $\Pair(\pi)$ denote the set of all singletons and pairs of $\pi$, receptively.  
The maximal element of $\Part(n)$ under this order is the partition consisting
of only one block and it is denoted by  $\hat{1}_{n}$.
On the other hand, the minimal element $\hat{0}_n$ 
is the unique partition where every block is a singleton.
 Given a partition $\pi$ of the set $[n]$, we write $\Semi(\pi)$ for the set of pairs of integers $(i, j)$ which occur in the same block of $\pi$ such that $j$ is the smallest element of the block greater than $i$. The same notation $\Semi(B)$ is applied to a block $B\in \pi$. Thus, when we draw the points of block then we think that consecutive elements in every block (bigger than one) are connected by arcs above the $x$ axis.
 For a given block of 
partition $\pi\in \P(n)$  every element   is an opener, a closer, a middle point or a singleton for this
block -- see Figure \ref{Polkola}.
%Let 
%$h:[r]\to\N$ be a map. We denote by $\ker h$ the set partition
%which is induced by the equivalence relation 
%$$
%k\sim_{\ker h} l
%\iff
%h(k)=h(l)
%.
%$$
%Similarly, for a multiindex $\underline{i} =( i(1),i(2),\dots ,i(n))\in\N^n$ we define its 
%kernel
%$\ker\underline{i}$ by the relation $k\sim l$ if and only if $i(k)=i(l)$.
%Note that writing $\ker\underline{i}=\pi$ will indicate that $( i(1),i(2),\dots ,i(n))$ is in
%the equivalence class identified with the partition $\pi\in\P(n)$.

\begin{figure}[h]
%\begin{center}block \hspace{4cm} arc\end{center}
\begin{center}
  \MMMatching{4}{1/2, 2/3, 3/4}{2}{1/2}{2} {2/3}{2} {3/4}
   \end{center}
\caption{The example of a block and corresponding arcs.    
}
\label{Polkola}
\end{figure}

\begin{definition} \label{def:partycji}
%We denote by $\P^{\O }(n)$ the  set of \emph{diagonal partitions}  of  $[n]\sqcup [\bar{n}]$ such that 
%\begin{enumerate}[I]
%\item  each block of $\pi$  is associated with $[n]$ or $[\bar{n}]$ (no connection between them);
%\item   $B=(a,\dots, b)$  is a block of at least two elements in $\pi|_{[n]}$ $\iff$  $\bar B=(\bar{a} ,\dots,\bar{c})$  is a block of at least two elements in $\pi|_{[\bar{n}]}$;
%\item $(a, b)$  is an arc in $\pi|_{[n]}$ if and only if    $(\bar{a},\bar{c})$  is an arc  in $\pi|_{[\bar{n}]}$;
%\item  $a$  is singleton in $\pi|_{[n]}$, if and only if  $\bar a$  is singleton in $\pi|_{[\bar{n}]}$.
%\end{enumerate}
We denote by $\P^{\O }(n)$ the  set of \emph{diagonal partitions}  of  $[n]\sqcup [\bar{n}]$ such that each block of $\pi$  is associated with $[n]$ or $[\bar{n}]$ (there is no connection between them) and  
 the collections of openers, middle points and singletons
of blocks of $\pi|_{[ n]}$ and $\pi|_{[\bar n]}$ coincide (this implies that the collections of closers coincide as well).
%\end{enumerate}
% Conditions form Definition \ref{def:partycji}  says that  
%I think this condition is much easier to understand than the assumptions
%in Definition 3.1. In definition 4.5, if the author wants to use only conditions I and II from that
%definition (for partitions into pairs or singletons), they say that the collections of openers on both
%sides coincide. In definition 5.10, the assumptions say that if an element is an opener or a middle
%element on one side, it is an opener or a middle element on the other side, with an additional
%assumption when the element is an opener on both sides.
\end{definition}
When $n$ is even, we call $\pi \in \P^{\O }(n)$ a \emph{ pair partition of $[n]\sqcup [\bar{n}]$}. If $\pi$ is a pair partition, then each block consists of one arc. The set of diagonal pair partitions of $[n]\sqcup [\bar{n}]$ is denoted by $\P_2^{\O }(n)$ and the set of  pairs or singletons of $[n]\sqcup [\bar{n}]$ is denoted by $\P_{1,2}^{\O }(n)$. 

From Definition \ref{def:partycji} it follows that for every block in $B\in \pi|_{[n]}$ there exists a unique  \textit{conjugate block} $\bar B\in\pi|_{[\bar n]}$,  which starts from the same point. 
This  leads to one more definition. 
 We call block $\BL$ a  \emph{diagonal  block} if  $\BL=B\O \bar B:=(B,\bar B)$, where  $B\in \pi|_{[n]}$ and  $\bar B \in \pi|_{[\bar n]}$ are conjugate. 
The set of such diagonal blocks  is denoted by $\MPair(\pi)$.  
A block $\BL$ of $\MPair(\pi)$ is called a \emph{singleton} if $\BL=(a)\O(\bar a)$.
We note that the diagonal  block it is not a block of $\pi$, but a pair of conjugate blocks.
\begin{example}
In Figure \ref{fig:FiguraExemple1} (a), (b),  (c) and (e) we have:  
\begin{align*}
   &\{{\color{red}(1,4)\O(\bar 1 ,\bar 6)},{\color{blue}(2,6)\O(\bar 2,\bar 5)},{\color{green}( 3, 5)\O(\bar 3,\bar 4)}\}\in \MPair(\pi),
   \\
      &\{{\color{red}(1,6)\O(\bar 1 ,\bar 6)},{\color{blue}(2,5)\O(\bar 2,\bar 5)},{\color{green}( 3, 4)\O(\bar 3,\bar 4)}\}\in\MPair(\pi), 
   \\
 &\{{\color{red}( 1,3,6)\O(\bar 1,\bar 4, \bar 6)},{\color{blue}(2,4,5)\O(\bar 2,\bar 3, \bar 5)}\}\in \MPair(\pi),
    \\
 &\{{\color{red}( 1,3)\O(\bar 1,\bar 4, \bar 5, \bar 6)},{\color{blue}(2,4,5,6)\O(\bar 2,\bar 3)}\}\in \MPair(\pi).
  \end{align*}
  \end{example}

\begin{figure}[h]
\begin{center}$(a)$ \hspace{6.5cm} $(b)$\end{center}
\begin{center}
  \MatchingE{6}{1/4, 2/6, 3/5}{6}{1/6, 2/5, 3/4} \hspace{0.5cm}    \MatchingEE{6}{1/6, 2/5, 3/4}{6}{1/6, 2/5, 3/4}
   \end{center}
   \vspace{0.5cm}    
   \begin{center}$(c)$ \hspace{6.5cm} $(d)$\end{center}
\begin{center}
  \MatchingEEE{6}{1/3, 2/4, 3/6,4/5}{6}{1/4, 4/6,2/3,3/5} \hspace{0.5cm}    \Matching{6}{1/4, 2/6, 3/5}{6}{1/5, 2/3, 4/6}
   \end{center}  
      \vspace{0.5cm}    
   \begin{center}$(e)$ \hspace{6.5cm} $(f)$\end{center}
\begin{center}
  \MatchingEEEE{6}{1/3, 2/4,4/5, 5/6}{6}{1/4, 2/3,4/5,5/6} \hspace{0.5cm}    \Matching{6}{1/2, 2/3, 3/4, 5/6}{6}{1/2, 2/3, 3/4,4/5,5/6}
   \end{center} 
\caption{Examples of $\pi\in\P^{\otimes_{ \scriptscriptstyle\S}}(6)$ -- $(a),(b),(c), (e)$ and $\pi\notin\P^{\otimes_{ \scriptscriptstyle\S}}(6)$ -- $(d),(f)$ .    
}
\label{fig:FiguraExemple1}
\end{figure}
\begin{remark}  \label{rem:partycje}
\begin{enumerate}[(1).]
\item One should note that the Definition \ref{def:partycji} means that left legs of arcs are the same  between different  parts of $[n]$ and $[\bar{n}]$. This means that partitions on $[n]$ and $[\bar{n}]$ are not independent, see Figure \ref{fig:FiguraExemple1} $(a),(b),(c),(e)$.  
\item  A set partition $\pi$ is {noncrossing} if $\rc(\pi)=0$ (for definition of $\rc$ see Subsection \ref{sec:restrictedcrosing}). 
From Definition \ref{def:partycji} it follows that  a set of  partitions 
$\pi$ with respect to $[n]\sqcup [\bar{n}]$ is  noncrossing if and only $\pi|_{[n]}$  is noncrosing and $B=(b_1,\dots,b_n)\in \pi|_{[n]} \iff \bar B=(\bar b_1,\dots,\bar b_n)\in \pi|_{[\bar n]}$; 
%i.e. $\bar \pi|_{[\bar n]}=\pi|_{[ n]}$
see Figure \eqref{fig:FiguraExemple1} $(b)$. 
Indeed, if $\pi_{[n]}$ is noncrossing partition on $[n]$ then there exists an arc of  consecutive integers $(a,a+1)\in \Semi(\pi_{[n]})$. By definition there exists an arc $(\bar a,\bar c)\in\Semi(\pi_{[\bar n]})$. If $\bar c\neq \bar{a}+\bar 1$, then $\bar{a}+\bar 1$ can not be a left leg or singleton of some block in $\pi_{[\bar n]}$ and thus arc $(\cdot,\bar{a}+\bar 1)$ must cross $(\bar a,\bar c)$;  hence we get contradiction and thus $(\bar a,\bar c)=(\bar a,\bar{a}+\bar 1)$. Further, we proceed with the same algorithm.
The set of noncrossing blocks is denoted by $\NC^{\O }(n)$. 

\item  Let $S=\{s_{a_i}\}_{i\in I}$ and $T=\{t_{b_i}\}_{i\in I}$ be two  indexing  sets $a_i,b_i \in \N$ with the same cardinality $|S|=|T|$. 
The diagonal Cartesian product of two sets $S$ and $T$, denoted $S\times_\Delta T$ is the set of all ordered pairs $(s_{a_i},t_{b_i})$ where $s_{a_i}$ is in $S$ and $t_{b_i}$ is in $T$. In terms of set-builder notation, that is 
$$S\times_\Delta T=\{(s_{a_i},t_{b_i})|s_{a_i}\in S, t_{b_i}\in T\}.$$
%Then the diagonal mapping on $\{1,\dots,n\}$ is defined as 
%\begin{align*}
%\begin{aligned}
 % \Delta : \{1,\dots,n\} &\to S\times T \\
  % &i\mapsto (s_i,t_i),
%\end{aligned}
%\end{align*}
%$$\delta (x)=(x,x)$$
Clearly $S\times_\Delta T\subset S\times T$, where $S\times T$ is the Cartesian product of $S$ with $T$.
For a
class $B\in \pi$ , denote by $f(B)$ its first element.
Let us number the  block of $\pi$ according to the order of their first elements, i.e. $$\pi=\{B_{f(B)}\mid B\in \pi\},\quad \pi \in \Part([n]).$$ 
Then for $\pi\in  \P^{\O }(n)$ we have 
\begin{align*} \MPair(\pi) =\pi_{[n]}\times_\Delta \pi_{[\bar n]}
%\{& \pi \in \Part_2([n])\times \Part_2([\bar n])\mid  (B_i,\bar B_i)\in \pi_{[n]}\times \pi_{[\bar n]}\text{ for all } i=1\dots,n/2\}.   
\end{align*}
 That is why
 we called it the diagonal partitions. For example 
 if $$\pi=\{(1,4)\,(2,6),( 3, 5),(\bar 1 ,\bar 6),(\bar 2,\bar 5),(\bar 3,\bar 4)\}$$
 then $\quad \pi_{[\bar n]}=\{(1,4)\,(2,6),( 3, 5)\}$, $ \pi_{[\bar n]}= \{(\bar 1 ,\bar 6),(\bar 3,\bar 4),(\bar 2,\bar 5)\}$ and 
 \begin{align*} \MPair(\pi) =\pi_{[n]}\times_\Delta \pi_{[\bar n]}&= \{((1,4), (\bar 1 ,\bar 6)),((2,6),(\bar 2,\bar 5)),( ( 3, 5),\bar 3,\bar 4))\} 
 \\&= \{(1,4)\O (\bar 1 ,\bar 6),(2,6)\O(\bar 2,\bar 5), ( 3, 5)\O(\bar 3,\bar 4)\} .
\end{align*}
%Notice that if we take the Cartesian product of pair partitions $\Part_2([n])$ and  $\Part_2([\bar n])$ 
% then  $\P_2^{\O }(n)$  consists of these pair  partitions $\pi_{[n]}\times \pi_{[\bar n]}$ that lie in the  diagonal of the product   with respect to the first element. 
%Indeed for a
%class B, denote by $f(B)$ its first element. Order
%the block of $\pi$ according to the order of their first elements, i.e. $f(B_1)<f(B_2)\dots <f(B_l)$, then
 %Indeed  if for a
 %block $B$, denote by $f(B)$ its first element, then 
%\begin{align*} \P_2^{\O }(n)=\{& \pi \in \Part_2([n])\times \Part_2([\bar n])\mid  (B_i,\bar B_i)\in \pi_{[n]}\times \pi_{[\bar n]}\text{ for all } i=1\dots,n/2\}.   
%\end{align*}

\item From Definition \ref{def:partycji} it follows that whenever $$\{\hat{1}_{n}\O \bar B\}\in \MPair(\pi) \text{ or }\{B\O \hat{1}_{\bar n}\}\in \MPair(\pi) \text{ for }\pi \in  \P^{\O }(n)$$ then $\hat{1}_{n}\O \bar B=B\O \hat{1}_{\bar n}=\hat{1}_{n}\O \hat{1}_{\bar n}$. 
\item %For a given block of  partition $\pi$ of $[n]$ or $[\bar n]$ every element   is an opener, a closer, a middle point or a singleton for this
%partition.  Conditions form Definition \ref{def:partycji}  says that the collections of openers, middle points and singletons
%of $\pi_{[ n]}$ and $\pi_{[\bar n]}$ coincide. This implies that the collections of closers coincide as well. 
The Definition \ref{def:partycji}  say that
\begin{enumerate}[I]
\item  each block of $\pi$  is associated with $[n]$ or $[\bar{n}]$ (no connection between them);
\item   $B=(a,\dots, b)$  is a block of at least two elements in $\pi|_{[n]}$ $\iff$  $\bar B=(\bar{a} ,\dots,\bar{c})$  is a block of at least two elements in $\pi|_{[\bar{n}]}$;
\item $(a, b)$  is an arc in $\pi|_{[n]}$ if and only if    $(\bar{a},\bar{c})$  is an arc  in $\pi|_{[\bar{n}]}$;
\item  $a$  is singleton in $\pi|_{[n]}$, if and only if  $\bar a$  is singleton in $\pi|_{[\bar{n}]}$.
\end{enumerate}
\item It is worth to emphasize that full Fock space $(\bigoplus_{n=0}^\infty H^{\otimes n})\otimes (\bigoplus_{n=0}^\infty \bar H^{\otimes n})$  is studied in the context of  semi-meander polynomials which are
used in the enumeration of semi-meandric systems; see  \cite{NicaPing2018}. In the context of enumeration of such objects we can say that the diagonal pair partition corresponds to a special subset  of  the self-intersecting meandric system $(t=\w=1)$, such that closed  pairs which intersect the x-axis  have the same parity. We skip formal description of it and just heuristically explain that we can combine negative and positive parts by $i\leftrightarrow \bar i$  as in the figure: 
%\begin{figure}[h]
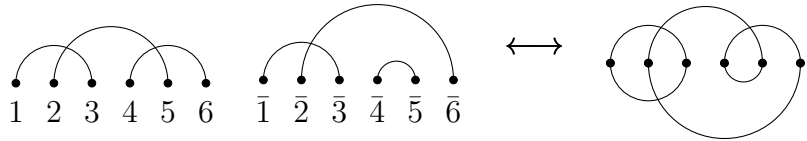
\begin{center}
\Matching{6}{1/3, 2/5, 4/6}{6}{4/5, 1/3,2/6}  \MatchingMeanders{6}{1/3, 2/5, 4/6}{6}{4/5, 1/3,2/6}
\end{center} 
It is worth to mention that the number of such pair partitions is the Euler number; see Corollary \ref{cor:euler}.
\end{enumerate} 
\end{remark}
\subsection{Statistics}We introduce some partition statistics for $\pi \in \P(n)$ which are necessary to obtain  the  moment-cumulant formula. 
 These  statistics  are naturally extended to the set $\P^{\O}(n)$,  i.e. separately for the parts $[n]$ and $[\bar n]$.
  We say that an arc (or pair) $(i,j)$ \emph{is crossing} the arc $(i',j')$ if  $i < i' < j < j'  \textrm{ or } i' < i < j' < j $ and similarly, we say that  an arc (or pair) $(i,j)$ \emph{nests} $(i',j')$ if $i <k <j$ for any $k\in \{i',j'\}$. Similarly we say that  an arc (or pair) $(i,j)$ \emph{covers} singleton $k$ if $i <k <j$.  
\subsubsection{Statistics related to Gausian operator -- pairs and singletons}   
For a set partition $\pi\in\P_{1,2}(n)$ let $\Cr(\pi)$  be the number of crossings of $\pi$, i.e.\ 
\begin{align*}
\Cr(\pi)&=\#\{(V,W) \in \Pair(\pi) \times\Pair( \pi) \mid V \textrm{ is crossing } W \}. 
\intertext{
Let $\InS(\pi)$ be the number of pairs of a singleton and a covering block: }  
\InS(\pi)&=\#\{(V,W) \in \Sing(\pi) \times\Pair( \pi) \mid   W \text{~covers~}V  \}. 
  \intertext{Let $\Cov(\pi)$ be the number of pairs of a nesting block: }
\Cov(\pi)&=\#\{(V,W) \in  \Pair(\pi) \times\Pair( \pi) \mid  W \text{~nests~}V  \}. 
\intertext{Let $\SL(\pi)$ be the number of singleton  to the right of pairs:  }
\SL(\pi)&=\#\{(V,W) \in \Sing(\pi) \times \Pair(\pi) \mid  
 V =(i) \text{ and } i > j \text{~for all~} j\in W \}. 
\end{align*}
\subsubsection{Statistics related to gauge operator -- block of size at least three} \label{sec:restrictedcrosing} 
For  $\pi\in\P(n)$ we define two statistics, which are related to the gauge operator. 

\noindent \textbf{\emph{Restricted crossings}}. 
 We use the same definition of \emph{restricted crossings} as given in Biane \cite{Bia97}, namely
\begin{align*}
\rc ( B,\widetilde{B}):
= &\# \big\{(V,W)\in \Semi(B)\times \Semi(\widetilde{B}) \mid
 \textrm{ such that } V \textrm{ is crossing } W \big\}.
\end{align*}
For a set partition $\pi\in\P(n)$ let $\rc(\pi)$ be the number of  restricted crossings of $\pi$: 
$$\rc{(\pi)} := \sum_{i<j} \rc(B_i, B_j),$$ where $\pi\setminus \Sing(\pi)=\{B_1,\dots,B_l\}$.

\noindent \textbf{\emph{Restricted  nestings}}. 
 Now we define the number of \emph{restricted  nestings} of the partition $\pi\in\P(n)$.  The set of restricted nestings of $B,\widetilde{B}$ is
\begin{align*} \nest(B,\widetilde{B}):=&\#\big\{(V,W)\in \Semi(B)\times \Semi(\widetilde{B})\mid V \text{~nests~} W\textrm{ } or \textrm{ } W \text{~nests~} V\big\},
\end{align*}
and the set of restricted nestings of $\pi$ is 
\begin{align*}\nest(\pi) := \sum_{i<j} \nest(B_i, B_j),\end{align*} 
where $\pi\setminus \Sing(\pi)=\{B_l,\dots,B_l\}$. 
\section{Moments of Gaussian  operator}
We present an example of a generalized Gaussian operator. It is given by
creation and annihilation operators on a quadrabasic Fock space.
We show that the distribution of these operators with
respect to the vacuum expectation is a generalized Gaussian distribution, in the
sense that all moments can be calculated from the second moments with the help of
a combinatorial formula.
\subsection{Gaussian operator}
\begin{definition}
The operator 
\begin{equation}
\G_{\xx}= \B_{\xx} +\B^\ast_{\xx},\qquad \xx \in \HH_\R
\end{equation}
on $\FQ$ is called the  {quadrabasic  Gaussian operator}.
\end{definition}
\begin{remark}
\begin{enumerate}[(1).]
\item From estimating the norm of the creation (annihilator) operators we conclude that $\G_{\xx}$ is bounded,  whenever $|q|,|\v|<1.$
\item If $q=t=\v=\w=1$, then  the kernel of the symmetrization is nontrivial. In this case the square of the operator $P_{(1,1,1,1)}^{(n)}$ given by the  expression ${[P_{(1,1,1,1)}^{(n)}]}^2=\frac{1}{n!^2}P_{(1,1,1,1)}^{(n)}$  projects $\H$ to the space of  special  symmetric functions ${\tilde {\bf \mathcal H}_{ n}}$
\begin{align*}
\vec{\xi}_n\tilde{\O}\vec{\eta}_{\bar n}:&=  (\xi_1 \tilde \otimes\dots\tilde \otimes \xi_n)\tilde \O(\eta_{\bar 1}\tilde \otimes\dots \tilde \otimes \eta_{\bar n}) \\
&=P_{(1,1,1,1)}^{(n)}(\xi_1  \otimes\dots  \otimes \xi_n) \O(\eta_{\bar 1} \otimes\dots  \otimes \eta_{\bar n})
\\& = \big(\sum_{\sigma \in \S_n}\xi_{\sigma(1)}  \otimes\dots  \otimes \xi_{\sigma(n)}\big) \O\big(\sum_{\sigma \in \S_{\bar n}} \eta_{\sigma(\bar 1)} \otimes\dots  \otimes \eta_{\sigma(\bar n)}\big)
 \intertext{i.e.  we have }
 \sigma(\vec{\xi}_n\tilde{\O}\vec{\eta}_{\bar n})&=\vec{\xi}_n\tilde{\O}\vec{\eta}_{\bar n}\quad \sigma \in \W. 
\end{align*}
More precisely, the operator $\B^\ast_{\xi_1\O \eta_{\bar 1}}$ for $\vec{\xi}_{n-1}\tilde \O\vec{\eta}_{\overline{n-1}}=(\xi_2 \tilde \otimes\dots \tilde \otimes \xi_{n-1})\tilde \O(\tilde \eta_{\bar 2}\tilde \otimes\dots\tilde \otimes \eta_{\overline{ n-1}})\in{\tilde {\bf \mathcal H}_{ n}}$ and  $\xi_1\O \eta_{\bar 1} \in \HH_\R$   works in the following way 
\begin{align*}
\B^\ast_{\xi_1\O \eta_{\bar 1}}\vec{\xi}_{n-1}\tilde \O\vec{\eta}_{\overline{n-1
}} &= (\xi_1 \tilde \otimes \vec{\xi}_{n-1})\tilde \O(\eta_{\bar 1} \tilde \otimes  \vec{\eta}_{\overline{n-1}} ) \\& = \big(\sum_{\sigma \in \S_n}\xi_{\sigma(1)}  \otimes\dots  \otimes \xi_{\sigma(n)}\big) \O\big(\sum_{\sigma \in \S_{\bar n}} \eta_{\sigma(\bar 1)} \otimes\dots  \otimes \eta_{\sigma(\bar n)}\big). 
%
%\vec{\xi}_{n-1}\tilde \O(\vec{\eta}_{\overline{n-1
%}} & (\xi_0\tilde    \otimes \xi_1 \otimes\dots \tilde \otimes \xi_n+\dots + \xi_1 \tilde  %\otimes\dots \tilde \otimes \xi_n \tilde  \otimes \xi_0) \O(\eta_0 \tilde \otimes \eta_{\bar 1}%\tilde \otimes\dots\tilde \otimes \eta_{\bar n}+ \dots+\eta_{\bar 1}\tilde \otimes\dots\tilde %\otimes \eta_{\bar n}\tilde \otimes \eta_0  )
\end{align*}
%where as usually $\hat{\eta}_i$ means omitting the $i$-th term.

Thus the scalar product automatically implements the additional relations 
 $\B^\ast_{\xi_1\O\xi_{\bar 1}}\B^\ast_{\xi_2\O\xi_{\bar 2}}=\B^\ast_{\xi_2\O\xi_{\bar 2}}\B^\ast_{\xi_1\O\xi_{\bar 1}}$ on ${\tilde {\bf \mathcal H}_{ n}}$. From this we conclude that when $q=t=\v=\w=1$,  the operators $ \G_{\xi_1\O\xi_{\bar 1}}$ and $ \G_{\xi_2\O\xi_{\bar 2}}$ are commutative on ${\tilde {\bf \mathcal H}_{ n}}$.   
Notice that in this case if $\xi_1 \perp \xi_2$ and  $\xi_{\bar 1} \perp \xi_{\bar 2}$, then 
$\state(\G_{\xi_1\O\xi_{\bar 1}}^n\G_{\xi_2\O\xi_{\bar 2}}^m)=\state(\G_{\xi_1\O\xi_{\bar 1}}^n)\state(\G_{\xi_2\O\xi_{\bar 2}}^m)$  which coincides
with classical independence. % (this may be seen from Theorem \ref{momentymieszane}).
 \end{enumerate}
\end{remark}
 
 \subsection{Orthogonal polynomials } \label{wielortogonalne}
For a probability measure $\mu$ with finite moments of all orders, let us orthogonalize the sequence $(1,x,x^2,x^3,\dots)$ in the Hilbert space $L^2(\R,\mu)$, following the Gram-Schmidt method. This procedure yields orthogonal polynomials $(P_0(x), P_1(x), P_2(x), \dots)$ with $\text{deg}\, P_n(x) =n$. Multiplying by constants, we take $P_n(x)$ to be monic, i.e., the coefficient of $x^n$ is 1. It is known that they satisfy a recurrence relation
\begin{align} \label{wielortogonalnerekursia}
x P_n(x) = P_{n+1}(x) +\beta_n P_n(x) + \gamma_{n-1} P_{n-1}(x),\qquad n =0,1,2,\dots
\end{align}
with the convention that $P_{-1}(x)=0$. The coefficients $\beta_n$ and $\gamma_n$ are called \emph{Jacobi parameters} and they satisfy $\beta_n \in \R$ and $\gamma_n \geq 0$. 
It is known that 
\begin{equation}\label{eq54}
\gamma_0 \cdots \gamma_n=\int_{\R}|P_{n+1}(x)|^2\mu(d x),\qquad n \geq 0.
\end{equation} 
Moreover, the measure $\mu$ has a finite support of cardinality $N$ if and only if $\gamma_{N-1}=0$ and $\gamma_n > 0$ for $n = 0,\dots, N-2$. 

The continued fraction representation of the Cauchy transform can be expressed in terms of the Jacobi parameters:
\[
\int_{\R}\frac{\mu(d t)}{z-t} = \dfrac{1}{z-\beta_0 -\dfrac{\gamma_0}{z-\beta_1-\dfrac{\gamma_1}{z- \beta_2 - \cdots}}}.
\]
This representation is useful to calculate the Cauchy transform when Jacobi parameters are given. More details can be found in \cite{HO07}. Notice that if $\lim_{n\to \infty }\beta_n=\beta>0 $ and $\lim_{n\to \infty }\gamma_n=\gamma $ then absolutely continuous part of measure $\mu$ is supported on $\left(-2\sqrt{\gamma}+\beta,{2}{\sqrt{\gamma}+\beta}\right),$ see \cite{Chihara1978}.

Let $(Q_n^{(\q)}(x))_{n=0}^\infty$ be the quadrabasic Hermite orthogonal polynomials \eqref{recursion}. 
The orthogonalizing probability measure $\qMP_{\qq}$ is unknown in general case but from Section \ref{wielortogonalne}  we see that if $t=\w=1$, then absolutely continous part of $\qMP_{q,1,\v,1}$ is supported on 
$\left(\frac{-2}{\sqrt{1-q}\sqrt{1-\v}},\frac{2}{\sqrt{1-q}\sqrt{1-\v}}\right).$
\begin{theorem} \label{rozkladgaussa} Suppose that  $ q,\v\in(-1,1)$,  $\xx \in \HH_\R$ and $ \|\xx\|=1$. Let $\m$ be the probability distribution of $\G_{\xx}$ with respect to the vacuum state. Then $\m$ is equal to $\qMP_{\qq}$. 
\end{theorem}
\begin{proof} 
Let $\gamma_{n-1}= [n]_{q,t}[n]_{\v,\w}$, then 
%$$\|(\xx)^{\otimes n}\|_{\q}^2=[n]_{q_1}! [n]_{q_2}!$ 
\begin{align*}
\|\xi^{\otimes n}\O\eta^{\otimes n}\|_{\q}^2
&= \langle \xi^{\otimes n},P_{q,t}^{(n)}\xi^{\otimes n}\rangle_{0,0} \langle \eta^{\otimes n},P_{\v,\w}^{(n)}\eta^{\otimes n}\rangle_{0,0} \\&
 =\langle \xi^{\otimes n},( I\otimes P^{(n-1)}_{q,t})R^{(n)}_{q,t}\xi^{\otimes n}\rangle_{0,0}\langle \eta^{\otimes n},( I\otimes P^{(n-1)}_{\v,\w})R^{(n)}_{\v,\w}\eta^{\otimes n}\rangle_{0,0}
 \intertext{We remind that the operator $R^{(n)}_{q,t}$ given by $R^{(n)}_{q,t} = t^{n-1}+\sum_{k=1}^{n-1}q^{k}t^{n-k-1}\sigma_{1}\cdots \sigma_{k}$  acts as follows $R^{(n)}_{q,t}\xi^{\otimes n}=[n]_{q,t}\xi^{\otimes n}$, because $\sigma_{1}\cdots \sigma_{k}\xi^{\otimes n}=\xi^{\otimes n}$ and $[n]_{q,t}=t^{n-1}+\sum_{k=1}^{n-1}q^{k}t^{n-k-1}$.
 We now expand further $R^{(n)}_{q,t}\xi^{\otimes n}=[n]_{q,t}\xi^{\otimes n}$,  $R^{(n)}_{\v,\w}\eta^{\otimes n}=[n]_{\v,\w}\eta^{\otimes n}$ and obtain 
 }
 &=%\left([n]_{q_j}\langle \xi_j^{\otimes n},( I\otimes P^{(n-1)}_{q_j})\xi_j^{\otimes n}\rangle_{0}\right)=  
\|\xi\|^2\|\eta\|^2 [n]_{q,t} [n]_{\v,\w}\|\xi^{\otimes (n-1)}\O\eta^{\otimes (n-1)}\|_{\q}^2 =\gamma_0\gamma_1\cdots \gamma_{n-1}. 
\end{align*}
and hence from \eqref{eq54} it follows that 
\begin{equation}
\|\xi^{\otimes n}\O\eta^{\otimes n}\|_{\q}= \|Q_n^{(\q)}\|_{L^2},\qquad n\in\N\cup\{0\}. 
\end{equation}
Therefore, the map $\Phi\colon (\text{span}\{\xi^{\otimes n}\O\eta^{\otimes n}\mid n \geq 0\}, \|\cdot\|_{\q}) \to L^2(\R,\qMP_{\qq})$ defined by $\Phi(\xi^{\otimes n}\O\eta^{\otimes n})= Q_n^{(\q)}(x)$ is an isometry. 
Note that  
\begin{align*}
\G_{\xx}\text{ }\xi^{\otimes n}\O\eta^{\otimes n} 
&= \B^\ast_{\xx}\text{ }\xi^{\otimes n}\O\eta^{\otimes n}+\B_{\xx} \text{ }\xi^{\otimes n}\O\eta^{\otimes n}
\\&= \xi^{\otimes (n+1)}\O\eta^{\otimes (n+1)}+(a_{q,t}(\xi)\xi ^{\otimes n})\O (a_{\v,\w}(\eta)\eta ^{\otimes n})
 % \\
%&= (\xx)^{\otimes (n+1)}+\sum_{i_1,i_2=1}^n q_{1}^{i_{1}-1} q_{2}^{i_{2}-1} \r_{\xx} \text{ }\pix_{i_1,\dots,i_m} \xx ^{\otimes n}  
\\
&=\xi^{\otimes (n+1)}\O\eta^{\otimes (n+1)}+([n]_{q,t}\xi ^{\otimes (n-1)})\O ([n]_{\v,\w}\eta ^{\otimes (n-1)})\\
&= \xi^{\otimes (n+1)}\O\eta^{\otimes (n+1)}+[n]_{q,t}[n]_{\v,\w}\xi^{\otimes (n-1)}\O\eta^{\otimes (n-1)}, 
\end{align*}
%where Proposition \ref{prop2} was used on the second line. 
Hence, by induction we can compute $\G_{\xx}^n\Omega\O \bar \Omega$ and show that $\Phi(\G_{\xx}^n\Omega \O \bar \Omega) = x^n$. Since $\Phi$ is an isometry we get 
$\langle \Omega\O \bar \Omega, \G_{\xx}^n\Omega\O \bar \Omega\rangle_{\q} = m_n(\qMP_{\qq})$ for  $n\in \N$. 
%For odd integers $n$ we can show that $\langle \Omega\O \bar \Omega, \G_{\xx}^n|\Omega\O \bar \Omega\rangle_{\q} =0= m_n(\qMP_{\qq})$.  
Since $\qMP_{\qq}$ is compactly supported, probability measures giving the moment sequence $m_n(\qMP_{\qq})$ are unique and hence $\qMP_{\qq}=\m$.  
\end{proof}
\subsubsection{Interesting cases of orthogonal polynomials }  \label{subsec:intersingpolynomial}
We present a class of orthogonal polynomials  corresponding to known
measures.
\newline
\noindent \textbf{$q$-Meixner-Pollaczek polynomials $q=\v$ and $t=\w=1$}. 
For $-1 < \alpha,q <1$ let $(\tilde{P}_n^{(\alpha,q)}(x))_{n=0}^\infty$ be the orthogonal polynomials with the recursion relation
\begin{equation}\label{recursion2}
x\tilde{P}_n^{(\alpha, q)}(x) = \tilde{P}_{n+1}^{(\alpha, q)}(x) +[n]_q(1 + \alpha q^{n-1})\tilde{P}_{n-1}^{(\alpha, q)}(x), \qquad n=0,1,2,\dots
\end{equation}
where $\tilde{P}_{-1}^{(\alpha,q)}(x)=0,\tilde{P}_0^{(\alpha,q)}(x)=1$. These polynomials are called \emph{$q$-Meixner-Pollaczek polynomials}. The orthogonalizing probability measure $\qqMP_{\alpha,q}$ is known in \cite[(14.9.4)]{KLS10}, supported on $(-2/\sqrt{1-q}, 2/\sqrt{1-q})$ and absolutely continuous with respect to the Lebesgue measure with density 
\begin{equation}\label{eq10}
\frac{d \qqMP_{\alpha,q}}{dt}(x)=  \frac{(q;q)_\infty(\beta^2; q)_\infty}{2\pi\sqrt{4/(1-q) -x^2}}\cdot\frac{g(x,1;q) g(x,-1;q) g(x,\sqrt{q};q) g(x,-\sqrt{q};q)}{g(x, \underline{i} \beta;q)g(x,-\underline{i}\beta;q)}
\end{equation}
where 
\begin{align}
&g(x,b;q)= \prod_{k=0}^\infty(1-4 b x (1-q)^{-1/2} q^k + b^2 q^{2k}), \\
&(s;q)_\infty=\lim_{n\to\infty}(s; q)_n=\prod_{k=0}^\infty(1-s q^{k}),\qquad s \in \R, \\
%=(2 \underline{i}  b e^{\underline{i} \theta}/\sqrt{1-q}, -2 \underline{i} b e^{\underline{i} \theta}/\sqrt{1-q};q)_\infty, \\ 
&\beta
= 
\begin{cases}
\sqrt{-\alpha}, & \alpha \leq 0, \\
\ri \sqrt{\alpha}, & \alpha \geq0.   
\end{cases}
\end{align}

\begin{proposition} \label{MeixnerPollaczekPolynomials}The measure  $\qMP_{(q,1,q,1)}$ belongs to the family $\qqMP_{-q,q}$ for $\abs q <1$.
\end{proposition} 
\begin{proof}%Recall that random variables $G^{(q,1,q,1)}_{\xx}$, with $\|\xi\|=1$ and $\|\eta\|=1$ have the distribution $\qMP_{q,1,q,1}$, i.e. they ortogonalize \eqref{recursion}. 
If we put $\alpha=-q$ in the recurrence \eqref{recursion2}, then 
\begin{align}
y U_n(y) = U_{n+1}(y)  + [n]_q(1 - q^{n}) U_{n-1}(y), \qquad n\geq 1. \label{pomtfreepoisson}
\end{align}
Now, let us substitute $L_n(y)=U_n(y\sqrt{1-q})/{(1-q)}^{\frac{n}{2}}$, multiply \eqref{pomtfreepoisson} by ${(1-q)}^{\frac{-n-1}{2}}$ and replace $y $ by $\sqrt{1-q}y$,  then we get the recursion
$$
y L_n(y) = L_{n+1}(y)  + [n]_q^2 L_{n-1}(y), \qquad n\geq 1, \label{IsmailSalam}
$$
with $L_{0}(y)=1$ and $L_{1}(y)=y$, so we see that $L_n(y)={Q}_{n}^{(q,1,q,1)}(y).$ This observation means that the 
 recurrence \eqref{pomtfreepoisson} corresponds to monic orthogonal polynomials which orthogonalize  the distribution of $\sqrt{1-q}G^{(q,1,q,1)}_{\xx}$.
\end{proof}

\noindent\textbf{Hyperbolic secant}. 
Professor M. Ismail let us know that the case $\v=\w=1$ is interesting because in this situation we get a measure  interpolated between classical normal ($q=t=0$) and hyperbolic secant distribution ($q=t=1$) i.e. we have the relation ($Q_n^{(q,t)}:=Q_n^{(q,t,1,1)}$)
\begin{align*}
x Q_n^{(q,t)}(x) = Q_{n+1}^{(q,t)}(x) +[n]_{q,t}nQ_{n-1}^{(q,t)}(x), \qquad n=0,1,2,\dots.
\end{align*}
Now we will explain that the moment problem is determined in this situation. 
Probably the best known criterion (in this case) of Hamburger moment problem is due to Carleman \cite{Carleman,Shohat1943,Chihara}.  Carleman's theorem states that the moment problem is determined if 
$$\sum_{n\geq 1}\gamma_{n}^{-\frac{1}{2}}=\infty, $$
where $\gamma_n$ are Jacobi parameters from \eqref{wielortogonalnerekursia}. In this case we have 
$$\infty=\sum_{n\geq 1}\frac{1}{n}=\sum_{n\geq 1}\left(n^2\right)^{-\frac{1}{2}}\leq \sum_{n\geq 1}\left([n]_{q,t}n\right)^{-\frac{1}{2}} .$$
Hence, the moments uniquely determine the measure and we  can use the argument from the proof of Theorem \ref{rozkladgaussa} to conclude that $\G_{\xx}^{(q,t,1,1)} $ has the distribution $\mu_{q,t,1,1}$. 

In particular we show that  $\G_{\xx}^{(1,1,1,1)} $ has the 
hyperbolic  secant  distribution; see \eqref{recursionhiperbolic}.  
 A random variable follows a hyperbolic secant distribution if its probability density function can be related to the following standard form $\rho(dx)=\frac{1}{2\cosh(\pi x/2)}dx$. We note that the
characteristic function of this distribution is  $\frac{1}{\cosh(x)}$ i.e. the density and its characteristic function differ only by scale parameters (the normal distribution is the prime example for this phenomenon)  -- see \cite{Tsehaye2004,BozejkoHasebe2013}.
The  secant distribution  is infinitely divisible distribution generated by some
particular processes with stationary independent increments (L\'evy processes -- see \cite{Pitman2003}).
The moments $E_n=m_{2n}(\rho)$  are Euler numbers with positive signs (see \cite[eq. (8)]{Pitman2003} or \cite[Chapter 23]{Abramowitz})
$$(m_0(\rho),m_2(\rho),m_4(\rho),m_6(\rho),m_8(\rho),\dots)=(1, 1, 5, 61, 1385, 50521,\dots).$$

\noindent\textbf{Discrete $q$-Hermite I polynomials}. 
Important classes of orthogonal polynomials studied here are the continuous and the discrete
q-Hermite I polynomials, which are both special cases of the Al-Salam–Chihara polynomials.
%The last interesting example to discuss
%is the case of discrete $q$-Hermite I polynomials (Al-Salam–Carlitz polynomials , which 
We recover them for $q=\w$, $t=1$ and  $\v=0$; see \cite[(3.28.3)]{KLS10} or \cite[(1.6)]{BergIsmail1996} (after simple manipulation as in the proof of Proposition \ref{MeixnerPollaczekPolynomials}). The measure
is known to be discrete \cite[Pages  195-198]{Chihara1978}.

\subsection{Gaussian moment} In the proof below,    we also use additional notation and definitions.
  \begin{definition}
 Let  $\PS_{1,2}^{\O}(n)$ be a set of diagonal partitions of  $[n]\sqcup [\bar{n}]$  such that every block is a pair or a singleton, each block of $\pi$  is associated with $[n]$ or $[\bar{n}]$  and the  collections of openers of pair  coincide.  
 \end{definition}
We emphasize  that from the definition above it follows that for $\pi\in\PS_{1,2}^{\O}(n)$ the number of singletons in $\pi|_{[n]}$ is the same as in $\pi|_{[\bar n]}$, which we use in the proof, and if $a$ is a singleton in $\pi|_{[n]}$, then it does not mean that   $\bar a$ is a singleton in $\pi|_{[\bar n]}$;
%\item  condition $III$ from Definition \ref{def:partycji} is automatically true.
%$$\pi = \{(1,4,6,7)_\primes,(2)_E,(3,5)_\primes,(9)_\primes,(8,10)\}\sqcup \{(\bar 1,\bar 3, \bar 4,\bar 6,\overline {10})_\primes ,(\bar 2)_\primes,(\bar 5)_\primes,(\bar 7)_\primes,(\bar 8,\bar 9)\}$$
%\end{itemize} 
For these partitions we define
 $$\MPairPair(\pi)=\{B\O\bar B\mid B \text{ is a pair and }\bar B \text{ is a conjugate pair}\}.$$
 \begin{example}\label{ex:extendedpart} For example 
 if $\pi=\{(1,3),(2)\}\sqcup \{(\bar 1 ,\bar 2),(\bar 3)\}$, then $\pi\in \PS_{1,2}^{\O}(3)$ and $\pi\notin \P_{1,2}^{\O}(3)$. In this case we also have $\MPairPair(\pi)=\{(1,3)\O(\bar 1 ,\bar 2)\}$ and on this set  $\MPair(\pi)$ is undefined because  singletons are not coincide. 
\end{example}

   For a given diagonal pair $\BL=(a,b)\O (\bar a,\bar c)$, we denote the left and right legs  of $\BL$ 
 by $\lb_\BL=a$, $\bar{\lb}_\BL=\bar{a}$, $\rb_\BL=b$ and $\bar{\rb}_\BL=\bar{c}$. 
\begin{theorem}\label{momentymieszane}
Suppose that  $\xi_{i}\O \xi_{\bar i} \in  \HH_{\R}$, $i\in\{1,\dots,n\}$, then
\begin{align} \label{eq:multipair}
\state( \G_{\xi_{1}\O \xi_{\bar 1}}\cdots \G_{\xi_{n}\O \xi_{\bar n}})=\sum_{\pi\in\P_{2}^{\O }(n)}  \ProdQ\prod_{{{\BL \in \MPair(\pi)}}}
%&\nawiasl \lb^B_{{(\xi_{1}\O \xi_{\bar 1})\otimes\dots  \otimes (\xi_{n}\O \xi_{\bar n})}}, \rb^B_{{(\xi_{1}\O \xi_{\bar 1})\otimes\dots  \otimes (\xi_{n}\O \xi_{\bar n})}}\nawiasp
\nawiasl  \xi_{\lb_\BL}\O \xi_{\bar \lb_\BL}, \xi_{ \rb_\BL}\O \xi_{\bar \rb_\BL}\nawiasp
\end{align}
\end{theorem}
\begin{remark}
Let $S=\{i_1,\dots,i_k\}$ such that $i_1<\dots <i_k$ be a finite subset of natural numbers and $\{T_i\}_{i\in\N}$ be the sequence of operators then
when we write 
$\prod_{i\in S}T_i $ we mean that $T_{i_1}\dots T_{i_k}.$
\end{remark}
\begin{proof}
Given
 $\epsilon=(\epsilon(1), \dots, \epsilon(n))\in\{1,\ast\}^{n}$, let
 $\PS_{1,2;\epsilon}^{\O }$ be the set of partitions $\pi\in
 \PS_{1,2}^{\O }$ such that
 \begin{itemize}
   \item if $(a,b)$ is a pair in
     $\pi|_{[n]}$ then $\epsilon(b)= \ast,$ $\epsilon(a)=1$,
        \item if $(\bar a,\bar b)$ is a pair in
     $\pi|_{[\bar n]}$ then $\epsilon(\bar b)= \ast,$ $\epsilon(\bar a)=1$,
     \item if $\{c\}$ is a singleton in $\pi$ then $\epsilon( |c|)=\ast$. 
   \end{itemize}

Let   $\epsilon(i)\in\{1,\ast\}$.  We will 
prove the result for $Z=\{1,\dots,n\}$ by induction
 %We will prove that for $Z=\{2,\dots,n\}$, we have
\begin{equation}\label{formula101}
\begin{split} 
\prod_{i\in{Z}}\B^{\epsilon(i)}_{\xi_{i}\O \xi_{\bar i}}\Omega\O \bar \Omega = &\sum_{\pi\in\PS_{1,2;\epsilon}^{\O }(Z)}{{{q}^{\Cr(\pi|_{Z})+\InS(\pi|_{Z})}}{{t}^{\Cov(\pi|_{Z})+\SL(\pi|_{Z})}}{{\v}^{\Cr(\pi|_{\bar Z})+\InS(\pi|_{\bar Z})}}{{\w}^{\Cov(\pi|_{\bar Z})+\SL(\pi|_{\bar Z})}}}
\\
& \prod_{{{\BL \in \MPairPair(\pi)}}}\nawiasl  \xi_{\lb_\BL}\O \xi_{\bar \lb_\BL}, \xi_{ \rb_\BL}\O \xi_{\bar \rb_\BL}\nawiasp
\times (\otimes_{i\in \Sing(\pi|_{Z})} \xi_{i})\O (\otimes_{i\in \Sing(\pi|_{\bar{Z}})} \xi_{i}),
\end{split}
\end{equation}
it is assumed to be true for $Z=\{2,\dots,n\}$. 
%Indeed, we give the proof by induction. 
When $n=1$, $\B_{\xi_{1}\O \xi_{\bar 1}}\Omega\O\bar \Omega =0$ and $\B^\ast_{\xi_{1}\O \xi_{\bar 1}}\Omega\O\bar \Omega=\xi_{1}\O \xi_{\bar 1}$ and hence the formula is true. Suppose that the formula \eqref{formula101} is true for $Z=\{2,\dots,n\}$. %Then for any $\epsilon\in\{1,\ast\}^{n-1}$, we get  

 We will show that the action of $\B^{\epsilon(1)}_{\xi_{1}\O \xi_{\bar 1}}$ corresponds to the inductive pictorial description of set partitions. 
We fix $\pi\in \PS^{\O }_{1,2;\epsilon}(\{2,\dots,n\})$ and run the argument below over all partitions of this type. Suppose that 
\begin{itemize}
\item $\pi|_{{\{2,\dots,n\}}}$ has singletons  $s_{1}<\cdots <s_{p_1} <  \cdots <s_r$ and pair blocks $W_{1},\dots, W_{u_1}$ which cover $s_{p_1}$ and pairs $U_1,\dots, U_{l_1}$ to the left of $s_{p_1}$,
\item $\pi|_{{\{\bar 2,\dots,\bar n\}}}$ has singletons  $k_1<\cdots <k_{p_2} <  \cdots <k_r$ and pair blocks
 $\bar W_{1},\dots,\bar   W_{u_2}$ which cover $k_{p_2}$ and pairs $\bar U_1,\dots, \bar U_{l_2}$ to the left of $k_{p_2}$. 
\end{itemize}
  Note that when there is no singleton, the arguments below can be modified easily. 

Case 1. If $\epsilon(1)=\ast$, then the operator $\B^\ast_{\xi_{1}\O \xi_{\bar 1}}$ acts on the tensor product, putting $\xi_{1}\O \xi_{\bar 1}$ on the left. This operation pictorially corresponds to adding the singletons  as follows
$$\{s_1,\cdots,s_r\}\sqcup \{k_1,\cdots,k_r\}\mapsto \{1,s_1,\cdots,s_r\}\sqcup \{\bar 1, k_1,\cdots,k_r\}.$$ This yields a new  partition $\tilde{\pi}\in\PS^{\O}_{1,2;\epsilon}(n)$. This map $\pi\mapsto \tilde{\pi}$ does not change the numbers related to our statistic like $ \Cr(\cdot)$, which is compatible with the fact that the action of $\B^\ast(\xi_{1}\O \xi_{\bar 1})$ does not change the coefficient, hence the formula \eqref{formula101} holds when we moved form  $n-1$ to $n$ and $\epsilon(1)=\ast$.

Case 2. If $\epsilon(1)=1$, then $\B_{\xi_{1}\O \xi_{\bar 1}}$  acts on the tensor product, contributing to new $r^2$ terms. In the $(p_1, p_2)^{\rm th}$ term in the action we obtain the term $q^{p_1-1}\v^{p_2-1} t^{r-p_1}\w^{r-p_2} \langle \xi_{1},\xi_{s_{p_1}}\rangle_H \langle \xi_{\bar 1},\xi_{k_{p_2}}\rangle_{\bar H}$, with tensor product 
 $$(\xi_{s_1}\otimes 
 \cdots \otimes 
 \hat{\xi}_{s_{p_1}} \otimes \cdots\otimes \xi_{s_{r}})\O (\xi_{k_1}\otimes \cdots \otimes 
 \hat{\xi}_{k_{p_2}} \otimes \cdots \otimes \xi_{k_r}).$$
 Pictorially this corresponds to getting a set partition $\tilde{\pi} \in \PS^{\O }_{1,2;\epsilon}(n)$  by adding the most left singletons $1$ to set
$\{s_1,\cdots,s_r\}$ and   $\bar 1$ to the set
$\{k_1,\cdots,k_r\}$ and creating the diagonal pair $${\BL}=({1}, s_{p_1})\O (\bar 1, k_{p_2})\in \MPairPair(\tilde{\pi}),$$ with the same left legs  $1$ and $\bar 1$. 
Note also that 
$$\nawiasl \xi_{\lb_\BL}\O \xi_{\bar \lb_\BL}, \xi_{ \rb_\BL}\O \xi_{\bar \rb_\BL}\nawiasp= \langle \xi_{1},\xi_{s_{p_1}}\rangle_H \langle \xi_{\bar 1},\xi_{k_{p_2}}\rangle_{\bar H}.$$
The diagonal pair $\BL$ crosses the blocks $W_1, \dots, W_{u_1}, \bar W_1, \dots,\bar W_{u_2}$ and so increases the number of crossings by $u_1$ in $\pi|_{[n]}$ and by $u_2$ in $\pi|_{[\bar n]}$,  but decreases the number of inner singletons by $u_1$ and $u_2$. %because originally $s_{p_1}$ was the inner singleton of $W_1,\dots, W_{u_1}$. 
The diagonal pair $\BL$ covers the pairs $U_1,\dots, U_{l_1}, \bar U_1,\dots, \bar U_{l_2}$  so increases the nesting by $l_1$ in $\pi|_{[n]}$ and by $l_2$ in $\pi|_{[\bar n]}$  and create a new  $r-p_1$ right singletons in $\pi|_{[n]}$ and  $r-p_2$ in $\pi|_{[\bar n]}$ .  In the new situation $s_{p_1}$ and $k_{p_2}$ are not singletons, so the number of right singletons of $U_1,\dots,U_{l_1}$ and $\bar U_1,\dots,\bar U_{l_2}$  decreases by $l_1$ and $l_2$, respectively.
Now some new inner singletons $s_{1}, \dots, s_{p_{1}-1}, k_{1}, \dots, k_{p_{2}-1}$  appear. 
Altogether we have:
\begin{align*}  \begin{split}
&\Cr(\tilde{\pi}|_{[n]})=\Cr(\pi|_{{\{2,\dots,n\}}})+u_1 , \\
&\InS(\tilde{\pi}|_{[n]})=\InS(\pi|_{{\{2,\dots,n\}}})-u_1+p_1-1, \\
&\Cov(\tilde{\pi}|_{[n]})=\Cov(\pi|_{{\{2,\dots,n\}}})+l_1,\\ 
&\SL(\tilde{\pi}|_{[n]})=\SL(\pi|_{{\{2,\dots,n\}}})+r-p_1-l_1, \\ 
%&\le(f)\Omega =0,
  \end{split}
  \begin{split}
&\Cr(\tilde{\pi}|_{[\bar n]})=\Cr(\pi|_{{\{\bar 2,\dots,\bar n\}}})+u_2, \\
&\InS(\tilde{\pi}|_{[\bar n]})=\InS(\pi|_{{\{\bar 2,\dots,\bar n\}}})-u_2+p_2-1, \\
&\Cov(\tilde{\pi}|_{[\bar n]})=\Cov(\pi|_{{\{\bar 2,\dots,\bar n\}}})+l_2, \\ 
&\SL(\tilde{\pi}|_{[\bar n]})=\SL(\pi|_{{\{\bar 2,\dots,\bar n\}}})+r-p_2-l_2. \\ 
%&\ri(f)\Omega =0
  \end{split}
\end{align*}
So we see that the exponent of $q$'s,  $\v$'s, $t$'s and $\w$'s   increases by $q^{p_1-1}$, $\v^{p_2-1}$,  $t^{r-p_1}$ and  $\w^{n-p_2}$ respectively; see Figure \ref{fig:FiguraExemple2}.  
\begin{figure}[h]
\begin{center}
\MMatching{10}{1/7, 2/4,3/5,8/10}{10}{1/6, 2/5, 3/7,8/9}
\end{center} 
\caption{The main structure of partition $\tilde{\pi}\in \PS^{\otimes_{ \scriptscriptstyle\S}}_{1,2}(n)$ in the induction step.    }
\label{fig:FiguraExemple2}
\end{figure}
Note that as $\pi$ runs over $\PS^{\O }_{1,2;(\epsilon(2),\dots,\epsilon(n))}(n-1)$, every set partition $\tilde{\pi} \in \PS^{\O }_{1,2;(\epsilon(1),\dots,\epsilon(n))}(n)$ appears exactly once either in   Case 1 or in Case 2 as shown by induction that the formula \eqref{formula101} holds for all $n \in \N$. 
Finally, formula \eqref{eq:multipair} follows from \eqref{formula101} by taking the sum over all $\epsilon$ such that $\Sing(\pi)=\emptyset$ (in this case we understand that  $(\otimes_{i\in \Sing(\pi|_{Z})} \xi_{i})\O (\otimes_{i\in \Sing(\pi|_{\bar{Z}})} \xi_{i})=\Omega\O \bar \Omega $) and applying the state action, because then $\PS^{\O }_2(n)=\P^{\O }_2(n)$ and $\MPairPair(\pi)=\MPair(\pi)$.
\end{proof}

\begin{corollary}\label{cor:euler}
It is interesting to compare the moment of $\G^{(1,1,1,1)}_{\xi\O \eta}$, where $\|\xx\|=1$, with  Euler numbers $E_n$ with positive signs (see moment of hyperbolic secant distribution in  Subsection \ref{subsec:intersingpolynomial}) because it provides a  new  combinatorial interpretation of these  numbers:
\begin{align*} 
E_n=\#\P_{2}^{\O }(2n).  
\end{align*}
These numbers  also occur in combinatorics, specifically when counting the number of alternating permutations of a set with an even number of elements and it is not clear why they are the same. 
\end{corollary}

\subsection{Trace}
Let $\A$ be the von Neumann algebra generated by $\{\G_{\xx}\mid \xx\in \HH_\R\}$ acting on the completion of $\F$ with respect to the inner product $\langle\cdot,\cdot\rangle_{\q}$.

\begin{proposition}  Suppose that $dim(H_\R)\geq 2$, $dim(\bar H_\R)\geq 2$ and $q,\v\in (-1,1)$. Then the vacuum
state is a trace on $\A$ if and only if $q=\v=0$ and $t=\w=1.$
\end{proposition}
\begin{proof} By using Theorem \ref{momentymieszane}, we obtain 
\begin{align*}
&\state( \G_{\xi_1\O\xi_{\bar 1}}\G_{\xi_2\O\xi_{\bar 2}}\G_{\xi_3\O\xi_{\bar 3}}\G_{\xi_4\O\xi_{\bar 4}})=\langle\xi_1, \xi_2\rangle_H
\langle\xi_3, \xi_4\rangle_H\langle\xi_{\bar 1},\xi_{\bar 2}\rangle_{\bar H}\langle\xi_{\bar 3}, \xi_{\bar 4}\rangle_{\bar H}
\\&+qv\langle\xi_1, \xi_3\rangle_H 
\langle\xi_2, \xi_4\rangle_H\langle\xi_{\bar 1},\xi_{\bar 3}\rangle_{\bar H}\langle\xi_{\bar 2}, \xi_{\bar 4}\rangle_{\bar H}
+q\w\langle\xi_1, \xi_3\rangle_H 
\langle\xi_2, \xi_4\rangle_H\langle\xi_{\bar 1},\xi_{\bar 4}\rangle_{\bar H}\langle\xi_{\bar 2}, \xi_{\bar 3}\rangle_{\bar H}\\&+t\v\langle\xi_1, \xi_4\rangle_H 
\langle\xi_2, \xi_3\rangle_H\langle\xi_{\bar 1},\xi_{\bar 3}\rangle_{\bar H}\langle\xi_{\bar 2}, \xi_{\bar 4}\rangle_{\bar H}+t\w\langle\xi_1, \xi_4\rangle_H
\langle\xi_2, \xi_3\rangle_H\langle\xi_{\bar 1},\xi_{\bar 4}\rangle_{\bar H}\langle\xi_{\bar 2}, \xi_{\bar 3}\rangle_{\bar H}
\intertext{and by permuting
}
&\state( \G_{\xi_2\O\xi_{\bar 2}}\G_{\xi_3\O\xi_{\bar 3}}\G_{\xi_4\O\xi_{\bar 4}}\G_{\xi_1\O\xi_{\bar 1}})=\langle\xi_2, \xi_3\rangle_H
\langle\xi_4, \xi_1\rangle_H\langle\xi_{\bar 2},\xi_{\bar 3}\rangle_{\bar H}\langle\xi_{\bar 4}, \xi_{\bar 1}\rangle_{\bar H}
\\&+qv\langle\xi_2, \xi_4\rangle_H 
\langle\xi_3, \xi_1\rangle_H\langle\xi_{\bar 2},\xi_{\bar 4}\rangle_{\bar H}\langle\xi_{\bar 3}, \xi_{\bar 1}\rangle_{\bar H}
+q\w\langle\xi_2, \xi_4\rangle_H 
\langle\xi_3, \xi_1\rangle_H\langle\xi_{\bar 2},\xi_{\bar 1}\rangle_{\bar H}\langle\xi_{\bar 3}, \xi_{\bar 4}\rangle_{\bar H}\\&+t\v\langle\xi_2, \xi_1\rangle_H 
\langle\xi_3, \xi_4\rangle_H\langle\xi_{\bar 2},\xi_{\bar 4}\rangle_{\bar H}\langle\xi_{\bar 3}, \xi_{\bar 1}\rangle_{\bar H}+t\w\langle\xi_2, \xi_1\rangle_H
\langle\xi_3, \xi_4\rangle_H\langle\xi_{\bar 2},\xi_{\bar 1}\rangle_{\bar H}\langle\xi_{\bar 3}, \xi_{\bar 4}\rangle_{\bar H}.
%\intertext{and by permuting
%}
%&\state( \G_{\xi_3\O\xi_{\bar 3}}\G_{\xi_4\O\xi_{\bar 4}}\G_{\xi_1\O\xi_{\bar 1}}\G_{\xi_2\O\xi_{\bar 2}})=\langle\xi_3, \xi_4\rangle_H
%\langle\xi_1, \xi_2\rangle_H\langle\xi_{\bar 3},\xi_{\bar 4}\rangle_{\bar H}\langle\xi_{\bar 1}, \xi_{\bar 2}\rangle_{\bar H}
%\\&+qv\langle\xi_3, \xi_1\rangle_H 
%\langle\xi_4, \xi_2\rangle_H\langle\xi_{\bar 3},\xi_{\bar 1}\rangle_{\bar H}\langle\xi_{\bar 4}, \xi_{\bar 2}\rangle_{\bar H}
%+q\w\langle\xi_3, \xi_1\rangle_H 
%\langle\xi_4, \xi_2\rangle_H\langle\xi_{\bar 3},\xi_{\bar 2}\rangle_{\bar H}\langle\xi_{\bar 4}, \xi_{\bar 1}\rangle_{\bar H}\%\&+t\v\langle\xi_3, \xi_2\rangle_H 
%\langle\xi_4, \xi_1\rangle_H\langle\xi_{\bar 3},\xi_{\bar 1}\rangle_{\bar H}\langle\xi_{\bar 4}, \xi_{\bar 2}\rangle_{\bar H}%+t\w\langle\xi_3, \xi_2\rangle_H
%\langle\xi_4, \xi_1\rangle_H\langle\xi_{\bar 3},\xi_{\bar 2}\rangle_{\bar H}\langle\xi_{\bar 4}, \xi_{\bar 1}\rangle_{\bar H}.
\end{align*}
Since  $dim(H_\R)\geq 2$, there are two orthogonal unit eigenvectors $e_1$, $e_2$ and we put $\xi_1=\xi_2=e_1$ and $\xi_3=\xi_4=e_2$. We also take $\xi_{\bar 1}= \xi_{\bar 2}=\xi_{\bar 3}=\xi_{\bar 4}=\eta\neq 0$ and so the difference is
$$\state( \G_{\xi_1\O\eta_1}\G_{\xi_2\O\eta_2}\G_{\xi_3\O\eta_3}\G_{\xi_4\O\eta_4})-\state( \G_{\xi_2\O\eta_2}\G_{\xi_3\O\eta_3}\G_{\xi_4\O\eta_4}\G_{\xi_1\O\eta_1})=(1-t\v-t\w)\|\eta\|^4.$$
Therefore, the traciality of the vacuum state implies $1-t\v-t\w=0.$ 
Now we  do the same analysis with orthogonal vectors $\xi_{\bar 1}= \xi_{\bar 2}=e_{\bar 1}$ and $\xi_{\bar 3}=\xi_{\bar 4}=e_{\bar 2},$ and obtain $1-t\w=0$, which by the restriction $|t|,|w|\leq 1$  implies that $t=\w=1$. If we substitute this into the first equation, we see that  $\v=0$.  It is clear by symmetry that using the same argument as above, we will show that $\state$ is not a trace when $q\neq 0$. 
When $q=\v=0$ and $t=\w=1$, we get
\begin{align*}
\state( G^{(0,1,0,1)}_{\xi_{1}\O \xi_{\bar 1}}\cdots G^{(0,1,0,1)}_{\xi_{n}\O \xi_{\bar n}})&= \sum_{\pi\in\NC_{2}^{\O }(n)} \prod_{{{\BL \in \MPair(\pi)}}}\nawiasl  \xi_{\lb_\BL}\O \xi_{\bar \lb_\BL}, \xi_{ \rb_\BL}\O \xi_{\bar \rb_\BL}\nawiasp
%\\ &= \sum_{\pi\in\NC_{2}^{\O }(n)} \prod_{{{\BL \in \MPair(\pi)}}} \langle\xi_{\lb_\BL} ,\xi_{ \rb_\BL}  \rangle_H \langle\xi_{\bar \lb_\BL} ,\xi_{\bar \rb_\BL}  \rangle_{\bar H}
\intertext{by Remark \ref{rem:partycje} (2), we have }
&= \sum_{\pi\in\NC_{2}(n)} \prod_{{{(i,j) \in \pi}}} \nawiasl  \xi_{i}\O \xi_{\bar i}, \xi_{ j}\O \xi_{\bar j}\nawiasp
\end{align*}
where the traciality is known because $\nawiasl  \xi_{i}\O \xi_{\bar i}, \xi_{ j}\O \xi_{\bar j}\nawiasp=\nawiasl \xi_{ j}\O \xi_{\bar j}, \xi_{i}\O \xi_{\bar i} \nawiasp$.
%the von Neumann
%algebra becomes the a free deformed von Neumann algebra and the traciality is known. 

\end{proof}
\begin{remark} %During my presentation 
% Prof. M. Bo\.zejko and Prof. F. Lehner said that they did not understand why $\state$ is not a trace when $t=1$ and $\w=1$, because just then the crossing partition plays a role. 
\begin{enumerate}
[(1).]
,\item The vacuum expectation $\state$ is a trace on $\A$ if 
$$(q, t, \v,\w) = (q, 1, 0, 1) \text{ and  }(dim(H_\R)=1\text{ or  }dim(\bar H_\R)=1).$$   
Indeed, if $\xi_{\bar i}=\xi$ and  $\|\xi\|_{\bar H}=1$ then  by Theorem \ref{momentymieszane} we have 
\begin{align*}
\state( \G_{\xi_{1}\O \xi_{\bar 1}}\cdots \G_{\xi_{n}\O \xi_{\bar n}})&=\sum_{\substack{ \pi\in\P_{2}^{\O }(n)\\ \Cr(\pi|_{{[\bar n]}})=0  }}  {q}^{\Cr(\pi|_{{[n]}})}\prod_{{{\BL \in \MPair(\pi)}}}
\nawiasl  \xi_{\lb_\BL}\O \xi_{\bar \lb_\BL}, \xi_{ \rb_\BL}\O \xi_{\bar \rb_\BL}\nawiasp
\intertext{Let $\pi\in\P_{2}(n)$.  There exists a unique non-crossing partition $\hat \pi\in\P_{2}(n)$, such that  the set of right and left legs of the  pairs of $\pi$ and $\hat \pi$ coincide -- see \cite[Page 215]{BozejkoGuta}.  In our situation this means that the set of all partitions 
$\pi \in \P_{2}^{\O }(n)$ such that $\pi|_{{[\bar n]}}$ is noncrosing is isomorphic to $\P_{2}(n)$.
 From this and  $\|\xi\|_{\bar H}=1$  we can deduce}
&=\sum_{\pi\in \P_{2}(n)}  {q}^{\Cr(\pi)}\prod_{{{(i,j) \in \pi}}}
\langle  \xi_{i},\xi_{j}\rangle_H,
\end{align*}
and in this case  the traciality is known -- see \cite{BozejkoSpeicher1994}. Unfortunately, this argument can not be applied under the assumption that $dim(H_\R)\geq 2\text{ and }dim(\bar H_\R)\geq 2$ (the reason is given in the next point). 
\item 
 At first glance the impression may be that there should  a trace, when $t=1$ and $\w=1$, because just then the crossing partition plays a role. 
 The essence of the problem is  that under the cyclic permutation action  of a partition of the form $\P^{\O }(n)$ is not a map to itself; see Figure \ref{fig:FiguraTrace} for specific example.
\begin{figure}[h]
\begin{center}
\Matching{4}{1/3, 2/4}{4}{1/4, 2/3}  \MatchingTrace{4}{1/3, 2/4}{4}{1/2, 3/4}
\end{center} 
\begin{center}
 \begin{tikzpicture}[scale=0.5]
        \node at (0,1) {A diagonal partition };
          \node at (14,1) {A non-diagonal partition};
  \end{tikzpicture}%
\end{center} 
\caption{A cyclic permutation action of $[n]$ and $[\bar n]$   }
\label{fig:FiguraTrace}
\end{figure}
\end{enumerate}
\end{remark}
\section{Poisson-type operators}

\subsection{Gauge operator}
Now  we  define differential second quantization operator on $\FQ$. 
%In order to do this we consider the number operator, the differential second quantization of the identity operator with eigenvalue $[n]_{q,t}[n]_{\v,\w}$.
In order to this, we introduce some special operators. 
Let $T$ and $\bar T$ be the operators on Hilbert spaces $H$ and $\bar{H}$ with dense domains ${D}$ and $\bar{{D}}$, respectively. We also assume  that $T({D}) \subset {D}$ and $\T(\bar{ {D}} )\subset \bar{ {D} }$ and  $\mc{D}:=D \O \bar{ {D}}$. 
The following gauge operator  is motivated by the papers  \cite{Ans01,Ejsmont1}. 
\begin{definition}
The  gauge operator $p_{\TG}$ is an operator on $\FQ$  defined by 
\begin{align}\label{rqT}
p_{\TG}&:= p_{T}^{(q,t)}\O  p^{(\v,\w)}_{\T}
\end{align}
with a dense domain $\FQ$.
\end{definition}
In this part let us recall the properties of the  gauge operator from \cite[Proposition 1.8]{Ejsmont1} 
 (the proof is almost identical  and it can be omitted).
\begin{proposition}\label{Prop:samosprzezone}
If $\TG$ is essentially self-adjoint on a dense domain $\mc{D}$, then $p_{\TG}$ is  essentially self-adjoint on a dense domain $\FQ$.
\end{proposition}
Directly from  Proposition \ref{seqqtgauge}  we can state the following proposition. 
\begin{proposition} If $T$ and $\T$  are bounded operators on $\HH$,  then $p_{\TG}$ is a bounded operator on the $\FQ$.
\label{Ansh+}
\end{proposition}

\subsection{Quadrabasic  operators and cumulants }In  non-commutative setting, random variables are understood to be the elements of the $*$-algebra generated by creator, annihilator or gauge operators. Particularly interesting are their joint mixed moments.
In order to work effectively on this object we need to combine joint moments with corresponding cumulants. This topic in the case of $q$-deformed Fock space was deeply analyzed in the literature; see \cite{Ans01,Ans04,Bia97,Nica96}. Our approach is close to \cite{Ans01,Ejsmont1}. We define $\lambda_1 \O\lambda_{\bar 1}$ to be $\lambda_i \lambda_{\bar i}I\O I$.
We also use a special convention that   $\bar T_{{ i}}:=T_{{\bar i}}$ where $i\in[\bar n]$.
\begin{definition} The operator 
\begin{equation}
\X_{\xi_i,\lambda_i ,T_i}^{\xi_{\bar i},\lambda_{\bar i} ,T_{\bar i}}:= \B_{{\xi_{ i}}\O {\xi_{\bar i}}} +\B^\ast_{{\xi_{ i}}\O {\xi_{\bar i}}}+p_{T_{{ i}}\O T_{ {{\bar i}}}}+\lambda_i \O\lambda_{\bar i},\qquad {\xi_{ i}}\O {\xi_{\bar i}} \in \HH_\R, \quad\lambda_i ,\lambda_{\bar i}\in\R,
\end{equation}
on $\FQ$ is called a  {quadrabasic   operator}. 
\end{definition}

\begin{definition}\label{Defi:Cumulant} Let $\pi\in \P^{\O}(n)$, $\BL=B\O\bar B=\{i_1,\dots,i_m\}\O \{\bar i_1,\dots,\bar i_k\}  \in \MPair(\pi)$, $\lambda_i,\lambda_{\bar i}\in \R$ and 
 $\xi_i\O \xi_{\bar i} \in \HH_\R$
 and  {the diagonal cumulant} is defined by
\begin{align*}
&\Cum \BL) := 
\begin{cases}
 \lambda_{ i_1 }  \lambda_{\bar i_1}& \text{ if $\BL$ is a singleton,} \\
\langle \xi_{i_1}, T_{{i_{2}}} \dots T_{{i_{m-1}}}\xi_{i_m} \rangle_{H} 
\langle \xi_{\bar i_1}, T_{{\bar i_{2}}} \dots T_{{\bar i_{k-1}}}\xi_{\bar i_k} \rangle_{\bar H}  & \text{ otherwise. }
\end{cases}
\\
&R_{\pi}^{\xi,T}:=\prod_{\BL \in \MPair(\pi)}\Cum {\BL}).
\end{align*}
\end{definition}

\noindent
The following theorem is the main result of this section. Its proof is given in Subsection \ref{sec:proofwick}.  
\begin{theorem}\label{thm2} Suppose that $\xi_i\O \xi_{\bar i} \in \HH_\R^n$, then 
\begin{align} 
\state\big( \X_{\xi_1,\lambda_1 ,T_1}^{\xi_{\bar 1},\lambda_{\bar 1} ,T_{\bar 1}}\cdots \X_{\xi_n,\lambda_n ,T_n}^{\xi_{\bar n},\lambda_{\bar n} ,T_{\bar n}}\big)&=\sum_{\pi\in \PB(n)} q^{\rc(\pi|_{[n]})}t^{\nest(\pi|_{[n]})}\v^{\rc(\pi|_{[\bar n]})}\w^{\nest(\pi|_{[\bar n]})}R_{\pi}^{\xi,T}
%\begin{cases}
 %\displaystyle\sum_{\pi\in \PB(n)} q^{\rc(\pi|_{[n]})}t^{nest(\pi|_{[n]})}\v^{\rc(\pi|_{[\bar n]})}\w^{nest(\pi|_{[\bar n]})}%\Cum_{\pi} & \text{ if } \TT=0, \\
%\displaystyle\sum_{\pi\in \PT(n)} q^{\rc(\pi|_{[n]})}t^{\nest(\pi|_{[n]})}\v^{\rc(\pi|_{[\bar n]})}\w^{\nest(\pi|_{[\bar n]})}%\Cum_{\pi} & \text{ if }\T=0,
%\end{cases}
 \label{wick:glowna}
\end{align}
\end{theorem}

\begin{corollary}\label{cor13} 
\begin{enumerate}[\rm(1)]
\item For $t=\w=1$, $ \v=0$, $\xi_{\bar i}=\xi $,   $\|\xi\|=1$ and $\lambda_i=\lambda_{\bar i}=0$,  we obtain the $q$-deformed formula for moments of random variable on $q$-Fock space (see \cite{Ans01} or \cite[Proposition 6]{Ans04}). 
\item For $T_i\O  T_{\bar{i}}=\mathbf{0}$  (which is equivalent to $T_i=\mathbf{0}$ or  $ T_{\bar{i}}=\mathbf{0}$) and $\lambda_i\O\lambda_{\bar i}=0$, we get the formula  \eqref{eq:multipair}.
%\item For $\v=q$ and  ???? \todo{skip this?}
\end{enumerate}
\end{corollary}

\subsection{The orthogonal polynomial }

$(\q)$-Poisson polynomials are defined by the recursion relations
\begin{align}
\label{Charlier}
x \tilde Q_{ n}^{(\q)}( x) &= \tilde Q_{ n+1}^{(\q)}(x) + [n]_{q,t} [n]_{\v,\w} \tilde Q_{ n}^{(\q)}( x) + [n]_{q,t} [n]_{\v,\w}\tilde Q_{n-1}^{(\q)}( x), \qquad n\geq 1
\end{align}
with initial conditions $\tilde Q_{-1}^{(\q)}(x) = 0$, $\tilde Q_{0}^{(\q)}( x) = 1$ and $\tilde Q_{1}^{(\q)}( x) = x$. There exists a probability measure $\tilde{\mu}_{\q}$ which is associated to the orthogonal polynomials $\tilde Q_{ n}^{(\q)}$. %The distribution of the operator of type B coincides with $\mu_{\alpha,q}$. 
\begin{remark} The measure of orthogonality of the above polynomial sequence is not known. In special cases, we can identify this measure:
\begin{enumerate}[\rm(1)]
\item the measure $\tilde{\mu}_{1,1,0,0}$ is the classical Poisson law; 
\item the measure $\tilde{\mu}_{0,1,0,1}$ is the Marchenko-Pastur distribution; 
\item the measure $\tilde{\mu}_{q,1,0,1}$ is the $q$-Poisson law and the orthogonal polynomials $\tilde Q_n^{(q,1,0,1)}(x)$ are called \emph{$q$-Poisson-Charlier polynomials} (see \cite{Ans01}); 
\end{enumerate}
\end{remark}
Using the same argument as in Theorem \ref{rozkladgaussa}, we can prove the following. 
\begin{proposition} Let $\xx \in \HH_\R$ and $ \|\xx\|=1$ and $T=\bar T =\id$. Then the probability distribution of $\X_{\xi,0 ,\id}^{\eta,0 ,\id} $ with respect to the vacuum state is given by $\tilde{\mu}_{\q}$.
\label{WielomianyTypeB}
\end{proposition}

\subsection{Proof of  Theorem \ref{thm2}}
\label{sec:proofwick}

We begin with some special notations.

In order to prove Theorem \ref{thm2} we need the set $\P^{\O}_E(n)$ of so-called extended partitions.  Here some blocks  can be additionally marked by $\primes$ and so  we consider additional blocks  denoted by 
$\{i_1,\dots,i_m\}_\primes $ and $ \{\bar i_1,\dots,\bar i_k\}_\primes.
$
\begin{definition}\label{Partycjerozszerzone}
%\begin{enumerate}[I]
%\item  for a given block of 
%partition $\pi$ of $[n]$ or $[\bar n]$ every element   is an opener, a closer, a middle point or a singleton for this
%block (each block of $\pi$  is associated with $[n]$ or $[\bar{n}]$ -- no connection between them); 
%\item the collections of openers and closers points 
%of arcs of $\pi_{[ n]}$ and $\pi_{[\bar n]}$ coincide. 
%\end{enumerate}
We denote by $\P^{\O }_E(n)$ the set partition of $[n]\sqcup [\bar{n}]$   such that each block of $\pi$  is associated with $[n]$ or $[\bar{n}]$  and the collections of openers points 
of arcs of $\pi_{[ n]}$ and $\pi_{[\bar n]}$ coincide. 
Additionally, every block of  partitions of $\pi$ is regular  or expanded.
All   singletons are denoted as extended.
If for the block  $B\in \pi_{[ n]}$ exist conjugate block $\bar B\in \pi_{[\bar n]}$, then both of them must be denoted as regular or extended as $$(a,\dots,b),(\bar a,\dots,\bar c) \text{ or }(a,\dots,b)_E,(\bar a,\dots,\bar c)_E.$$ All other  blocks are denoted as extended.

 % satisfy condition   II  form Definition \ref{def:partycji} i.e.  start form the same point and have the size at least two, then both of them must be denoted as regular or extended as $$(a,\dots,b),(\bar a,\dots,\bar c) \text{ or }(a,\dots,b)_E,(\bar a,\dots,\bar c)_E.$$ All other  blocks are denoted as extended.

%We denote by $\P^{\O }_E(n)$ the set partition of $[n]\sqcup [\bar{n}]$  such that   conditions  I and III from Definition \ref{def:partycji} are satisfied. Additionally, every block of  partitions of $\pi$ is regular  or expanded.
%If block $B$ and $\bar B$   satisfy condition   II  form Definition \ref{def:partycji} i.e.  start form the same point and have the size at least two, then both of them must be denoted as regular or extended as $$(a,\dots,b),(\bar a,\dots,\bar c) \text{ or }(a,\dots,b)_E,(\bar a,\dots,\bar c)_E.$$
%
\end{definition}
\begin{example}\label{ex:extendedpart} For example 
\begin{align*}
\pi &= \{(1,4)_\primes,(2)_E,(3)_E,(5)_\primes\}\sqcup \{(\bar 1,\overline {5})_\primes ,(\bar 2)_\primes,(\bar 3)_\primes,(\bar 4)_\primes\}\in \P^{\O }_E(5)
\\\pi &= \{(1,4),(2)_E,(3)_E,(5)_\primes\}\sqcup \{(\bar 1,\overline {5}) ,(\bar 2)_\primes,(\bar 3)_\primes,(\bar 4)_\primes\}\in \P^{\O }_E(5)
\\\pi &= \{(1,4,5),(2)_E,(3)_E\}\sqcup \{(\bar 1,\overline {3}) ,(\bar 2)_\primes,(\bar 4, \bar 5)_\primes\}\in \P^{\O }_E(5)
\\\pi &= \{(1,4,5),(2)_E,(3)_E\}\sqcup \{(\bar 1,\overline {3}) ,(\bar 2)_\primes,(\bar 4, \bar 5)\}\notin \P^{\O }_E(5)
\end{align*}
%\begin{align*}
%\pi &= \{(1,4,6,7)_\primes,(2)_E,(3,5)_\primes,(9)_\primes,(8,10)\}\sqcup \{(\bar 1,\bar 3, \bar 4,\bar 6,\overline {10})_\primes ,(\bar 2)_\primes,(\bar 5)_\primes,(\bar 7)_\primes,(\bar 8,\bar 9)\}\in \P^{\O }_E(10)
%\\ \pi &= \{(1,4,6,7),(2)_E,(3,5)_\primes,(9)_\primes,(8,10)_\primes\}\sqcup \{(\bar 1,\bar 3, \bar 4,\bar 6,\overline {10}) ,(\bar 2)_\primes,(\bar 5)_\primes,(\bar 7)_\primes,(\bar 8,\bar 9)_\primes\}\in \P^{\O }_E(10)
%\\ \pi &= \{(1,4,6,7),(2)_E,(3,5),(9)_\primes,(8,10)_\primes\}\sqcup \{(\bar 1,\bar 3, \bar 4,\bar 6,\overline {10}) ,(\bar 2)_\primes,(\bar 5)_\primes,(\bar 7)_\primes,(\bar 8,\bar 9)_\primes\}\notin \P^{\O }_E(10)
%\end{align*}
\end{example}
\noindent 
For $\pi \in \P^{\O}_E(n) $ we denote by $Block_\primes(\pi)$ the tensor blocks of $\pi$ which are marked by $\primes$ and $Block(\pi)=\pi\setminus Block_\primes(\pi).$ Thus we can decompose an extended partition  as a disjoint subset $$\pi=Block_\primes(\pi)\cup Block(\pi).$$
\begin{remark}
  Note that if  $\P^{\O}_E(n)$ consists of  pairs and singletons, then it is not the same as $\PS_{1,2}^{\O}(n)$.
% All singletons are denoted as extended. 
The objects $\PS_{1,2}^{\O}(n)$  are the partitions
from $\P^{\O}_E(n)$ consisting of pairs and singletons, such that all pairs are regular, and all singletons are
expanded.
%\begin{enumerate}[(1).]
%\item If blocks $B$ and $\bar B$   satisfy condition   II  from Definition \ref{def:partycji}, then both of them can be marked as extended.
%\item   Note that if  $\P^{\O}_E(n)$ consists of a pair or a singleton, then it is not the same as $\PS_{1,2}^{\O}(n)$.
%\end{enumerate} 
\end{remark}

\noindent \textbf{\emph{Cover and left of max}}. We also
need to extend the definition of  $\InS(\pi|_Z)$ and $\SL(\pi|_Z)$ for $\pi\in \P^\O_E(n)$ where $Z$ is $[n]$ or $[\bar n]$, i.e. we define
\begin{align*}
&\InS(\pi|_Z):=\#\big\{ (V,W)\in Block_\primes(\pi|_Z)\times \Semi(\pi|_Z) \mid 
 i<\min V <j \textrm{ for } i,j\in W\big\},
\\&
\SL(\pi|_Z):=\#\{ (V,W)\in Block_\primes(\pi|_Z)\times \Semi(\pi|_Z) \mid \max V > j \text{~for all~} j\in W \}.
\end{align*}
\begin{remark} Note that  $\InS(\pi)$ represents the number of covered singletons and $\SL(\pi)$  the number of singletons to the right of  arcs, whenever all extended block are singletons  (which is the reason we use the same notation). 
\end{remark}

\noindent
In order to simplify  notation, we define the following operators, which map $H$ $(\bar H)$ into $H$ $(\bar H)$ and which are indexed by the block $B_E=\{i_1,\dots,i_m\}_\primes\in Block_\primes (\pi|_{[n]})$ i.e. $\OptTTyl_{B_E}=T_{{i_1}} \dots T_{{i_{m-2}}}T_{{i_{m-1}}}$ and for $B=\{i_1,\dots,i_m\}\in Block(\pi|_{[n]})$ we denote $\OptT_{B}=T_{{i_2}} \dots T_{{i_{m-2}}}T_{{i_{m-1}}}$. We use the same notation for $[\bar n]$, i.e. $\OptT_{\bar B_E} $  
or 
$\OptTTyl_{\bar B_E}$. 
With the notation above we also introduce: 
\begin{align*}
\begin{split}
&\Cumm=\prod_{\substack{B \in \pi|_{[n]}%\\ B\in Block(\pi|_{[n]})
}}
 \langle x_{\min B} ,\OptT_{B}\xi_{\max B}\rangle_H \prod_{\substack{\bar B \in \pi|_{[\bar n]}}} \langle x_{\min \bar B}, \OptT_{\bar B}\xi_{\max \bar B}\rangle_{\bar H}, \\ 
&\CummTensor=\Big[\bigotimes_{\substack{ B_E \in \pi|_{[n]}}}\{ \OptTTyl_{B_E} \xi_{\max B_\primes}\}_{\min B_\primes}\Big]\O\Big[\bigotimes_{\substack{\bar B_E \in \pi|_{[\bar n]}}}\{ \OptTTyl_{\bar B_E} \xi_{\max \bar B_\primes}\}_{\min \bar B_\primes}\Big] .
\end{split}
\end{align*}
Notice that in the above formula we use the following bracket notation $\{\cdot\}_{\min B_E}$, which should be understood that the position of $\cdot$ (in the tensor product) is ordered with respect to the $\min B_E$. 
\begin{example} For the partition $$\pi = \{(1,4,6,7)_\primes,(2)_E,(3,5)_\primes,(9)_\primes,(8,10)\}\sqcup \{(\bar 1,\bar 3, \bar 4,\bar 6,\overline {10})_\primes ,(\bar 2)_\primes,(\bar 5)_\primes,(\bar 7)_\primes,(\bar 8,\bar 9)\},$$ 
 we have 
\begin{align*}
\Cumm&=  \langle x_{8} ,\xi_{10}\rangle_H \langle x_{\bar 8} ,\xi_{\bar 9}\rangle_{\bar H}\\
\CummTensor&=\big[T_{1}T_{4}T_{6}\xi_7\otimes \xi_2\otimes T_{3}\xi_5 \otimes \xi_9\big]\O\big[ T_{{\bar 1}} T_{{\bar 3}} T_{{\bar 4}} T_{{\bar 6}}\xi_{\overline {10}}\otimes \xi_{\bar 2}\otimes \xi_{\bar 5} \otimes \xi_{\bar 7}\big] .
\end{align*}
\end{example}
We also use the following convention for $\epsilon\in\{1,\ast,\primes\}$
\[
\B^{\epsilon}_{\xi_i\O\xi_{\bar i}} = 
\begin{cases}
\B^{*}_{\xi_i\O\xi_{\bar i}} & \text{ if } \epsilon=\ast, \\
\B_{\xi_i\O\xi_{\bar i}} & \text{ if } \epsilon=1,
\\
p_{T_{i}\O T_{{\bar i}}} & \text{ if } \epsilon=\primes.
\end{cases}
\]
The main idea of the proof is  similar to that of Theorem
\ref{momentymieszane} so, for brevity, we will leave out some of the combinatorial details.
\begin{proof}[Proof of  Theorem \ref{thm2} ]
 Observe that when $n=1$, then $\B_{\xi_{1}\O \xi_{\bar 1}}\Omega\O \bar \Omega=p_{T_{_1}\O T_{{\bar 1}}}\Omega\O \bar \Omega=0$ and $\B^\ast_{\xi_{1}\O \xi_{\bar 1}}\Omega\O \bar \Omega={\xi_{1}\O \xi_{\bar 1}}$.  
Suppose that  $\xi_{i}\O \xi_{\bar i} \in  \HH_{\R}$, $i\in\{2,\dots,n\}$
and any $\epsilon=(\epsilon(2), \dots, \epsilon(n))\in\{1,\ast,\primes\}^n$,  we have
\begin{align}
\begin{split}
\B^{\epsilon(2)}_{\xi_{2}\O \xi_{\bar 2}}\cdots \B^{\epsilon(n)}_{\xi_{n}\O \xi_{\bar n}}\Omega \O\bar \Omega = \sum_{\pi\in\PB_{E;\epsilon}(\{2,\dots,n\})}&  q^{\rc(\pi|_{[n]})+\InS((\pi|_{[n]}))}t^{\nest(\pi|_{[n]})+\SL(\pi|_{[n]})}\\&\v^{\rc(\pi|_{[\bar n]})+\InS(\pi|_{[\bar n]})}\w^{\nest(\pi|_{[\bar n]})+\SL(\pi|_{[\bar n]})} \Cumm\CummTensor.
\end{split}
\end{align}
We will show that the action of $\B^{\epsilon(1)}_{\xi_{1}\O \xi_{\bar 1}}$ corresponds to the inductive graphic description of set tensor partitions.
We fix $\pi\in \PB_{E;\epsilon}(\{2,\dots,n\})$ and run the argument below over all partitions of this type. Suppose that 
\begin{itemize}
 \item $\pi|_{{\{2,\dots,n\}}}$ has blocks in $ Block_\primes(\pi|_{\{2,\dots,n\}})$ on the positions $s_{1}<\cdots <s_{p_1} <  \cdots <s_r$ and arcs $W_{1},\dots, W_{u_1}$ which cover $s_{p_1}$, arcs $U_1,\dots, U_{l_1}$ to the left of $s_{p_1}$,
\item $\pi|_{{\{\bar 2,\dots,\bar n\}}}$ has blocks in $ Block_\primes(\pi|_{\{\bar 2,\dots,\bar n\}})$ on the positions $k_1<\cdots <k_{p_2} <  \cdots <k_r$, arcs
 $\bar W_{1},\dots,\bar   W_{u_2}$ which cover $k_{p_2}$ and arcs $\bar U_1,\dots, \bar U_{l_2}$ to the left of $k_{p_2}$. 
\end{itemize}
Suppose that a partition $\pi$ has blocks $ \BC\in Block_\primes(\pi)$ on the $\ith^{\rm th}$ position, i.e. $\ith=(\min S_\primes,\min \bar K_\primes)$. In this case blocks $\BC$ have the following contribution to $ \CummTensor$: 
\begin{align*}
  \{\cdot\}_{s_1}\otimes\dots\otimes \{\OptTTyl_{S_E}\xi_{\max S_\primes}\}_{s_{p_1}}\otimes \dots\otimes\{\cdot\}_{s_r}\O\{\cdot\}_{k_1}\otimes\dots\otimes \{\OptTTyl_{\bar K_E}\xi_{\max\bar K_\primes}\}_{k_{p_2}}\otimes \dots\otimes\{\cdot\}_{k_r}
\end{align*}
\noindent Case 1. If $\epsilon(1)=\ast$, then the operator $\B^\ast_{\xi_{1}\O \xi_{\bar 1}}$ acts on the tensor product, putting $\xi_{1}\O \xi_{\bar 1}$ by adding (expanded) singleton on the left as in Case 1 of the proof of Theorem  \ref{momentymieszane}. 

Case 2. If $\epsilon(1)=1$, then $\B_{\xi_{1}\O \xi_{\bar 1}}$ acts on the tensor product, then new $r^2$ terms appear. In terms $s_{p_1}$ and $s_{p_2}$ the inner product 
\begin{align}\label{eq:case2proofcum}
\langle \xi_{1}, \OptTTyl_{S_E}\xi_{\max S_E} \rangle_{H} 
\langle \xi_{\bar 1 }, \OptTTyl_{\bar K_E}\xi_{\max \bar K_E} \rangle_{\bar H}
\end{align}
appears with coefficient $q^{p_1-1}\v^{p_2-1} t^{r-p_1}\w^{r-p_2}$. Graphically this corresponds to getting a set partition $\tilde{\pi} \in \PB_{E;\epsilon}(n)$ by adding $1$ and $\bar 1$ to $\pi$ and creating a new regular block 
$(1,S)$ and $(\bar 1,\bar K)$  by adding first arcs $(1,s_{p_1})$ to $S_\primes$  and the second  $(\bar 1,k_{p_2})$ to $\bar K_\primes$.
 We see that  Equation \eqref{eq:case2proofcum} can be written in the form:   
$$\langle x_{\min (1,S)} ,\OptT_{(1,S)}\xi_{\max (1,S)}\rangle_H  \langle x_{\min (\bar 1,\bar K) }, \OptT_{(\bar 1,\bar K)}\xi_{\max (\bar 1,\bar K)}\rangle_{\bar H}.$$
We can calculate the change in the statistic generated by the new arcs in the same way as in  Case 2  of the proof of Theorem  \ref{momentymieszane}. Indeed it suffices to repeat all steps of counting changes with   arcs instead of pairs, and $\rc$, $ \nest$ in place of $cr$, $\Cov$.  In this procedure  we can think that extended blocks are singletons.

Case 3. If $\epsilon(1)=\primes$, then we use the equation \eqref{rqT}, delete the element $$\OptTTyl_{S_E}\xi_{\max S_\primes}\O \OptTTyl_{\bar K_E}\xi_{\max \bar K_\primes}\text{ from }\CummTensor,$$   
 and then a new component $\widehat{\mathrm{ K}}^\mb{\xi}_{\tilde\pi}$  appears in the tensor product with coefficient $q^{p_1-1}\v^{p_2-1} t^{r-p_1}\w^{r-p_2}$ in the first position as shown in  Figure \ref{fig:FiguraExemple4}.
 \begin{figure}[h]
\begin{center}
  \begin{tikzpicture}[thick,font=\small,scale=.9]
     \path 
       %    (2.5,-0.7) node[] (n) {$\{s_{l-1}+1,\dots,s_l\}_j$ }   
            %(2.5,-1.5) node[] (nn) {Sit. 1 $(a)$ }       
           %(1,1.7) node[] (a) {$(a)$}
           (-1,0) node[] (b) {\tiny{$\widehat{\mathrm{ K}}^\mb{\xi}_{\tilde\pi}=$}}
           (1,0) node[] (bc) {\tiny{$T_{1}(\OptTTyl_{S_E}\xi_{\max S_\primes})\otimes$}}
           (1,0.4) node[] (pom1) {$\blacktriangledown$}
          %  (1,0) node[] (pom2) {}
           % (2.5,0) node[] (bcc) {$s_{l-1}+2$}
           %(2.7,0) node[] (c) {$\dots$}
           (2.68,0) node[] (d) {\tiny{$...$}}
           %(4.5,0) node[] (f) {$\dots$}
           (4.5,0) node[] (g) {\tiny{$\otimes\{\xcancel{\OptTTyl_{S_E}\xi_{\max S_\primes}}\}_{s_{p_1}}\otimes$}}
           (6.3,0) node[] (gg) {\tiny{$...$}}
           (7,0) node[] (gg) {\tiny{$\otimes\{\cdot\}_{s_r}$}}
           (8,0) node[] (gg) {$\O$};
           %(5,-0.3) node[] (k) {$\dots$}
          % (2.5,1.2) node[] (kk) {$(a)$}
          % (2.5,1.1) node[] (kk) {$\sigma_1\dots \sigma_{i_j}$};
            %(2.5,1.3) node (fdf) {$U_j \sim \pm1$}
          %  (7,1) node (fdf) {$\xmapsto{\mu_{\ast,.}(i(z_n),i(w_j))}$ };
            %(4.5,1) node (fdf) {$ \pm1$}
           %(6,0) node[] (h) {$\longrightarrow$}
     %\draw (a) -- +(0,0.75) -| (d);
    % \draw[thick,dashed] (b) -- +(0,1.1) -| (e);
     \draw[thick] (bc) -- +(0,0.8) -| (g);
       %\draw[line width=.1cm] (pom1) -- +(0,-0.3) -| (pom2);
    % \draw[thick] (d) -- +(0,0.5) -| (f);
    % \draw [thick](bc) -- +(0,-0.3) -| (g);
   \end{tikzpicture}
   %\hspace{1cm}
   \begin{tikzpicture}[thick,font=\small,scale=.9]
     \path 
       %    (2.5,-0.7) node[] (n) {$\{s_{l-1}+1,\dots,s_l\}_j$ }   
            %(2.5,-1.5) node[] (nn) {Sit. 1 $(a)$ }       
           %(1,1.7) node[] (a) {$(a)$}
        %   (-0.4,0) node[] (b) {$s_{l-1}$}
           (1,0) node[] (bc) {\tiny{$ T_{{\bar 1}}(\OptTTyl_{\bar K_E}\xi_{\max \bar K_\primes})\otimes$}}
           (1,0.4) node[] (pom1) {$\blacktriangledown$}
          %  (1,0) node[] (pom2) {}
           % (2.5,0) node[] (bcc) {$s_{l-1}+2$}
           %(2.7,0) node[] (c) {$\dots$}
           (2.68,0) node[] (d) {\tiny{$...$}}
           %(4.5,0) node[] (f) {$\dots$}
           (4.5,0) node[] (g) {\tiny{$\otimes \{\xcancel{\OptTTyl_{\bar K_E}\xi_{\max\bar K_\primes}}\}_{k_{p_2}}\otimes$}}
      (6.3,0) node[] (gg) {\tiny{$...$}}
           (7,0) node[] (gg) {\tiny{$\otimes\{\cdot\}_{k_r}$}};
           %(2.5,1.1) node[] (kk) {$\sigma_1\dots \sigma_{i_j}$};
            %(2.5,1.3) node (fdf) {$U_j \sim \pm1$}
          %  (7,1) node (fdf) {$\xmapsto{\mu_{\ast,.}(i(z_n),i(w_j))}$ };
            %(4.5,1) node (fdf) {$ \pm1$}
           %(6,0) node[] (h) {$\longrightarrow$}
     %\draw (a) -- +(0,0.75) -| (d);
    % \draw[thick,dashed] (b) -- +(0,1.1) -| (e);
     \draw[thick] (bc) -- +(0,0.8) -| (g);
       %\draw[line width=.1cm] (pom1) -- +(0,-0.3) -| (pom2);
    % \draw[thick] (d) -- +(0,0.5) -| (f);
    % \draw [thick](bc) -- +(0,-0.3) -| (g);
   \end{tikzpicture}
   \end{center}  
 \caption{The visualization of the action $p_{T_{1}\otimes_{ \scriptscriptstyle\S} T_{{\bar 1}}}$ on the tensor product $\CummTensor$. }
\label{fig:FiguraExemple4}
\end{figure}

Then we get a new partition $\tilde{\pi} \in \PB_{E;\epsilon}(n)$ by adding $1$ to $S_\primes$, $\bar 1$ to $\bar K_\primes$ (with the first arc $(1,s_{p_1})$  and $(\bar 1,k_{p_2})$)  and creating two blocks in $ Block_\primes(\tilde{\pi})$. Now the minimums of newly created blocks are $1$ and $\bar 1$ and  so we can calculate the change in the statistic generated by the new arc in Case 2, because new arcs cannot be covered or be to the right of some  arc. 
This situation is also compatible with changes inside the tensor product, i.e.
$$\OptTTyl_{(1,S)_E}\xi_{\max S_\primes}=T_{{1}}(\Tens_{S_E} \xi_{\max S_E})\text{ and }\OptTTyl_{(\bar 1,\bar K)_E}\xi_{\max \bar K_\primes}=T_{{\bar 1}}(\Tens_{\bar K_E} \xi_{\max \bar K_E})$$
%Note that as $\pi$ runs over $\pi\in \PB_{E;\epsilon}(\{2,\dots,n\})$ , every set partition $\tilde{\pi}\in \PB_{E;\epsilon}(n)$  appears exactly once either in   Case 1 or in Case 2 as shown by induction that the formula \eqref{formula101} holds for all $n \in \N$. 
We now present the final step. First, let us notice that for $\epsilon\in\{1,\ast,\primes\}^n$ we have
\begin{align} \label{wick:prawieglowna}
&\state\big(\B^{\epsilon(1)}_{\xi_{1}\O \xi_{\bar 1}} \cdots \B^{\epsilon(n)}_{\xi_{n}\O \xi_{\bar n}} \big)=\sum_{\pi \in \PB_{\geq 2,\epsilon}(n)} q^{\rc(\pi|_{[n]})}t^{\nest(\pi|_{[n]})}\v^{\rc(\pi|_{[\bar n]})}\w^{\nest(\pi|_{[\bar n]})}
\Cumm\end{align}
Indeed, from equation \eqref{formula101} we see that the following condition must hold: $\CummTensor=\Omega\O \bar \Omega$. This will happen if and only if $Block_\primes(\pi) =\emptyset,$ which implies \eqref{wick:prawieglowna}. If  $Block_\primes(\pi)=\emptyset$, then $R_{\pi}^{\xi,T}=\Cumm$  so by taking the sum over all $\epsilon$ from equation \eqref{wick:prawieglowna}, we see that
\begin{align*}
\state\Big( \big(\X_{\xi_1,\lambda_1 ,T_1}^{\xi_{\bar 1},\lambda_{\bar 1} ,T_{\bar 1}}-\lambda_1 \O \lambda_{\bar 1}\big) \cdots \big(\X_{\xi_n,\lambda_n ,T_n}^{\xi_{\bar n},\lambda_{\bar n} ,T_{\bar n}}-\lambda_n\O\lambda_{\bar n} \big)\Big)=\displaystyle\sum_{\pi \in \PB_{\geq 2}(n)} q^{\rc(\pi|_{[n]})}t^{\nest(\pi|_{[n]})}\v^{\rc(\pi|_{[\bar n]})}\w^{\nest(\pi|_{[\bar n]})}R_{\pi}^{\xi,T}.
\end{align*}
We also see that
\begin{align*} 
&\state\Big( \big(\X_{\xi_1,\lambda_1 ,T_1}^{\xi_{\bar 1},\lambda_{\bar 1} ,T_{\bar 1}}-\lambda_1\O\lambda_{\bar 1} +\lambda_1\O\lambda_{\bar 1} \big) \cdots \big(\X_{\xi_n,\lambda_n ,T_n}^{\xi_{\bar n},\lambda_{\bar n} ,T_{\bar n}}-\lambda_n\O \lambda_{\bar n}+\lambda_n\O \lambda_{\bar n}\big)\Big)
\intertext{by equation \eqref{wick:prawieglowna}, we get}
&=\sum_{\nu \subset \{1,\dots,n\}} \Bigg[\prod_{i\in\nu} \lambda_i\lambda_{\bar i} \sum_{\pi \in \PB_{\geq 2}([n]\setminus \nu)} q^{\rc(\pi|_{[n]})}t^{\nest(\pi|_{[n]})}\v^{\rc(\pi|_{[\bar n]})}\w^{\nest(\pi|_{[\bar n]})}R_{\pi}^{\xi,T}\Bigg]. 
\end{align*} 

\end{proof}
\section{Application to the L\'{e}vy process}
\label{Subsec:Processes}
The main goal of this section is to investigate a new class of noncommutative  L\'{e}vy processes.  
To make it clear, we use the following Anshelevich \cite{Ans01} notation.

 Here $\chf{I}$ is the indicator function of the set $I$, considered both as a vector in $L^2(\mf{R}_+)$ and a multiplication operator on it. 
Let $K$ be a Hilbert space, and let $H$ be the Hilbert space $L^2(\mf{R}_+, d x) \otimes K$. Let $\xi \in K$, and let $T$ be an essentially self-adjoint operator on a dense domain $D \subset K$ so that $D$ is equal to the linear span of $\set{T^n \xi}_{n=0}^\infty$; moreover $\xi$ is an analytic vector for $T$.
Let $\HH=H\O\bar H$, where we assume that $\bar H$ is a one-dimensional Hilbert space spanned by such $\eta$  that $\|\eta\|=1$,  $\bar{T}=\id$ and $\mc{D}=D\O \bar H$. 
Given a half-open interval $I \subset \mf{R}_+$ denote  $p_I(T) = p_{ (\chf{I}\otimes T)\O  \id}$. For $\lambda \in \mf{R}$ and $(\chf{I}\otimes \xi)\O \eta \in \HH_\R$ we define 
\begin{align*}
p_I(\xi\O \eta , T , \lambda) := \B_{(\chf{I}\otimes \xi)\O \eta} + \B^\ast_{(\chf{I}\otimes \xi)\O \eta} + p_I(T) + \abs{I} \lambda \O 1
.\end{align*}
  We will call a process of the form $I \mapsto p_I(\xi\O\eta, T, \lambda) $ a \emph{quadrabasic L\'{e}vy process or $(q,t,\v,\w)$-L\'{e}vy process}. 
\begin{remark} 
\begin{enumerate}
[(1).]

\item  For $t=\w=1,\v=0$ this is indeed a $q$-L\'{e}vy process in a sense of Anshelevich \cite{Ans01,Ans04}.

\item  We assume $\bar{T}=\id$  for several reasons. Firstly,  all theorems below are not true for general $\bar{T}$. Secondly, cumulants in a general sense are not   conditionally positive in a sense of Hamburger moment problem for the one-parameter moment-problems but maybe this analysis can can be considered in the two-parameter case in the context of papers \cite{Devinatz57,Putinar99,stochel98}.

 \end{enumerate}
\end{remark} 
 
\begin{definition}
Denote by $\mf{C} \langle \mb{x} \rangle = \mf{C} \langle x_1, x_2 , \ldots, x_k\rangle$ the algebra of polynomials in $k$ formal noncommuting variables with complex coefficients. 
%Note that in a more abstract language, this is just the tensor algebra of the complex vector space $V$ with a distinguished basis $\set{x_i}_{i=1}^k$.
We denote by $\delta_0(f)$ the constant term of $f \in \mf{C} \langle \mb{x} \rangle$. 
 While we take $V$ to be $k$-dimensional, the same arguments will work for an arbitrary $V$, as long as we use a more functional definition of a process, namely for $f = \sum a_i x_i \in V$, we would define $T (f) = \sum a_i T_{i} , \xi(f) = \sum a_i \xi_i, \lambda(f) = \sum a_i \lambda_i $.  
We define a process 
$$\mb{X}^{\underline{u}(i)}(I_{\underline{v}(i)}):=p_{I_{\underline{v}(i)}}(\xi_{\underline{u}(i)}\O \eta, T_{\underline{u}(i)}, \lambda_{\underline{u}(i)}) \text{ for a multi-
indices $\underline{v}$ and $\underline{u}$.}
%\text{ and } \mb{X}^{(\underline{u})}_{\underline{u}}= \mb{X}^{\underline{u}(1)}(I_{\underline{v}(1)})\dots \mb{X}^{\underline{u}(n)}(I_{\underline{v}(n)})  
$$
 Denote by  $ \mb X(\I)$ the appropriate objects corresponding to the interval $[0, \I)$.
 We define the  functional $M$ on $\mf{C} \langle \mb{x} \rangle$ by the following action on monomials: $M(1, \I; \mb{X}) = 1$ and 
\[
M(\mb{x}_{\underline{u}}, \I; \mb{X}) =
% \state{\big(\mb{X}^{(\underline{u})}(\I)\big)}=
\state{\big( \mb{X}^{\underline{u}(1)}(\I)\dots \mb{X}^{\underline{u}(n)}(\I)\big)},
\]
and extend linearly. 
We will call $M(\cdot, \I; \mb{X})$ the \emph{moment functional} of the process $\mb{X}$ at time $\I$.
If we equip $\mf{C} \langle \mb{x} \rangle$ with a conjugation $\ast$ extending the conjugation on $\mf{C}$ so that each $x_i^\ast = x_i$, it is clear that $M$ is a positive functional, i.e.\ $M(f f^\ast, \I; \mb{X}) \geq 0$ for all $f \in \mf{C} \langle \mb{x} \rangle$. 
\end{definition}

In order to keep the essentially self-adjoint operators, we should make an additional assumption on the family of operators $\set{T_j}_{j=1}^k$ (see \cite [Subsection 2.5]{Ans01}). Moreover, we emphasize that, since the second operator $\Bar T$ is  bounded and self-adjoint no additional restriction on it are necessary. 

\begin{assumption}\label{assumption:analitic}
Now fix a $k$-tuple $\set{T_j}_{j=1}^k$ of essentially self-adjoint operators on a common dense domain ${D} \subset V$, $T_j({D}) \subset {D}$, a $k$-tuple $\set{\xi_j}_{j=1}^k \subset {D}$ of vectors, and $\set{\lambda_j}_{j=1}^k \subset \mf{R}$. We will make an extra assumption that 
\begin{equation}
\label{Uniqueness}
\begin{split}
&\forall i, j \in [1 \ldots k], l \in \mf{N}, \underline{u} \in [1 \ldots k]^l, \\
&\mb{T}_{\underline{u}} \xi_i = T_{\underline{u}(1)} T_{\underline{u}(2)} \ldots T_{\underline{u}(l)} \xi_i \text{ is an analytic vector for } T_j, \\
&\text{ and } {D} = \ls{\set{\mb{T}_{\underline{u}} \xi_i: i \in [1 \ldots k], l \in \mf{N}, \underline{u} \in [1 \ldots k]^l}}.
\end{split}
\end{equation}
\end{assumption}

Now we define the joint cumulants of $\mb{X}$. 
\begin{definition}\label{def:cumulantyLevego}
The \emph{ cumulant} corresponding to the partition $\pi \in \P^{\O }(n)$, the block $B=(i_1,\dots,i_k)\in\pi|_{[\bar n]}$ and the sub-monomial $\mb{x}_{(B: \underline{u})}:=x_{\underline{u}(i_1)}\dots x_{\underline{u}(i_k)}$ is
\begin{align*}
& R(\mb{x}_{(B: \underline{u})}, \I):  = 
\begin{cases}
\I \lambda_{\underline{u} (i_1)}& \text{ if } k=1, \\
\I\langle \xi_{\underline{u}{(i_1)}}, T_{{\underline{u}{(i_2)}}} \dots T_{{\underline{u}{(i_{k-1})}}}\xi_{\underline{u}{(i_k)}} \rangle_{H} 
 & \text{ if } k \geq 2.
\end{cases}
\\& R_{\pi|_{[ n]}}(\mb{x}_{\underline{u}},\I; \mb{X}): = \prod_{B\in\pi|_{[n]}}R(\mb{x}_{(B: \underline{u})}, \I)
\end{align*}
Sometimes for a one-dimensional  process  we will write 
%$R\big( \mb{X}^{\underline{u}{(i_1)}}(\I),\dots, \mb{X}^{\underline{u}{(i_k)}}(\I)\big)=R(\mb{x}_{(B: \underline{u})}, \I)$and 
$$R_n(\mb X(\I))=R_n(\underbrace{\mb X(\I),\dots,\mb X(\I)}_{n \text{ times}}).$$ 
 In particular, we have 
$$
 R_{\hat{1}_{n}}(\mb{x}_{\underline{u}}, \I; \mb{X}) =\I\langle \xi_{\underline{u}{(1)}}, T_{{\underline{u}{(2)}}} \dots T_{{\underline{u}{({n-1})}}}\xi_{\underline{u}{(n)}}  \rangle_{H}$$ 
i.e. the $n$-th joint cumulant of $\mb{X}$ at time $\I$. Note that the functional $R(\cdot, \I; \mb{X})$ can be linearly extended to all of $\mf{C} \langle \mb{x} \rangle$. We call this functional the \emph{cumulant functional} of the process $\mb{X}$ at time $\I$. 
%For $\I=1$ we call the corresponding functional the cumulant functional of the process $\mb{X}$. 
\end{definition}
An explicit formula for moments in terms of cumulants, involving the number of restricted crossings and  nestings of a partition follows from 
 Theorem   \ref{thm2} and we have  
\begin{align}\label{MomnetCumulantLevy}
M(\mb{x}_{\underline{u}}, \I; \mb{X}) 
&= \sum_{\pi \in  \PB(n)} q^{\rc(\pi|_{[n]})}t^{\nest(\pi|_{[n]})}\v^{\rc(\pi|_{[\bar n]})}\w^{\nest(\pi|_{[\bar n]})} R_{\pi|_{[ n]}}(\mb{x}_{\underline{u}},\I; \mb{X}).
\end{align}
This is because all diagonal cumulants of order at least two from Definition \ref{Defi:Cumulant} involved with $\bar H$ are equal to one.

\begin{remark} We emphasize that  the  general algebraic notation of independence introduced by  K\"{u}mmerer \cite{Kummerer} (pyramidally independent increments) is not true for quadrabasic L\'{e}vy process, therefore  we do not use this argument
in our proof. 
Note that in the special case  $(q,1,0,1)$-L\'{e}vy process  has this property which follows directly from equation  \eqref{MomnetCumulantLevy}.
\end{remark}

\subsection{Multiple stochastic measures }
Rota and Wallstrom \cite{Rota}  introduced  the notion of partition-dependent stochastic measures. Their approach  unifies a number of combinatorial results in probability theory, for example the It\^{o} multi-dimensional stochastic integrals through the usual product measures, by employing the M\"{o}bius inversion on the lattice of all partitions.
We shall show that, in our context, such an approach has also some potential. 

Fix $\I>0$. For $N$ and a subdivision of $[0,\I)$ into disjoint ordered half-open intervals $\mc{I} = \set{I_1, I_2, \ldots, I_N}$, let $\delta(\mc{I)} = \max_{1 \leq i \leq N} \abs{I_i}$. 
%Denote $X_i, a_i, a^\ast_i, p_i$ the appropriate objects for the interval $I_i$.
 Fix a monomial $\mb{x}_{\underline{u}} \in \mf{C} \langle x_1, x_2 , \ldots, x_k\rangle$ of degree $n$.

%\begin{definition}
%The \emph{stochastic measure} corresponding to the partition $\pi \in \P(n)$, monomial $\mb{x}_{\underline{u}}$ and the subdivision $\mc{I}$ is
%\begin{equation*}
%\St{\pi}(\mb{x}_{\underline{u}}, \I; \mb{X}, \mc{I}) = \sum_{\substack{\underline{v} \in [1,\dots, N]^n\\
% \ker  \underline{v}=\pi  }}  \mb{X}^{\underline{u}(1)}(I_{\underline{v}(1)})\dots \mb{X}^{\underline{u}(n)}(I_{\underline{v}(n)}),
%\end{equation*}
%where $n\in\N.$ The stochastic measure corresponding to the partition $\pi$ and the monomial $\mb{x}_{\underline{u}}$ is
%\[
%\St{\pi}(\mb{x}_{\underline{u}}, \I; \mb{X}) = \lim_{\delta(\mc{I}) \rightarrow 0} \St{\pi}(\mb{x}_{\underline{u}}, \I; \mb{X}, \mc{I})
%\]
%if the limit, along the net of subdivisions of the interval $[0, \I)$, exists. In particular, %denote by  
%\[
%\Delta_n(\mb{x}_{\underline{u}}, \I; \mb{X}) = \St{{\hat{1}}_n} (\mb{x}_{\underline{u}}, \I; \mb{X})
%\]
%the $n$-dimensional \emph{diagonal measure}. 
%\end{definition}
\begin{definition}
The  $n$-dimensional \emph{diagonal measure} corresponding to the monomial $\mb{x}_{\underline{u}}$ and the subdivision $\mc{I}$ is
\begin{align*}
\Delta_n(\mb{x}_{\underline{u}}, \I; \mb{X},\mc{I}) = \sum_{i=1}^N\prod_{j=1}^n  \mb{X}^{\underline{u}(j)}(I_{\underline{v}(i)}) =\sum_{i=1}^N \mb{X}^{\underline{u}(1)}(I_{\underline{v}(i)})\dots \mb{X}^{\underline{u}(n)}(I_{\underline{v}(i)})
%\\ 
 %\sum_{\substack{\underline{v} \in [1,\dots, N]^n\\
 %\ker  \underline{v}={\hat{1}}_n  }}  \mb{X}^{\underline{u}(1)}(I_{\underline{v}(1)})\dots \mb{X}^{\underline{u}(n)}(I_{\underline{v}(n)}),
\end{align*}
where $n\in\N.$ The $n$-dimensional \emph{diagonal measure} corresponding the monomial $\mb{x}_{\underline{u}}$ is
\[
\Delta_n(\mb{x}_{\underline{u}}, \I; \mb{X}) = \lim_{\delta(\mc{I}) \rightarrow 0} \Delta_n(\mb{x}_{\underline{u}}, \I; \mb{X},\mc{I})
\]
if the limit, along the net of subdivisions of the interval $[0, \I)$, exists.
\end{definition}
\begin{remark}
If an element of $\mb{X}$ does not depend on ${\underline{u}}$, i.e. this is a one-dimensional process, then we write  $\Delta_n( \I; \mb{X})$.% and $\psi_n(\I; \mb{X})$
\end{remark}

\begin{proposition}
\label{Prop:Cumulant}
For the monomial $\mb{x}_{\underline{u}}$ of degree $n$, the cumulant functional of the  quadrabasic  L\'{e}vy process $\mb{X}$ is given by
\[
R_{\hat{1}_{n}}(\mb{x}_{\underline{u}}, \I; \mb{X}) =\lim_{\delta(\mc{I}) \rightarrow 0}\state(\Delta_n(\mb{x}_{\underline{u}}, \I; \mb{X})).
\]
\end{proposition}
\begin{remark}
We emphasize that the existence of the limit $\lim_{\delta(\mc{I}) \rightarrow 0} \Delta_n(\mb{x}_{\underline{u}}, \I; \mb{X},\mc{I}) $   is not essential in Proposition \ref{Prop:Cumulant}. 
\end{remark}
\begin{proof}
By definition, we have to calculate 
%\begin{align*}
%\lim_{\delta(\mc{I}) \rightarrow 0}\sum_{\substack{\underline{v} \in [1,\dots, N]^n\\
 %\ker  \underline{v}={\hat{1}}_n   }}\state\big(\mb{X}^{\underline{u}(1)}(I_{\underline{v}(1)})\dots \mb{X}^{\underline{u}(n)}(I_{\underline{v}(n)})\big)  .
%\end{align*}
\begin{align*}
\lim_{\delta(\mc{I}) \rightarrow 0}\sum_{i=1}^N\state\big(\mb{X}^{\underline{u}(1)}(I_{\underline{v}(i)})\dots \mb{X}^{\underline{u}(n)}(I_{\underline{v}(i)})\big)  .
\end{align*}
Let us denote  $R_{\sigma|_{[ n]}}(\mb{x}_{\underline{u}}, \mc{I}; \mb{X}):=R_{\sigma|_{[ n]}}(\mb{x}_{\underline{u}}, 1; \mb{X}) \prod_{B\in\sigma|_{[n]}}\abs{I_{\underline{v}(B)}}$,
where we write $\underline{v}(B)$ for any $\underline{v}(i)$, $i\in B$.  
Using this  and Theorem  \ref{thm2}, with  $\pi \in \P(n)$   we see that
\begin{align*}
&\sum_{i=1
}^N\state\big(\mb{X}^{\underline{u}(1)}(I_{\underline{v}(i)})\dots \mb{X}^{\underline{u}(n)}(I_{\underline{v}(i)})\big)
\\&=\sum_{i=1 }^N\sum_{\substack{\sigma\in \PB(n)\\  \sigma|_{[ n]}\leq {\hat{1}}_n}} q^{\rc(\sigma|_{[ n]})}t^{\nest(\sigma|_{[ n]})}\v^{\rc(\sigma|_{[\bar n]})}\w^{\nest(\sigma|_{[\bar n]})}R_{\sigma|_{[ n]}}(\mb{x}_{\underline{u}}, \mc{I}; \mb{X})
\\
&=\sum_{i=1 }^N\sum_{\substack{\sigma\in \PB(n)\\  \sigma|_{[ n]}= {\hat{1}}_n }} q^{\rc(\sigma|_{[ n]})}t^{\nest(\sigma|_{[ n]} )}\v^{\rc(\sigma|_{[\bar n]} )}\w^{\nest(\sigma|_{[\bar n]} )}R_{\sigma|_{[ n]}}(\mb{x}_{\underline{u}}, \mc{I}; \mb{X})\\
&+\sum_{i=1 }^N\sum_{\substack{\sigma\in \PB(n)\\  \sigma|_{[ n]}< {\hat{1}}_n }}q^{\rc(\sigma|_{[ n]})}t^{\nest(\sigma|_{[ n]})}\v^{\rc(\sigma|_{[\bar n]})}\w^{\nest(\sigma|_{[\bar n]})}R_{\sigma|_{[ n]}}(\mb{x}_{\underline{u}}, \mc{I}; \mb{X})
%\\
%&=\sum_{\substack{\sigma\in \PB(n)\\  \sigma|_{[\bar n]}= \pi}} 
%q^{\rc(\sigma)}t^{\nest(\sigma)}\v^{\rc(\sigma)}\w^{\nest(\sigma)}R_{\sigma|_{[ n]}}%(\mb{x}_{\underline{u}},\mc{I}; \mb{X})
%\\&
%+\sum_{\substack{\underline{v} \in [1,\dots, N]^n\\
% \ker  \underline{v}={\hat{1}}_n  }}\sum_{\substack{\sigma\in \PB(n)\\  \sigma|_{[\bar n]}< \pi}} 
% q^{\rc(\sigma|_{[n]})}t^{\nest(\sigma|_{[n]})}\v^{\rc(\sigma|_{[\bar n]})}\w^{\nest(\sigma|_{[\bar n]})}R_{\sigma|_{[ n]}}%(\mb{x}_{\underline{u}}, \mc{I}; \mb{X})
\intertext{By    Remark \ref{rem:partycje}  (4)  if $\sigma\in \PB(n)$ and $  \sigma|_{[n]}=\hat{1}_{n}  $ then $\sigma=\hat{1}_{n}\O \hat{ 1}_{\bar n} $. We now expand further and obtain }
&=
s{R_{\hat{1}_{n}}(\mb{x}_{\underline{u}},1; \mb{X}) }
+\underbrace{ \sum_{i=1
}^N\sum_{\substack{\sigma\in \PB(n)\\  \sigma|_{[ n]}< {\hat{1}}_n }}  \prod_{B\in\sigma}\abs{I_{\underline{v}(B)}}q^{\rc(\sigma|_{[ n]})}t^{\nest(\sigma|_{[ n]})}\v^{\rc(\sigma|_{[\bar n]})}\w^{\nest(\sigma|_{[\bar n]})}R_{\sigma|_{[ n]}}(\mb{x}_{\underline{u}},1; \mb{X})}_{(\star)}.
%\\
\end{align*}
%Let us observe that 
%$$s^{\# \pi}-n^{\# \pi-1}\delta(\mc{I})^{\# \pi} \leq \sum_{\substack{\underline{v} \in [1,\dots, N]^n\\
% \ker  \underline{v}=\pi  }}\prod_{B\in\pi}\abs{I_{\underline{v}(B)}}\leq \sum_{\substack{\underline{v} \in [1,\dots, N]^n }}\prod_{B\in\pi}\abs{I_{\underline{v}(B)}}= s^{\# \pi}$$
%and thus we have  $\lim_{\delta(\mc{I}) \rightarrow 0} \sum_{\substack{\underline{v} \in [1,\dots, N]^n\\
% \ker  \underline{v}=\pi  }}\prod_{B\in\pi}\abs{I_{\underline{v}(B)}}=s^{\# \pi}$
%and so the term $(I)$ converges to  
%$$\sum_{\substack{\sigma\in \PB(n)\\  \sigma|_{[n]}= \pi}} 
%q^{\rc(\sigma|_{[n]})}t^{\nest(\sigma|_{[n]})}\v^{\rc(\sigma|_{[\bar n]})}\w^{\nest(\sigma|_{[\bar n]})}R_{\sigma|_{[ n]}}(\mb{x}_{\underline{u}},\I; \mb{X})$$
Now we show that the limit of each of the remaining terms  $(\star)$ is 0. Indeed, if   $\sigma|_{[ n]}< \hat{1}_{n}$, then $\#\sigma|_{[ n]}>1$
and assume that the number of them is $d$, where $d\geq 2$.  
We may
assume $\delta(\mc{I})<1$, and thus for  each fixed $\pi$  the term $(\star)$  is bounded by 
\[
C \sum_{i=1}^N \abs{I_i}^d \leq C \delta(\mc{I}) \I^{d -1},
\]
where  $C$ is a constant independent of the subdivision $\mc{I}$. Therefore such a term converges to $0$ as $\delta(\mc{I}) \rightarrow 0$.

\end{proof}

\subsubsection{The higher diagonal measures}
%In general we do not know how to calculate the partition-dependent stochastic measures $\St{\pi}(\mb{X})$; indeed we do not expect a nice answer for a general process.  
%However, one element of it is known, namely, we can 
Now we calculate all the higher diagonal measures (this object appears in the functional It\^{o} formula for L\'{e}vy processes). 

\begin{proposition}
\label{Prop:Gaussian}
For a one-dimensional self-adjoint process $\mb X(\I)= p_\I(\xi\O \eta , T, \lambda)$ the  $n$-dimensional {diagonal measure} with $n \geq 2$ exists in the $L^2$-norm with respect
to $\state(\cdot)$, and equals
\[
\Delta_n(\I; \mb{X})  = p_\I(T^{n-1} \xi\O \eta, T^n, \ip{\xi}{T^{n-2} \xi}).
\]
\end{proposition}
\begin{proof}
Let $\mb{Y}_n(I)= p_I(T^{n-1} \xi\O \eta, T^n, \ip{\xi}{T^{n-2} \xi})$ be the  process from the right-hand side of above theorem.  
We will show that 
$$\lim_{\delta(\mc{I}) \rightarrow 0} \norm{ \Delta_n( \I; \mb{X},\mc{I})-\mb{Y}_n(\I)}_2=0.$$
where $\Delta_n( \I; \mb{X},\mc{I})=\sum_{\substack{1\leq i\leq N}}
\mb{X}^n(I_{\underline{v}(i)}).$
First expand 
\begin{align*}
[\Delta_n(\I; \mb{X},\mc{I})-\mb{Y}_n(\I)]^2&=\Delta_n^2( \I; \mb{X},\mc{I})
-\Delta_n( \I; \mb{X},\mc{I})\mb{Y}_n(\I)-\mb{Y}_n(\I)\Delta_n( \I; \mb{X},\mc{I})+\mb{Y}_n^2(\I). 
\end{align*}
We will show that in the limit (as $\delta(\mc{I}) \rightarrow 0$) the first two factors of above expansion disappear.
We start with the first factor 
\begin{align*}
%&\lim_{\delta(\mc{I})\rightarrow 0}
& \state\big(\Delta_n^2( \I; \mb{X},\mc{I})\big)=\state\big(\sum_{\substack{1\leq i\leq N\\
\\1\leq j \leq N}}
\mb{X}^n(I_{\underline{v}(i)})
\mb{X}^n(I_{\underline{v}(j)}))\big)
\\&=\state\big(\sum_{i=1}^N
\mb{X}^{2n}(I_{\underline{v}(i)})\big)
 +\state\big(\sum_{i\neq j}^N
\mb{X}^n(I_{\underline{v}(i)})
\mb{X}^n(I_{\underline{v}(j)})\big)
%\\&=\state\big(\sum_{i=1}^N
%\mb{X}^{2n}(I_{\underline{v}(i)})\big)
 %+\sum_{i\neq j}^N\state\big(
%\mb{X}^n(I_{\underline{v}(i)})\big)\state\big(
%\mb{X}^n(I_{\underline{v}(j)})\big)
\xrightarrow{\delta(\mc{I}) \rightarrow 0} 
R_{2n}(\mb{X}(\I))+R_{n}^2(\mb{X}(\I)).
\intertext{Indeed, by Proposition \ref{Prop:Cumulant} it follows that the first factor converges to $R_{2n}(\mb{X}(\I))$. Now using  cumulant expansions of the second factor we get }
& \sum_{i\neq j}^N\state\big(
\mb{X}^n(I_{\underline{v}(i)})\big)\state\big(
\mb{X}^n(I_{\underline{v}(j)})\big)=\sum_{i\neq j}^N {R_{n}(\mb{X}(I_{\underline{v}(i)}) )}{R_{n}(\mb{X}(I_{\underline{v}(j)})) }
\\& + \sum_{i\neq j}^N\sum_{\substack{\sigma\in \PB(2n)\\  \sigma|_{[ 2n]}< {\hat{1}}_{2n} }}  q^{\rc(\sigma|_{[2 n]})}t^{\nest(\sigma|_{[2n]})}\v^{\rc(\sigma|_{[\overline{2 n}]})} \v^{\rc(\sigma|_{[\overline{2n}]})} R_{\sigma|_{[ 2n]}}(\mb{X}(I_{\underline{v}(i)}) R_{\sigma|_{[ 2n]}}(\mb{X}(I_{\underline{v}(j)})  
%\intertext{where $Q(q,t,\v,\w,\sigma)=q^{\rc(\sigma)}t^{\nest(\sigma)}\v^{\rc(\sigma)} \w^{\nest(\sigma)}$ }
={R_{n}^2(\mb{X}(s) )}\\& \underbrace{-\sum_{i=1}^N {R_{n}^2(\mb{X}(I_{\underline{v}(i)}) )}   +\sum_{i\neq j}^N\sum_{\substack{\sigma\in \PB(2n)\\  \sigma|_{[ 2n]}< {\hat{1}}_{2n} }}  q^{\rc(\sigma|_{[2 n]})}t^{\nest(\sigma|_{[2n]})}\v^{\rc(\sigma|_{[\overline{2 n}]})} \v^{\rc(\sigma|_{[\overline{2n}]})} R_{\sigma|_{[ 2n]}}(\mb{X}(I_{\underline{v}(i)}) R_{\sigma|_{[ 2n]}}(\mb{X}(I_{\underline{v}(j)}) }_{(\star)}
\end{align*}
by the same argument as in Proposition \ref{Prop:Cumulant} we have that  each term in the expression   $(\star)$   converges to zero. 
Now we focus on the second factor, i.e. %If $Block_\primes(\sigma|_{[n]})=\{B_{E_1},\dots,B_{E_k}\}$, then this is equal to
\begin{align*}
&\state\big(\Delta_n( \I; \mb{X},\mc{I})\mb{Y}_n(\I)\big)=\state\big(\sum_{\substack{1\leq i \leq N\\
\\1\leq j\leq N
}}
\mb{X}^n(I_{\underline{v}(i)}) \mb{Y}_{n}(I_{\underline{v}(j)})\big).
\intertext{By similar argument as we used in the first part, we conclude that }
&\xrightarrow{\delta(\mc{I}) \rightarrow 0}R\big(\underbrace{\mb{X}(s),\dots ,\mb{X}(s)}_{n \text{ times}},\mb{Y}_{n}(I) \big)+R_n\big(\mb{X}(s))R_1(\mb{Y}_{n}(s) \big)
% \sum_{\substack{\gamma\in \PB(n+1)\\  \gamma|_{\{ 1,\dots, n\}}= \{ 1,\dots, n\} \\  \gamma|_{\{{n+1}\}}=\{{n+1}\}
% }
 %} q^{\rc(\gamma|_{[n+k]})}t^{\nest(\gamma|_{[n+k]})}\v^{\rc(\gamma|_{[\overline {n+k}]})}\w^{\nest(\gamma|_{[\overline{n+k}]})}R_{\gamma|_{[ n+k]}}( \I; \mb{X},\mb{Y}_{n}).
\end{align*}
By Definition \ref{def:cumulantyLevego}
we  observe that mixed  cumulants of   $\mb{X}(I) $ %$(=\mb{Y}_{1}(I))$ 
 and $\mb{Y}_{n}(I)$ are coincident 
in a sense that for $n\geq 2$ and $k\geq 0$ we have 
$$R\big(\underbrace{\mb{X}(I),\dots ,\mb{X}(I)}_{k \text{ times}},\mb{Y}_{n}(I) \big)=R_{k+n}\big(\mb{X}(I)\big).%,\text{ where }n =\sum_{i=1}^k n_i.
$$ 
So, the last limit expression reduces to $R_{2n}(\mb{X}(\I))+R_{n}^2(\mb{X}(\I)).$
Hence, the first two factors disappear and similarly we conclude for the two remaining elements. 
\end{proof}

\subsection{Generators}
\label{Sec:Generators}
The analysis in this subsection is partially motivated by  papers \cite{Ans01} and \cite{SchurCondPos}.

\begin{definition}\label{def:DefGeneratory}
\begin{enumerate}[I.]
%\item each block is a singleton or a pair;
\item A functional $\psi$ on $\mf{C} \langle \mb{x} \rangle$ is \emph{conditionally positive} if its restriction to the subspace of polynomials with zero constant term is positive semi-definite.

\item  
We say that a functional $\psi$ on $\mf{C} \langle \mb{x} \rangle$ is \emph{a generators} of $(\q)$-L\'{e}vy process if it is a derivative of the moment functional at zero.% on the constant $1$. 

\item 
 We say that the functional $\psi$ is \emph{analytic} if for any $i$ and any multi-index $\underline{u}$, 
\[
\limsup_{n \rightarrow \infty} \frac{1}{ \sqrt[n]{n!}} \psi[(\mb{x}_{\underline{u}})^\ast x_i^{2n} \mb{x}_{\underline{u}}]^{1/2n} < \infty.
\] 
\end{enumerate}
\end{definition}
\begin{remark}\label{rem:generatory}
The family of the moment functionals of a  $(\q)$-L\'{e}vy process is determined by its cumulant functional. %The 
Indeed, by equation \eqref{MomnetCumulantLevy},  we have 
\begin{align*}
M(\mb{x}_{\underline{u}}, \I; \mb{X}) 
&= \sum_{\pi \in  \PB(n)} q^{\rc(\pi|_{[n]})}t^{\nest(\pi|_{[n]})}\v^{\rc(\pi|_{[\bar n]})}\w^{\nest(\pi|_{[\bar n]})} \I^{\#\MPair(\pi)}   R_{\pi|_{[ n]}}(\mb{x}_{\underline{u}},1; \mb{X}),
\end{align*}
which implies that this is a polynomial in $\I$ for   $\pi=\hat{1}_{n}\O \hat{1}_{\bar n}$, and so by differentiating this equality, we obtain
$$
\frac{d}{d \I}  M(\mb{x}_{\underline{u}}, \I; \mb{X})\Bigr|_{\I=0} = R_{\hat{1}_{n}} (\mb{x}_{\underline{u}}, 1;\mb{X}) 
=\langle \xi_{\underline{u}{(1)}}, T_{{\underline{u}{(2)}}} \dots T_{{\underline{u}{({n-1})}}}\xi_{\underline{u}{(n)}}  \rangle_{H}.
%= R(\mb{x}_{\underline{u}}; \mb{X}). \eqno{\qed}
$$
%From the fact that each of the moment functionals is positive and equals $1$ on the constant $1$ it follows by differentiating that 
\end{remark}
The following proposition is an analog of the Schoenberg correspondence for
our context. 
The proof is almost identical to one of \cite{Ans01,SchurCondPos} and will be omitted.
We reproduce the proof on arXiv version for the reader's convenient.
\begin{proposition} \label{twr:twanaliticconditionallypositive}
A functional $\psi$ is analytic and conditionally positive if and only if it is the generator of the family of the moment functionals for some  $(\q)$-L\'{e}vy process.
\end{proposition}

\begin{proof}
Suppose $\psi$ is the generator of the family of moment functionals $M(\cdot, \I; \mb{X})$ for a  $(\q)$-L\'{e}vy process $\mb{X}(\I)$. 
By Definition \ref{def:DefGeneratory} we have $\psi(\mb{x}_{\underline{u}}) = R_{\hat{1}_{m}} (\mb{x}_{\underline{u}}, 1; \mb{X})$,
which is means that the cumulant functional is conditionally positive indeed: 
$$
 R_{\hat{1}_{n}} ((\mb{x}_{\underline{u}})^\ast\mb{x}_{\underline{u}}, 1;\mb{X}) =
%=\langle \xi_{\underline{u}{(n)}}, T_{{\underline{u}{(n-1)}}} \dots T_{{\underline{u}{({1})}}}T_{{\underline{u}{({1})}}}\dots T_{{\underline{u}{(n-1)}}}\xi_{\underline{u}{(n)}}  \rangle_{H}
\norm{T_{{\underline{u}{({1})}}}\dots T_{{\underline{u}{(n-1)}}}\xi_{\underline{u}{(n)}}}_H^2
\geq 0. 
%= R(\mb{x}_{\underline{u}}; \mb{X}). \eqno{\qed}
$$
For $\mb{x}_{\underline{u}}$ of degree $m$ we have 
\begin{align*}
\limsup_{n \rightarrow \infty} \frac{1}{ \sqrt[n]{n!}} \psi[(\mb{x}_{\underline{u}})^\ast x_i^{2n} \mb{x}_{\underline{u}}]^{1/2n}
%&= \limsup_{n \rightarrow \infty} \frac{1}{n} R((\mb{x}_{\underline{u}})^\ast x_i^{2n} \mb{x}_{\underline{u}}, t; \mb{X})^{1/2n} \\
&= \limsup_{n \rightarrow \infty} \frac{1}{ \sqrt[n]{n!}}  \langle{\xi_{\underline{u}(m)}}{\prod_{j=m-1}^{1} T_{\underline{u}(j)} T_i^{2n} \prod_{j=1}^{m-1} T_{\underline{u}(j)} \xi_{\underline{u}(m)}}\rangle^{1/2n} \\
&= \limsup_{n \rightarrow \infty} \frac{1}{ \sqrt[n]{n!}}  \|{T_i^{n} \prod_{j=1}^{m-1} T_{\underline{u}(j)} \xi_{\underline{u}(m)}}\|^{1/n} < \infty
\end{align*}
by Assumption \ref{assumption:analitic}.

Now suppose $\psi$ is conditionally positive and analytic.  %We have to show that there exist operator $T_i$

 The first step in the proof is that of Anshelevich \cite[Proposition 4.3]{Ans01} to show the existence of $T_i$ and space $K$ (the proof is practically identical to that of  this result and we provide an outline of the details for the reader's convenience).
 Since $\psi$ is positive then we can define semi-definite inner product on the space $\mf{C} \langle \mb{x} \rangle$ by 
$$\ip{f}{g}_{\psi} = \psi[(f - \delta_0(f))^\ast(g - \delta_0(g))].$$ 
Let $\mc{N}_\psi=\{a\in \mf{C}\mid\ip{a}{a}_{\psi}=0\}$ and  
let $K$ be the Hilbert space obtained by completing the quotient $\mf{C} \langle \mb{x} \rangle / \mc{N}_\psi $ with respect to this inner product. Denote by $\rho$ the canonical mapping $\mf{C} \langle \mb{x} \rangle \rightarrow K$, let ${D}$ be its image, and for $f, g \in \mf{C} \langle \mb{x} \rangle$ define the operator $\Gamma(a): {D} \rightarrow {D}$ by 
\[
\Gamma(f) \rho(g) = \rho(f g) - \rho(f) \delta_0(g).
\]
The operator $\Gamma$ is well defined since, by the Cauchy-Schwartz inequality, 
\[
\norm{\Gamma(f) \rho(g)}_\psi 
= \psi[(g - \delta_0(g))^\ast f^\ast f (g - \delta_0(g))]
\leq \norm{\rho(g)}_\psi \norm{f^\ast f (g - \delta_0(g))}_\psi.
\]
Clearly ${D}$ is dense in $K$, invariant under $\Gamma(a)$, and $\Gamma(a)$ is symmetric on it if $a$ is symmetric. 
%Put, for $i \in [1 \ldots k]$,  
We definite 
$\lambda_i = \psi[x_i], \xi_i = \rho(x_i), T_i= \Gamma(x_i)$. Each $T_i$ takes ${D}$ to itself. By construction, $\Gamma(x_i) \rho(\mb{x}_{\underline{u}}) = \rho(x_i \mb{x}_{\underline{u}})$, and so 
\begin{align*}
\limsup_{n \rightarrow \infty} \frac{1}{ \sqrt[n]{n!}} \norm{T_i^n \rho(\mb{x}_{\underline{u}})}_\psi^{1/n}
&= \limsup_{n \rightarrow \infty} \frac{1}{ \sqrt[n]{n!}} \norm{x_i^n \mb{x}_{\underline{u}}}_\psi^{1/n} \\
&= \limsup_{n \rightarrow \infty} \frac{1}{ \sqrt[n]{n!}} \psi[(\mb{x}_{\underline{u}})^\ast x_i^{2n} \mb{x}_{\underline{u}}]^{1/2n} < \infty
\end{align*}
since the functional $\psi$ is analytic. Therefore each of those vectors is analytic for $T_i$, and the linear span of these vectors is ${D}$. In particular, $T_i$ is essentially self-adjoint on ${D}$.
Let $\HH=H\O\bar H$ where $H=L^2(\mf{R}_+, d x) \otimes K$, $\bar H$ is a one-dimensional Hilbert space spanned by the $\eta$ with norm one   and $\mc{D}=D\O \bar H$.
Finally, we can define the  $(\q)$-L\'{e}vy process  by $\mb{X}^{{\underline u (i)}}(\I) = p_\I(\xi_{\underline u (i)}\O \eta, T_{\underline u (i)}, \lambda_{\underline u (i)})$ and by Remark \ref{rem:generatory} we obtain 
$
R(\mb{x}_{\underline{u}},1; \mb{X}) = \psi[\mb{x}_{\underline{u}}].
$
\end{proof}

\subsection{Convolution }
First, we introduce a product state which reduces to a usual (tensor) product state, for $q=t=\w=1$ and $\v=0$
 while for $q=\v=0$ and $t=\w=1$ it is the (reduced) free product state.

\begin{definition}
\begin{enumerate}[I.]
\item For functional  $\Phi$  on $\mf{C} \langle \mb{x} \rangle$  we  define the functional $\Psi =\Psi(\Phi)$ on $\mf{C} \langle \mb{x} \rangle$  by induction 
\[
\Psi(\mb{x}_{\underline{u}})
= \Phi(\mb{x}_{\underline{u}}) - \sum_{\substack{\pi \in \PB(n) \\ \pi \neq \hat{1}_{n}\O\hat{1}_{\bar n}}} q^{\rc(\pi|_{[n]})}t^{\nest(\pi|_{[n]})}\v^{\rc(\pi|_{[\bar n]})}\w^{\nest(\pi|_{[\bar n]})} \prod_{B \in \pi|_{[n]}} \Psi(\mb{x}_{(B: \underline{u})})
\]
 and extend linearly.  
Let $\Phi_1$, $\Phi_2$ be functionals on $\mf{C} \langle x_1, \ldots, x_{k} \rangle$ and  $\mf{C} \langle y_1,  \ldots, y_{l} \rangle$, respectively. On $\mf{C} \langle \mb{xy} \rangle:=\mf{C} \langle x_1, \ldots, x_{k},y_1,\ldots, y_{l} \rangle$ we define their {product functional} by rule that mixed cumulants of independent quantities equal zero
  \begin{align*}
%  \label{eq:ED}
\begin{aligned}
  \Phi_1 \times_{\q}\Phi_2 : \mf{C} \langle \mb{xy} \rangle &\to \C
  \\
  \Psi(\Phi_1 \times_{\q} \Phi_2)(\mb{xy}_{\underline{u}}) &\mapsto 
\begin{cases}
\Psi(\Phi_1)(\mb{xy}_{\underline{u}}) & \text{ if } \mb{xy}_{\underline{u}}\in \mf{C} \langle x_1,  \ldots, x_{k} \rangle \\
\Psi(\Phi_2)(\mb{xy}_{\underline{u}}) & \text{ if } \mb{xy}_{\underline{u}}\in \mf{C} \langle y_1, \ldots, y_{l} \rangle \\
0 & \text{ otherwise}.
\end{cases}
\end{aligned}
  \end{align*}
\item We define  $\ID$, i.e. the set of all  infinitely divisible  functionals on $\mf{C}\langle \mb{x} \rangle$ by 
\begin{align*}
\ID( k) :&= \set{\Phi: \Phi(\cdot) = M(\cdot, 1; \mb{X})}% \text{ for some $k$-dimensional $(\q)$-L\'{e}vy process } \mb{X}}
 \\&= \set{\Phi: \Psi(\Phi) \text{ is conditionally positive and analytic}}.
\end{align*}
\end{enumerate}
\end{definition}
From the definition above it is clear that  $R(\cdot, \I; \mb{X}) = \Psi (M(\cdot, \I; \mb{X}))$.

\begin{proposition}
\label{Lem:Product}
For $\Phi_1 \in \ID(k)$ and $ \Phi_2 \in \ID(l)$, their product functional is a state.
\end{proposition}

\begin{proof}  The proof follows by direct construction.
From Proposition \ref{twr:twanaliticconditionallypositive} we know that there exist processes $\mb X^{(i, 1)}(\I) $ and $\mb Y^{(i, 2)}(\I)$ which may be identified with the $(\q)$-L\'{e}vy processes whose distributions at time $1$ are $\Phi_1\in\ID(k) $ on $\mf{C} \langle x_1, \ldots, x_{k} \rangle$   and  $ \Phi_2\in \ID(l)$ on $\mf{C} \langle y_1,\ldots, y_{l} \rangle$,  respectively. We will explain that we can choose these processes in such a way that the product functional conditions are met. 
Let   $\xi_{i,1}\O\eta \in V_1 \O\bar H $ and  $T_{i,1}\O \id$ is an operator on $V_1\O\bar H $ with domain $\mc{D}_1=D_1\O\bar H$. Similarly $\xi_{i,2}\O\eta \in V_2\O\bar H $ and  $T_{i,2}\O \id$ is an operator on $V_2\O\bar H$ with domain $\mc{D}_2=D_2\O\bar H$. We   identify 
\begin{align*}
&\text{$\xi_{i,1}\O\eta$  with $(\xi_{i,1} \oplus 0) \O\eta $,}
\\ & \text{$\xi_{i,2}\O\eta $   with $(0 \oplus \xi_{i,2})\O\eta$,}
\\& \text{$T_{i,1}\O \id$  with $\bigl( \begin{smallmatrix} T_{i,1} & 0 \\ 0 & 0 \end{smallmatrix} \bigr)\O\id $,}
\\& \text{$T_{i,2}\O \id$ with $\bigl( \begin{smallmatrix} 0 & 0 \\ 0 & T_{i,2} \end{smallmatrix} \bigr)\O\id $.}
\end{align*}
%$\xi_{i,1}\O\eta $ with $(\xi_{i,1} \oplus 0) \O\eta $, $\xi_{i,2}\O\eta $ with $(0 \oplus \xi_{i,2})\O\eta$, $T_{i,1}$ with $\bigl( \begin{smallmatrix} T_{i,1} & 0 \\ 0 & 0 \end{smallmatrix} \bigr)$ and $T_{i,2}$ with $\bigl( \begin{smallmatrix} 0 & 0 \\ 0 & T_{i,2} \end{smallmatrix} \bigr) $ and we put  
 Let $V = (V_1 \oplus V_2)\O\bar H$ and $\mb{X} ^{(i, 1)}(\I) = p_\I(\xi_{i, 1}\O\eta, T_{i,1}, \lambda_{i,1})$, $\mb Y^{(i, 2 )}(\I)= p_\I(\xi_{i,2}\O\eta, T_{i,2}, \lambda_{i,2})$. 
 By Definition \ref{def:cumulantyLevego} we know
that  this identification does not change the mixed  cumulants of  $\mb X^{(i, 1)}(\I) $ and $\mb Y^{(i, 2)}(\I)$. 
From this identification it follows that
$$\state(A_1 \ldots A_{k+l})=\Phi_1 \times_{\q} \Phi_2 (a_1 \ldots a_{k+l}),$$
for all $A_i\in \{\mb X^{(1,1)}, \ldots, \mb X^{(k,1)}, \mb Y^{(1,2)}, \ldots, \mb Y^{(l,2)}\}$ and $$a_i=\begin{cases}
x_i & \text{ if } A_i\in \{\mb X^{(1,1)}, \ldots, \mb X^{(k,1)}\}\\
y_i & \text{ if } A_i\in\{\mb Y^{(1,2)}, \ldots, \mb Y^{(l,2)}\} 
\end{cases}.$$ Thus we prove that 
 $$\Phi_1 \times_{\q} \Phi_2 \in \ID( k + l).$$
\end{proof}

\subsubsection{$(\q)$-convolution and the L\'{e}vy-Hinchin representation.}
Now we consider the $(\q)$-L\'{e}vy processes in the simplest case of one-dimensional $K$.
We use the following notations in this subsection:
\begin{enumerate}
\item $\mc{M}$ denotes the space of finite positive Borel measures on $\mf{R}$;
\item $\mc{M}_P \subset \mc{M}$ denotes the subset of probability measures;
\item For $\mu \in \mc{M}_P$ and $n \geq 1$, we define a cumulant $r_n(\mu)$ by
\begin{align}
\label{MC}
r_n(\mu) = m_n(\mu) -\sum_{\substack{\pi \in \PB(n) \\ \pi \neq \hat{1}_{n}\O\hat{1}_{\bar n}}} q^{\rc(\pi|_{[n]})}t^{\nest(\pi|_{[n]})}\v^{\rc(\pi|_{[\bar n]})}\w^{\nest(\pi|_{[\bar n]})} \prod_{B \in \pi|_{[n]}}  r_{\abs{B}}.
\end{align}
\item Let  $\tau \in \mc{M}_u \subset \mc{M}$ such that   $$\limsup_{n \rightarrow \infty} \frac{1}{ \sqrt[n]{n!}} m_{2n}^{1/2n}(\tau) < \infty.$$  This condition means  that a measure in  $\mc{M}_u $ is  uniquely determined by its moments.
Indeed 
Carleman's theorem for moments (see \cite{Akhiezer}) states that the moment problem is determined if 
the following condition holds: $$\sum_{n\geq 0} m_{2n}^{-\frac{1}{2n}}(\tau)=\infty. $$
In our case $m_{2n}^{-1/2n}(\tau)\geq  \frac{C}{ \sqrt[n]{n!}}$ for some $C$, so the conclusion follows.
Equivalently, under our assumption $\tau \in \mc{M}_u \iff $  
there exists the operator $T$ which has the distribution $\tau$, with
respect to the vector functional $\ip{\xi}{\cdot \xi}_H$ and moreover, $\xi$ is an analytic vector for $T$. 
\end{enumerate}

\begin{definition}
\begin{enumerate}[I.]
\item For  $\tau \in \mc{M}_u$ such that the operator $T$ has the distribution $\tau$ with respect to the vector functional $\ip{\xi}{\cdot \xi}_H$ and $\lambda \in \mf{R}$ we define an injective map 
\begin{align*}
\IF:  \mf{R} \times \mc{M}_u &\rightarrow \mc{M}_P;
\\
(\lambda, \tau)&\mapsto \mu \quad \text{such that $\mu$ is the distribution  of  $p_1(\xi\O\eta, T, \lambda)$};
%\\ &\text {  such that the operator $T$ has distribution $\tau$ with respect to the vector functional $\ip{\xi}{\cdot \xi}$.}
\end{align*}
Let $\IF_{\q}(\ast):=\{\IF(\lambda, \tau)\mid (\lambda, \tau)\in  \mf{R} \times \mc{M}_u\}$. We define the analog of the L\'{e}vy-Hinchin representation $\Rtr_{\q}: \IF_{\q}(\ast) \rightarrow \mf{R} \times \mc{M}_u$ to be the inverse of $\IF(\lambda, \tau)$. 
\item 
For $\mu, \nu \in \IF_{\q}(\ast)$, define their $(\q)$-convolution $\mu \ast_{\q} \nu$ by the rule that $$\Rtr_{\q}(\mu \ast_{\q} \nu) = \Rtr_{\q}(\mu) + \Rtr_{\q}(\nu).$$ 
\end{enumerate}
\end{definition}
\begin{corollary} 
The $(\q)$-convolution of two positive measures is positive, i.e. $$\IF(\lambda_1, \tau_1) \ast_{\q} \IF(\lambda_2, \tau_2) = \IF(\lambda_1 + \lambda_2, \tau_1 + \tau_2).$$
\end{corollary}
Finally, we present the relation between the convolution of measures and  product states. 
\begin{proposition}
For $\mu_1, \mu_2 \in \IF_{\q}(\ast)$, and $\Phi_1, \Phi_2 \in  \ID(1)$ we have 
\[
(\mu_1 \ast_{\q} \mu_2)(x^n) = (\Phi_1 \times_{\q} \Phi_2)((x_1 + x_2)^n).
\]
\end{proposition}

\begin{proof}
We use the
representation described in the proof of Proposition \ref{Lem:Product} and  obtain
\[
M(x^n, 1;\mb Z) = M((x_1 + y_1)^n, 1; (\mb X^{(1)}, \mb Y^{(2)})). 
\]
where $\mb Z=\mb X^{(1)}+\mb Y^{(2)}$, because $r_n(\mu_{\mb Z})= r_n(\mu_{\mb X^{(1)}}) + r_n(\mu_{\mb Y^{(2)}})$.
\end{proof}

\begin{corollary}
\begin{enumerate} [(1).]
\item Let $V = \mf{C}$, $\xi= 1 \in V, T = 0$ and $ \lambda = 0$. Then the \emph{$(\q)$-Brownian motion} is the process $X(\I) = p_\I(\xi\O \eta , 0, 0) $. The distribution $\mu$ of $X(\I)$ is the $(\q)$-Gaussian distribution with parameter $\I$, given by $\Rtr_{\q}(\mu)=(0, \I \delta_0)$.
\item  
Let $V = \mf{C}$, $\xi= 1 \in V, T = \id $ and $ \lambda = 1$. The \emph{$(\q)$-Poisson process} is the process $X(\I) = p_\I(\xi\O\eta, \id, \I)$. The distribution $\mu$ of $X(\I)$ is the $(\q)$-Poisson distribution with parameter $\I$, given by $\Rtr_{\q}(\mu)=(\I, \I \delta_1)$.

\end{enumerate}

\end{corollary}
\section{Concluding Remark}
Finally, we summarize our conclusions and
contributions and give some perspectives for future research directions.
\begin{enumerate}[(1).]
\item The construction presented in this article can be extended for $\S_n ^k$ with  multipolar Hermite orthogonal polynomial  of the type 
\begin{align*}
&x P_n(x) = P_{n+1}(x) +\underbrace{[n]_{\cdot,\cdot} \dots [n]_{\cdot,\cdot}}_{k \text{ times}} P_{n-1}(x). 
\end{align*}
In these cases combinatorics and partitions are of the same type as those described in Section \ref{sec:kombinatorykapartycje}, except in the limit case because then the measure is not necessarily uniquely determined, for example when $P_n(x) = P_{n+1}(x) +n^3 P_{n-1}(x)$.
\item Let $P_n$ be a family of orthogonal polynomials. One standard combinatorial task is  to calculate the linearization coefficients, when we  are interested in the expectations $\state(P_{n_1}P_{n_2} \dots P_{n_k})$.
The name stems from the fact that these are the coefficients in the expansion of
products of this type in the basis $P_n$, that is, expansions as the sums of orthogonal
polynomials. Many of these coefficients are positive integers, and so they \emph{count something}; see \cite{Anshelevich2005,Anshelevich2018}.
 We expected  that for $t=\w=1$, we can obtain a nice result for the polynomials \eqref{recursion} by using diagonal pair partition because just one crossing plays a role. 
 \item It is worth to find the central limit theorem for the quadrabasic  Gaussian operator as in \cite{BlitvicEjsmont}. Our initial investigation shows that this problem is nontrivial. 
\end{enumerate}
\bibliography{sumcomm}
\bibliographystyle{amsplain}

\end{document}